\pdfoutput=1
\RequirePackage{ifpdf}
\ifpdf % We are running pdfTeX in pdf mode
\documentclass[pdftex]{sigma}
\else
\documentclass{sigma}
\fi

\newcommand{\C}{\mathbb{C}}

\newcommand{\Z}{\mathbb{Z}}
\newcommand{\K}{\mathbb{K}}
\newcommand{\bB}{\mathbf{B}}
\newcommand{\bD}{\mathbf{D}}
\newcommand{\bG}{\mathbf{G}}
\newcommand{\bH}{\mathbf{H}}
\newcommand{\bK}{\mathbf{K}}
\newcommand{\bO}{\mathbf{O}}
\newcommand{\bT}{\mathbf{T}}
\newcommand{\osp}{\mathfrak{osp}}
\newcommand{\fsl}{\mathfrak{sl}}
\newcommand{\fgl}{\mathfrak{gl}}
\newcommand{\bsigma}{\mathbf{\sigma}}
\newcommand{\bx}{{\boldsymbol{x}}}

\newcommand{\GP} {\bG_{{}_{\cP}}}

\newcommand{\Ker}{\mathrm{Ker}}

\newcommand{\Zbsigma}{{\Z[\bsigma]}}
\newcommand{\Zbx}{{\Z[\bx]}}

\newcommand{\Cbx}{{\C[\bx]}}
\newcommand {\sqbullet} {{\scriptstyle \blacksquare}}

\numberwithin{equation}{section}

\newtheorem{thm}{Theorem}[section]
\newtheorem{prop}[thm]{Proposition}

\theoremstyle{definition}
\newtheorem{Remark}[thm]{Remark}
\newtheorem{free text}{}[subsection]

\newcommand{\zero} {{\bar{0}}}
\newcommand{\one} {{\bar{1}}}
\newcommand{\fg}{\mathfrak{g}}
\newcommand{\fh}{\mathfrak{h}}
\newcommand{\fa}{\mathfrak{a}}
\newcommand{\fb}{\mathfrak{b}}
\newcommand{\fc}{\mathfrak{c}}
\newcommand{\fd}{\mathfrak{d}}

\newcommand{\fS}{\mathfrak{S}}

\newcommand{\fgcpx} {\fg_{{}_{\C(x)}}}

 \newcommand{\fgzx} {\fg_{{}_\Z}(\boldsymbol{x})}
 \newcommand{\fgpzx} {\fg'_{{}_\Z}(\boldsymbol{x})}
 \newcommand{\fgszx} {\fg''{{}_\Z}(\boldsymbol{x})}

 \newcommand{\fgs} {\fg(\bsigma)}

 \newcommand{\fgps} {\fg'(\bsigma)}
 \newcommand{\fgss} {\fg''(\bsigma)}
 
 \newcommand{\fgpsc} {{\fg'(\bsigma)}_{{}_\C}}

 \newcommand{\fgx} {\fg(\bx)}
 \newcommand{\fgpx} {\fg'(\bx)}
 \newcommand{\fgsx} {\fg''(\bx)}

\newcommand{\fgzs} {\fg_{{}_\Z}(\bsigma)}
 \newcommand{\fgpzs} {\fg'_{{}_\Z}(\bsigma)}
 \newcommand{\fgszs} {\fg''_{{}_\Z}(\bsigma)}

\newcommand{\Hom}{{\mathrm {Hom}}}

\newcommand{\vep}{\varepsilon}

\newcommand{\cat}[1]{{\boldsymbol{\mathsf{#1}}}} % categoria
 \newcommand{\alg} {\text{\rm ($\cat{alg}$)}}
 \newcommand{\Walg} {\text{\rm ($\cat{Walg}$)}}
 \newcommand{\salg} {\text{\rm ($\cat{salg}$)}}
 \newcommand{\Wsalg} {\text{\rm ($\cat{Wsalg}$)}}
 
 \newcommand{\lie} {\text{\rm ($\cat{Lie}$)}}
 \newcommand{\sLie} {\text{\rm ($\cat{sLie}$)}}
 \newcommand{\sHCp} {\text{\rm ($\cat{sHCp}$)}}

\newcommand{\cA}{\mathcal{A}}
\newcommand{\cB}{\mathcal{B}}

\newcommand{\cH}{\mathcal{H}}

\newcommand{\cL}{\mathcal{L}}
\newcommand{\cO}{\mathcal{O}}
\newcommand{\cI}{\mathcal{I}}
\newcommand{\cP}{\mathcal{P}}
\newcommand{\cS}{\mathcal{S}}

\newcommand{\cW}{\mathcal{W}}

\begin{document}
\allowdisplaybreaks

\newcommand{\arXivNumber}{1709.04717}

\renewcommand{\PaperNumber}{137}

\FirstPageHeading

\ShortArticleName{Singular Degenerations of Lie Supergroups of Type $D(2,1;a)$}

\ArticleName{Singular Degenerations of \\ Lie Supergroups of Type $\boldsymbol{D(2,1;a)}$}

\Author{Kenji IOHARA~$^\dag$ and Fabio GAVARINI~$^\ddag$}

\AuthorNameForHeading{K.~Iohara and F.~Gavarini}

\Address{$^\dag$~Univ Lyon, Universit\'{e} Claude Bernard Lyon 1, CNRS UMR 5208, Institut Camille Jordan,\\
\hphantom{$^\dag$}~43 Boulevard du 11 Novembre 1918, F 69622 Villeurbanne Cedex, France}
\EmailD{\href{mailto:iohara@math.univ-lyon1.fr}{iohara@math.univ-lyon1.fr}}
\URLaddressD{\url{http://math.univ-lyon1.fr/~iohara/}}

\Address{$^\ddag$~Dipartimento di Matematica, Universit\`a di Roma ``Tor Vergata'',\\
\hphantom{$^\ddag$}~Via della ricerca scientifica 1, I-00133 Roma, Italy}
\EmailD{\href{mailto:gavarini@mat.uniroma2.it}{gavarini@mat.uniroma2.it}}

\ArticleDates{Received October 31, 2017, in final form December 11, 2018; Published online December 31, 2018}

\Abstract{The complex Lie superalgebras $\fg$ of type $D(2,1;a)$~-- also denoted by $\osp(4,2;a) $~-- are usually considered for ``non-singular'' values of the parameter $a$, for which they are simple. In this paper we introduce five suitable integral forms of $\fg$, that are well-defined at singular values too, giving rise to ``singular specializations'' that are no longer simple: this extends the family of {\it simple} objects of type $D(2,1;a)$ in five different ways. The resulting five families coincide for general values of $ a$, but are different at ``singular'' ones: here they provide non-simple Lie superalgebras, whose structure we describe explicitly. We also perform the parallel construction for complex Lie supergroups and describe their singular specializations (or ``degenerations'') at singular values of $a$. Although one may work with a~single complex parameter $a$, in order to stress the overall $\fS_3$-symmetry of the whole situation, we shall work (following Kaplansky) with a two-dimensional parameter $\bsigma = (\sigma_1,\sigma_2,\sigma_3)$ ranging in the complex affine plane $\sigma_1 + \sigma_2 + \sigma_3 = 0$.}


\Classification{14A22; 17B20; 13D10}

\vspace{-5mm}

{\setcounter{tocdepth}{2}
\tableofcontents}

\vspace{-2mm}

\section{Introduction}

{\sloppy In the classification of simple finite-dimensional complex Lie superalgebras~-- due to Kac (cf.~\cite{K1})~-- a special one-parameter family occurs, whose elements~$\fg_a$ depend on a parameter $a \in \C \setminus \{0,-1\}$. These are ``generically non-isomorphic'', and all isomorphisms between them are encoded in a~free action of the symmetric group $ \fS_3 $ on the family $\{\fg_a\}_{a \in \C \setminus \{0,-1\}}$. It was pointed out in~\cite{K1} that the Cartan matrix $ A = \left(\begin{smallmatrix}
 2 & -1 & 0 \\
 -1 & 0 & a \\
 0 & -1 & 2 \end{smallmatrix} \right)$ used to define this Lie superalgebra had already appeared in~\cite{VK}, as a Cartan matrix of a one-parameter family of $16$-dimensional simple Lie algebras over a field $ k $ of characteristic $ 2 $ with $ a \in k \setminus \{0,1\}$.

 }

 For any $ a \in \{ 1, -2, -1/2\} $ one has $ \fg_1 \cong \osp(4,2) $, which is of type $ D(2,1) $: thus Kac called each $ \fg_a $ to be ``{\it of type $ D(2,1;a) $}''~-- while $ D(m,n) $ is the type of the orthosymplectic Lie superalgebra~$ \osp(2m,2n) $. For the same reason, some authors, for example \cite{ccll}~-- cf.\ also~\cite{bgl}~-- use instead notation $ \osp(4,2;a) $.
 By general theory, one can complete each of the (simple) Lie superalgebras~$ \fg_a $ and form a so-called super Harish-Chandra pair: and then one associates to the latter a corresponding complex Lie supergroup, say $ \bG_a $, whose tangent Lie superalgebra is~$ \fg_a $~-- as prescribed in Kac' classification of simple algebraic supergroups, cf.~\cite{K2}. All these~$ \bG_a $'s form a family $ \{\bG_a\}_{a \in \C \setminus \{0,-1\}} $, which bears a free $ \fS_3 $-action that induces the $ \fS_3 $-action on~$ \{\fg_a\}_{a \in \C \setminus \{0,-1\}} $. The starting point of the present paper is the following question: can we ``take the limit'' (in some sense) of $ \fg_a $ for $a$ approaching to the ``singular values'' $ a = 0 $ and $ a = -1 $? And if yes, what is the structure of the resulting ``limit'' Lie superalgebra? Similarly, we raise the same questions for the family of the supergroups~$ \bG_a $.

 In this article, we show that there are several ways to answer, in the positive, these questions. In fact, we present {\it five} possible ways to complete the family of simple Lie superalgebras $ D(2,1;a) $ with additional Lie superalgebras for the ``singular values'' $ a \in \{0,-1\} $. Each one of these new, extra objects can be thought of as a ``limit'' of the older ones; however, the existence of different options show that such ``limits'' have no intrinsic meaning, but strongly depend on some choice~-- roughly, on ``how you approach the singular point''. For each of these choices, the corresponding new objects that are ``limits'' of the (original) simple Lie superalgebras $ D(2,1;a) $ happen to be non-simple, and we describe explicitly their structure, which is different for the different choices. Therefore, we extend the old family $\{ \fg_a = \osp(4,2;a)\}_{a \in \C \setminus \{0,-1\}} $ of simple Lie superalgebras to five larger families, indexed by the points of $ \mathbb{P}^1(\C) \cup \{\ast\} $, whose elements at ``non-singular values'' $ a \in \{ 0, -1, \infty, \ast \} $ are {\it non-simple}~-- which is why we call them ``degenerations''~-- and (when comparing one family with a different one) non-isomorphic.

 By the way, our analysis is by no means exhaustive: one can still provide further ways to complete the family of the simple $ \fg_a $'s (for non-singular values of $ a $) by adding some extra objects at singular values of $ a $, right in the same spirit but with different outcomes. Our goal here is only to explain the {\it existence} and {\it non-uniqueness} of such constructions.
A few words about our construction. First, instead of working with Lie superalgebras $ \fg_a $ indexed by a single parameter $ a \in \C \setminus \{0,-1\} $~-- later extended to $ a \in \C $~-- we rather deal with a multiparameter $ \bsigma \in V := \big\{ (\sigma_1,\sigma_2,\sigma_3) \in \C^3 \,\big|\, \sum_i \sigma_i = 0 \big\} $. The starting point is a construction~-- due to Kaplansky, cf.~\cite{Kap1}; see also~\cite{Sc}~-- that for each $ \sigma \in V $ provides a Lie superalgebra $ \fg_\bsigma $: this yields a full family of Lie superalgebras $ {\big\{ \fg_\bsigma \big\}}_{\bsigma \in V} $, forming a bundle over $ V $, naturally endowed with an action of the group $ \mathcal{G} := \C^\times \times \fS_3 $ via Lie superalgebra isomorphisms. For each $ \bsigma $ in the ``general locus'' $ V^\times := V \setminus \big( \bigcup_{ i=1}^{ 3} \{\sigma_i=0\} \big) $ we have $ \fg_\bsigma \cong \fg_a $ for some $ a \in \C \setminus \{0,-1\} $ so the original family $\{ \fg_a = \osp(4,2;a)\}_{a \in \C \setminus \{0,-1\}} $ of simple Lie superalgebras
is taken into account; in addition, the $ \fg_\bsigma $'s are well-defined also at singular values $ \bsigma \in V \cap \big( \bigcup_{ i=1}^{ 3} \{\sigma_i=0\} \big) $, but there they are non-simple instead.

Thus Kaplansky's family of Lie superalgebras provides a first solution to our problem. In addition, we re-visit this construction and devise five recipes to construct similar families, as follows. For $ \bsigma \in V^\times $, we fix in $ \fg_\bsigma $ a particular $ \C $-basis, call it $B$, in such a way that the structure constants are polynomials in $ \bsigma $. When we replace $ \bsigma = (\sigma_1, \sigma_2, \sigma_3) $ with a formal parameter $ \bx = (x_1, x_2, x_ 3) $, the previous multiplication table defines a Lie superalgebra structure on the free $ \Cbx $-module with basis $ B $, denoted by $ \fg_{{}_B}(\bx) $. Then for each $ \bsigma = (\sigma_1, \sigma_2, \sigma_3) \in V $ the quotient $ \fg_{{}_B}(\bsigma) := \fg_{{}_B}(\bx) / {(x_i - \sigma_i)}_{i=1,2,3} \fg_{{}_B}(\bx) $ is a well-defined complex Lie super\-algebra, such that $ \fg_{{}_B}(\bsigma) \cong \fg_\bsigma $ for $ \bsigma \in V^\times $; thus we get a whole family $\{ \fg_{{}_B}(\bsigma)\}_{\bsigma \in V} $ as requested, that actually depends on the choice of the basis $ B $. We present five explicit examples that give rise to five different outcomes~-- one being Kaplansky's family. Indeed, at each point of the ``singular locus'' $ V \cap \big( \bigcup_{ i=1}^{ 3} \{\sigma_i=0\} \big) $ these families present different (non-isomorphic) non-simple Lie superalgebras, that we describe in detail. As a second contribution, we perform a~parallel construction at the level of Lie supergroups: namely, for each $ \bsigma \in V $ we ``complete'' each Lie superalgebra $ \fg_{{}_B}(\bsigma) $ to form a super Harish-Chandra pair, and then take the corresponding (complex holomorphic) Lie supergroup. This yields a family $\{ \bG_{{}_B}(\bsigma)\}_{\bsigma \in V} $ of Lie supergroups, with $ \bG_\bsigma $ isomorphic to $ \bG_a $ for a suitable $ a \in \C \setminus \{0,-1\} $ for non-singular values of $ \bsigma $, while $ \bG_{{}_B}(\bsigma) $ is not simple for singular values instead; moreover, the group $ \mathcal{G} := \C^\times \times \fS_3 $ freely acts on this family via Lie supergroup isomorphisms. In other words, we complete the ``old'' family of the simple Lie supergroups $ \bG_a $'s (isomorphic to suitable $ \bG_\bsigma $'s) by suitably adding new, non-simple Lie supergroups at singular values of $ \bsigma $. The construction depends on $ B $, and with our five, previously fixed choices we find five different families: for each of them, we describe explicitly the non-simple supergroups $ \bG_\bsigma $ at singular values of $ \bsigma $~-- which are referred to as ``degenerations'' of the (previously known, simple) $ \bG_a $'s.

 This analysis might be reformulated in the language of {\it deformation theory} of supermanifold~-- e.g., as treated in~\cite{va}. However, this goes beyond the scope of the present article.
 This article is organized as follows. In Section~\ref{sect_preliminaires}, we briefly recall the basic algebraic background necessary for this work, in particular, some language about {\it supermathematics}. In Section~\ref{sect: Lie-superalg_osp(4,2;s)}, we introduce our Lie superalgebras $ \fg_\bsigma = \osp(4,2;\bsigma) $. Several {\it integral forms} of the Lie superalgebra~$ \fg_\bsigma $ are introduced in Section~\ref{forms-degen.s - Lie s-alg.s}. In particular, as an application, the structure of their {\it singular degenerations} is studied in detail (Theorems \ref{thm_g-spec}, \ref{thm_g'-spec}, \ref{thm_g''-spec}, \ref{thm_g_eta-spec} and \ref{thm_g'_eta-spec}). Section~\ref{section5} is the last highlight of this paper: we introduce and analyze the Lie supergroups whose Lie superalgebras are studied in Section~\ref{forms-degen.s - Lie s-alg.s}, and we describe the (non-simple) structure of their degenerations~-- i.e., the member of the families at singular values of $ \bsigma $ (Theorems \ref{thm_G-spec}, \ref{thm_G'-spec}, \ref{thm_G''-spec}, \ref{thm_Ghat-spec} and \ref{thm_Ghat'-spec}).

 As the main objects treated in this article have many special features, most of the above descriptions are given in a down-to-earth manner, so that even the readers who are not familiar with the subject could follow easily our exposition.

\section{Preliminaries} \label{sect_preliminaires}

 In this section, we recall the notions and language of Lie superalgebras and Lie supergroups. Our purpose is to fix the terminology, but everything indeed is standard matter.

\subsection{Basic superobjects} \label{subsec:basic-sobjcs}

\looseness=-1 All throughout the paper, we work over the field $ \C $ of complex numbers (nevertheless, immediate generalizations are possible), unless otherwise stated. By {\it $ \C $-supermodule}, or {\it $ \C $-super vector space}, any $ \C $-module $ V $ endowed with a $ \Z_2 $-grading $ V = V_\zero \oplus V_\one $, where $ \Z_2 = \{\zero,\one\} $ is the group with two elements. Then $ V_\zero $ and its elements are called {\it even}, while $ V_\one $ and its elements {\it odd}. By $ |x| (\in \Z_2) $ we denote the {\it parity} of any non-zero homogeneous element, defined by the condition $ x \in V_{|x|} $.

We call {\it $ \C $-superalgebra} any associative, unital $ \C $-algebra $ A $ which is $ \Z_2 $-graded: so $ A $ has a $ \Z_2 $-grading $ A = A_\zero \oplus A_\one $, and $ A_{\mathbf{a}} A_{\mathbf{b}} \subseteq A_{\mathbf{a}+\mathbf{b}} $. Any such $ A $ is said to be {\it commutative} if $ x y = (-1)^{|x| |y|} y x $ for all homogeneous $ x, y \in A $; so, in particular, $ z^2 = 0 $ for all $ z \in A_\one $. All $ \C $-superalgebras form a category, whose morphisms are those of unital $ \C $-algebras preserving the $ \Z_2 $-grading; inside it, commutative $ \C $-superalgebras form a subcategory, that we denote by~$ \salg $. We denote by $ \alg $ the category of (associative, unital) commutative $ \C $-algebras, and by $ {(\text{\bf mod})} $ that of $ \C $-modules. Note also that there is an obvious functor $ {(\ )}_\zero \colon \salg \longrightarrow \alg $ given on objects by $ A \mapsto A_\zero $. We call {\it Weil superalgebra} any finite-dimensional commutative $ \C $-superalgebra $ A $ such that $ A = \C \oplus \mathfrak{N}(A) $ where $ \C $ is even and $ \mathfrak{N}(A) = {\mathfrak{N}(A)}_\zero \oplus {\mathfrak{N}(A)}_\one $ is a~$ \Z_2 $-graded nilpotent ideal (the {\it nilradical} of~$ A $). Every Weil superalgebra $ A $ is endowed with a~canonical epimorphisms $ p_A\colon A \relbar\joinrel\relbar\joinrel\relbar\joinrel\twoheadrightarrow \C $ and an embedding $ u_A \colon \C \lhook\joinrel\longrightarrow A $, such that \mbox{$ p_A \circ u_A = {\rm id}$}. Weil superalgebras over $ \C $ form a full subcategory of $ \salg $, denoted by~$ \Wsalg $. Finally, let $ \Walg := \Wsalg \cap \alg $ be the category of {\it Weil algebras $($over~$ \C )$}, i.e., the full subcategory of all totally even objects in $ \Wsalg $~-- namely, those whose odd part is trivial. Then the functor $ {(\ )}_\zero \colon \salg \longrightarrow \alg $ obviously restricts to a similar functor $ {(\ )}_\zero \colon \Wsalg \longrightarrow \Walg $ given again by $ A \mapsto A_\zero $.

\subsection{Lie superalgebras} \label{subsec:Lie-superalgebras}

 By definition, a {\it Lie superalgebra} is a $ \C $-supermodule $ \fg = \fg_\zero \oplus \fg_\one $ with a {\it $($Lie super$)$bracket} $ [ \cdot, \cdot ] \colon \fg \times \fg \longrightarrow \fg $, $ (x,y) \mapsto [x,y] $, which is $ \C $-bilinear, preserving the $ \Z_2 $-grading and satisfies the following (for all homogenenous $ x, y, z \in \fg $):
\begin{enumerate}\itemsep=0pt
\item[(a)] $ [x,y] + {(-1)}^{|x| |y|}[y,x] = 0 $ {\it $($anti-symmetry$)$};
\item[(b)] ${(-1)}^{|x| |z|} [x, [y,z]] + {(-1)}^{|y| |x|} [y, [z,x]] + {(-1)}^{|z| |y|} [z, [x,y]] = 0 $ {\it $($Jacobi identity$)$}.
\end{enumerate}
 In this situation, we write $ Y^{\langle 2 \rangle} := 2^{-1} [Y,Y]$ $(\!\in \fg_\zero) $ for all $ Y \in \fg_\one $. All Lie $ \C $-superalgebras form a category, denoted by $ \sLie $, whose morphisms are $ \C $-linear, preserving the $ \Z_2 $-grading and the bracket. Note that if $ \fg $ is a Lie $ \C $-superalgebra, then its even part $ \fg_\zero $ is automatically a~Lie $ \C $-algebra.

Lie superalgebras can also be described in functorial language. Indeed, let $ \lie $ be the category of Lie $ \C $-algebras. Then every Lie $ \C $-superalgebra $ \fg \in \sLie $ defines a functor
\begin{gather*} \cL_\fg \colon \ \Wsalg \relbar\joinrel\relbar\joinrel\longrightarrow \lie, \qquad A \mapsto \cL_\fg(A) := ( A \otimes \fg )_\zero = (A_\zero \otimes \fg_\zero)\oplus (A_\one \otimes \fg_\one).
\end{gather*}
 Indeed, $ A \otimes \fg $ is a Lie superalgebra (in a suitable, more general sense, over $ A $) on its own, with Lie bracket
$[ a \otimes X, a' \otimes X'] := {(-1)}^{|X| |a'|} a a' \otimes [X,X'] $; now $ \cL_\fg(A) $ is the even part of $ A \otimes \fg $, hence it is a Lie algebra on its own.

\subsection{Lie supergroups} \label{subsec:Lie-sgrps - LONG}

 We shall now recall, in steps, the notion of complex holomorphic ``Lie supergroups'', as a special kind of ``supermanifold''. The following is a very concise summary of a long, detailed theory: further details are, for instance, in \cite{bcf, dm, va}.

\begin{free text} \label{supermanifolds}
{\bf Supermanifolds.} By {\it superspace} we mean a pair $ S = \big( |S|, \cO_S \big) $ of a topological space~$ |S| $ and a sheaf of commutative superalgebras $ \cO_S $ on it such that the stalk $ \cO_{S,x} $ of $ \cO_S $ at each point $ x \in |S| $ is a local superalgebra. A {\it morphism} $ \phi \colon S \longrightarrow T $ between superspaces $ S $ and $ T $ is a~pair $( |\phi|, \phi^*) $ where $ |\phi| \colon |S| \longrightarrow |T| $ is a continuous map of topological spaces and the induced morphism $ \phi^* \colon \cO_T \longrightarrow |\phi|_* (\cO_S ) $ of sheaves on $ |T| $ is such that $ \phi_x^*(\mathfrak{m}_{|\phi|(x)}) \subseteq \mathfrak{m}_x $, where $ \mathfrak{m}_{|\phi|(x)} $ and $ \mathfrak{m}_{x} $ denote the maximal ideals in the stalks $ \cO_{T,|\phi|(x)} $ and $ \cO_{S,x} $, respectively.

 As basic model, the {\it linear supervariety} $\mathcal{C}_\C^{ p|q} $ (in Leites' terminology) is, by definition, the topological space $ \C^p $ endowed with the following sheaf of commutative superalgebras: $ \cO_{\mathcal{C}_\C^{p|q}}(U) := \cH_{\C^p}(U) \otimes_\C \Lambda_\C(\xi_1,\dots,\xi_q) $ for any open set $ U \subseteq \C^p $, where $ \cH_{\C^p} $ is the sheaf of holomorphic functions on $ \C^p $ and $ \Lambda_\C(\xi_1,\dots,\xi_q) $ is the complex Grassmann algebra on $ q $ variables $ \xi_1 $, $ \dots $, $ \xi_q $ of {\it odd} parity.
A~{\it $($complex holomorphic$)$ supermanifold} of (super)dimension $ p|q $ is a superspace $ M = ( |M|, \cO_M ) $ such that $ |M| $ is Hausdorff and second-countable and $ M $ is locally isomorphic to $ \mathcal{C}_\C^{ p|q} $, i.e., for each $ x \in |M| $ there is an open set $ V_x \subseteq |M| $ with $ x \in V_x $ and $ U \subseteq \C^p $ such that $ \cO_M{|}_{V_x} \cong \cO_{\mathcal{C}_\C^{p|q}}{|}_U $ (in particular, it is locally isomorphic to $ \mathcal{C}_\C^{ p|q} $). A {\it morphism} between holomorphic supermanifolds is just a morphism (between them) as superspaces.

 We denote the category of (complex holomorphic) supermanifolds by ($\cat{hsmfd}$).

Let now $ M $ be a holomorphic supermanifold and $ U $ an open subset in~$ |M| $. Let $ \cI_M(U) $ be the (nilpotent) ideal of $ \cO_M(U) $ generated by the odd part of the latter: then $ \cO_M \big/ \cI_M $ defines a sheaf of purely even superalgebras over $ |M| $, locally isomorphic to $ \cH_{\C^p} $. Then $ M_{\rm rd} := ( |M|, \cO_M / \cI_M ) $ is a {\it classical} holomorphic manifold, called the {\it underlying holomorphic $($sub$)$manifold} of $ M $; the standard projection $ s \mapsto \tilde{s} := s + \cI_M(U) $ (for all $ s \in \cO_M(U) $) at the sheaf level yields an embedding $ M_{\rm rd} \longrightarrow M $, so $ M_{\rm rd} $ can be seen as an embedded sub(super)manifold of $ M $. The whole construction is clearly functorial in~$ M $.

Finally, each ``classical'' manifold can be seen as a ``supermanifold'', just regarding its structure sheaf as one of superalgebras that are actually {\it totally even}, i.e., with trivial odd part. Conversely, any supermanifold enjoying the latter property is actually a manifold, nothing more. In other words, every manifold identify with a supermanifold $ M $ that actually coincides with its underlying (sub)manifold~$ M_{\rm rd} $.
\end{free text}

\begin{free text} \label{Lie supergroups - functors} {\bf Lie supergroups and the functorial approach.} A {\it group object} in the category ($\cat{hsmfd}$) is called {\it $($complex holomorphic$)$ Lie supergroup}. These objects, together with the obvious morphisms, form a subcategory among supermanifolds, denoted ($\cat{Lsgrp}$).

 Lie supergroups~-- as well as supermanifolds~-- can also be conveniently studied via a functorial approach that we now briefly recall (cf.~\cite{bcf} or~\cite{Ga3} for details).
Let $ M $ be a supermanifold. For every $ x \in |M| $ and every $ A \in \Wsalg $ we set $ M_{A,x} = \Hom_\salg ( \cO_{M,x}, A ) $ and $ M_A = \bigsqcup_{ x \in |M|} M_{A,x} $; then we define $ \cW_M \colon \Wsalg \longrightarrow \text{\rm ($\cat{set}$)} $ to be the functor given by $ A \mapsto M_A $ and $ \rho \mapsto \rho^{(M)} $ with $ \rho^{(M)} \colon M_A \longrightarrow M_B $, $ x_A \mapsto \rho \circ x_A $. Overall, this provides a functor $ \cB \colon \text{\rm ($\cat{hsmfd}$)} \longrightarrow [\Wsalg,\text{\rm ($\cat{set}$)}] $ given on objects by $ M \mapsto \cW_M $; we can now refine still more.

 Given a finite dimensional commutative algebra $ A_\zero $ over $ \C $, a (complex holomorphic) {\it $ A_\zero $-manifold} is any manifold that is locally modelled on some open subset of some finite dimensional $ A_\zero $-module, so that the differential of every change of charts is an $ A_\zero $-module isomorphism. An {\it $ A_\zero $-morphism} between two $ A_\zero $-manifolds is any smooth morphism whose differential is everywhere $ A_\zero $-linear (we then say that ``it is $ A_\zero $-smooth''). Gathering all $ A_\zero $-manifolds (for all possible~$ A $), and suitably defining morphisms among them, one defines the category $ \text{\rm ($\mathcal{A}_\zero\text{-}\cat{hmfd}$)} $ of all ``$ \mathcal{A}_\zero $-manifolds''.

A key point now is that {\it each $ \cW_M $ turns out to be a functor from $ \Wsalg $ into $ \text{\rm ($\mathcal{A}_\zero\text{-}\cat{hmfd}$)} $}.
Furthermore, let $[[ \Wsalg, \text{\rm ($\mathcal{A}_\zero\text{-}\cat{hmfd}$)} ]] $ be the subcategory of $ \big[ \Wsalg, \text{\rm ($\mathcal{A}_\zero\text{-}\cat{hmfd}$)} \big] $ with the same objects but whose morphisms are all natural transformations $ \phi \colon \mathcal{G} \relbar\joinrel\longrightarrow \mathcal{H} $ such that for every $ A \in \Wsalg $ the induced $ \phi_A \colon \mathcal{G}(A) \relbar\joinrel\longrightarrow \mathcal{H}(A) $ is $ A_\zero $-smooth. Then {\it the second key point is that if $ \phi \colon M \relbar\joinrel\longrightarrow N $ is a morphism of supermanifolds, then $ \phi_A $ is a morphism in $[[ \Wsalg, \text{\rm ($\mathcal{A}_\zero\text{-}\cat{hmfd}$)} ]] $}. The final outcome is that we have a functor $ \cS \colon \text{\rm ($\cat{hsmfd}$)} \longrightarrow [[ \Wsalg$, $\text{\rm ($\mathcal{A}_\zero\text{-}\cat{hmfd}$)} ]] $,
given on objects by $ M \mapsto \cW_M $; the key result is that {\it this embedding is full and faithful}, so that for any two supermanifolds $ M $ and $ N $ one has {\it $ M \cong N $ if and only if $ \cS(M) \cong \cS(N) $, i.e., $ \cW_M \cong \cW_N $}.

 Still relevant to us, is that the embedding $ \cS $ preserves products, hence also group objects. Therefore, {\it a supermanifold $ M $ is a {\it Lie supergroup} if and only if $ \cS(M) := \cW_M $ takes values in the subcategory $($among $ \cA_\zero $-manifolds$)$ of group objects}~-- thus each $ \cW_M(A) $ is a group.

Finally, in the functorial approach the ``classical'' manifolds (i.e., totally even supermanifolds) can be recovered as follows: in the previous construction one simply has to replace the words ``Weil superalgebras'' with ``Weil algebras'' everywhere. It then follows, in particular, that the functor of points $ \cW_{\mathcal{M}} $ of any holomorphic, manifold $ \mathcal{M} $ is actually a functor from $ \Walg $ to $ \text{\rm ($ \mathcal{A}_\zero\text{-}\cat{hmfd} $)} $; one can still see it as (the functor of points of) a {\it super}manifold~-- that is totally even, though~-- by composing it with the natural functor $ {(\ )}_\zero \colon \Wsalg \longrightarrow \Walg $. On the other hand, given any supermanifold $ M $, say holomorphic, the functor of points of its underlying submanifold $ M_{\rm rd} $ is given by $ \cW_{M_{\rm rd}}(\cA) = \cW_M(\cA) $ for each $ \cA \in \Walg $, or in short $ \cW_{M_{\rm rd}} = \cW_M{\big|}_\Walg $.

Finally, it is worth stressing that the functorial point of view on supermanifolds was originally developed~-- by Leites, Berezin, Deligne, Molotkov, Voronov and many others~-- in a slightly different way. Namely, they considered functors defined, rather than on Weil superalgebras, on {\it Grassmann $($super$)$algebras}. Actually, the two approaches are equivalent: see~\cite{bcf} for a detailed, critical analysis of the matter.

There are some advantages in restricting the focus onto Grassmann algebras. For instance, they are the superalgebras of global sections onto the superdomains of dimension $ 0|q $~-- i.e., ``super-points''. Therefore, if $ M $ is a supermanifold considered as a super-ringed space, its description via a functor defined on Grassmann algebras (only) can be really seen as the true restriction of the functor of points of $ M $, considered as a super-ringed space.

On the other hand, the use of Weil superalgebras has the advantage that one can use it to perform differential calculus on functors $ \cW_M$, much in the spirit of Weil's approach to differential calculus in algebraic geometry. Note also that some peculiar properties for Grassmann algebras are still available for every Weil superalgebra $ A $: e.g., the existence of the maps $ p_A \colon A \relbar\joinrel\relbar\joinrel\twoheadrightarrow \C $ and $ u_A \colon \C \lhook\joinrel\relbar\joinrel\longrightarrow A $, key tools in the theory (for instance, for any Lie supergroup $ G $ this implies the existence of a semidirect product splitting of the group $ G(A) $ of $A$-points of $ G $). See~\cite{bcf} for further details.
\end{free text}

\subsection{Super Harish-Chandra pairs and Lie supergroups} \label{subsec:sHCp_Lie-sgrps}
 A different way to deal with Lie supergroups (or algebraic supergroups) is via the notion of ``super Harish-Chandra pair'', that gathers together the infinitesimal counterpart~-- that of Lie superalgebra~-- and the classical (i.e., ``non-super'') counterpart~-- that of Lie group~-- of the notion of Lie supergroup. We recall it shortly, referring to~\cite{Ga3} (and~\cite{Ga2}) for further details.

\begin{free text} \label{sHCp's} {\bf Super Harish-Chandra pairs.} We call {\it super Harish-Chandra pair}~-- or just ``{\it sHCp}'' in short~-- any pair $ (G, \fg) $ such that $ G $ is a (complex holomorphic) Lie group, $ \fg $ a complex Lie superalgebra such that $ \fg_\zero = \operatorname{Lie}(G) $, and there is a (holomorphic) $ G $-action on $ \fg $ by Lie superalgebra automorphisms, denoted by $ \operatorname{Ad} \colon G \relbar\joinrel\relbar\joinrel\longrightarrow \operatorname{Aut}(\fg) $, such that its restriction to $ \fg_\zero $ is the adjoint action of $ G $ on $ \operatorname{Lie}(G) = \fg_\zero $ and the differential of this action is the restriction to $ \operatorname{Lie}(G) \times \fg = \fg_\zero \times \fg $ of the adjoint action of $ \fg $ on itself. Then a {\it morphism} $ (\Omega,\omega)\colon ( G', \fg' ) \relbar\joinrel\longrightarrow ( G'', \fg'' ) $ between sHCp's is given by a morphism of Lie groups $ \Omega\colon G' \longrightarrow G'' $ and a morphism of Lie superalgebras $ \omega\colon \fg' \longrightarrow \fg'' $ such that $ \omega{|}_{\fg_\zero} = d\Omega $ and $ \omega \circ \operatorname{Ad}_g = \operatorname{Ad}_{\Omega_+(g)} \circ \omega $ for all $ g \in G $.

We denote the category of all super Harish-Chandra pairs by $ \sHCp $.
\end{free text}

\begin{free text} \label{Liesgrp's->sHCp's}
 {\bf From Lie supergroups to sHCp's.} For any $ A \in \Wsalg $, let $ A[\varepsilon] := A[x]\big/\big(x^2\big) $, with $ \varepsilon := x \mod \big(x^2\big) $ being {\it even}. Then $ A[\varepsilon] = A \oplus A \varepsilon \in \Wsalg $, and there exists a natural morphism $ p_{{}_A}\colon A[\varepsilon] \longrightarrow A $ given by $ ( a + a'\varepsilon ) {\buildrel { p_{{}_{A_{ }}}} \over \mapsto} a $. For a Lie supergroup $ \bG $, thought of as a~functor $ \bG \colon \Wsalg \relbar\joinrel\longrightarrow \text{\rm ($ \cat{groups} $)} $~-- i.e., identifying $ \bG \cong \cW_\bG $~-- let $ \bG(p_A)\colon \bG (A[\varepsilon] ) \relbar\joinrel\longrightarrow \bG(A) $ be the morphism associated with $ p_A\colon A[\varepsilon] \relbar\joinrel\longrightarrow A $. Then there exists a unique functor $ \operatorname{Lie}(\bG)\colon \Wsalg \relbar\joinrel\longrightarrow \text{\rm ($ \cat{groups} $)}$ given on objects by $ \operatorname{Lie}(\bG)(A) := \Ker \big(\bG(p)_A\big) $. The key fact now is that {\it $ \operatorname{Lie}(\bG) $ is actually valued in the category $ \text{\rm ($ \cat{Lie} $)} $ of Lie algebras, i.e., it is a functor $ \operatorname{Lie}(\bG) \colon \Wsalg \longrightarrow \text{\rm ($ \cat{Lie} $)} $}. Furthermore, there exists a Lie superalgebra $ \fg $~-- identified with the {\it tangent superspace to $ \bG $ at the unit point}~-- such that $ \operatorname{Lie}(\bG) = \cL_\fg $ (cf.\ Section~\ref{subsec:Lie-superalgebras}). Moreover, for any $ A \in \Wsalg $ one has $ \operatorname{Lie}(\bG)(A) = \operatorname{Lie}(\bG(A)) $, the latter being the Lie algebra of the Lie group $ \bG(A) $.

 Finally, the construction $ \bG \mapsto \operatorname{Lie}(\bG) $ for Lie supergroups is actually natural, i.e., provides a functor $ \operatorname{Lie}\colon \text{\rm ($ \cat{Lsgrp} $)} \relbar\joinrel\relbar\joinrel\longrightarrow \text{\rm ($ \cat{sLie} $)} $ from Lie supergroups to Lie superalgebras.

On the other hand, each Lie supergroup $ \bG $ is a group object in the category of (holomorphic) supermanifolds: therefore, its underlying submanifold $ \bG_{\rm rd} $ is in turn a group object in the category of (holomorphic) manifolds, i.e., it is a Lie group. Indeed, the naturality of the construction $ \bG \mapsto \bG_{\rm rd} $ provides a functor from Lie supergroups to (complex) Lie groups.

 On top of this analysis, if $ \bG $ is any Lie supergroup, then $ ( \bG_{\rm rd}, \operatorname{Lie}(\bG) ) $ is a super Harish-Chandra pair; more precisely, we have a functor $ \Phi \colon \text{\rm ($ \cat{Lsgrp} $)} \relbar\joinrel\longrightarrow \sHCp $ given on objects by $ \bG \mapsto ( \bG_{\rm rd}, \operatorname{Lie}(\bG)) $ and on morphisms by $ \phi \mapsto ( \phi_{\rm rd}, \operatorname{Lie}(\phi) ) $.
\end{free text}

\begin{free text} \label{sHCp's->Liesgrp's} {\bf From sHCp's to Lie supergroups.} The functor $ \Phi\colon \text{\rm ($ \cat{Lsgrp} $)} \relbar\joinrel\longrightarrow \sHCp $ has a~quasi-inverse $ \Psi \colon \sHCp\relbar\joinrel\longrightarrow \text{\rm ($ \cat{Lsgrp} $)} $ that we can describe explicitly (see~\cite{Ga2, Ga3}).

Indeed, let $ \cP := \big( G, \fg \big) $ be a super Harish-Chandra pair, and let $ B := \{ Y_i \}_{i \in I} $ be a $ \C $-basis of $ \fg_\one $. For any $ A \in \Wsalg $, we define $ \GP(A) $ as being the group with generators the elements of the set $ \Gamma^{B}_{A} := G(A) {\textstyle \bigcup} \{ (1 + \eta_i Y_i) \}_{(i,\eta_i) \in I \times A_\one} $ and relations
\begin{gather*}
 1_{{}_G} = 1, \qquad g' \cdot g'' = g' \cdot_{{}_G} g'', \\
( 1 + \eta_i Y_i ) \cdot g = g \cdot ( 1 + c_{j_1} \eta_i Y_{j_1} ) \cdots ( 1 + c_{j_k} \eta_i Y_{j_k} )\\
\hphantom{( 1 + \eta_i Y_i ) \cdot g =}{} \ \text{with} \quad \operatorname{Ad}\big( g^{-1}\big)(Y_i) = c_{j_1} Y_{j_1} + \cdots + c_{j_k} Y_{j_k}, \\
( 1 + \eta'_i Y_i ) \cdot ( 1 + \eta''_i Y_i ) = \big( 1_{{}_G} + \eta''_i \eta'_i Y_i^{\langle 2 \rangle} \big)_{{}_G} \cdot ( 1 + ( \eta'_i + \eta''_i ) Y_i ),\\
( 1 + \eta_i Y_i ) \cdot ( 1 + \eta_j Y_j ) = ( 1_{{}_G} + \eta_j \eta_i [Y_i,Y_j] )_{{}_G} \cdot ( 1 + \eta_j Y_j ) \cdot ( 1 + \eta_i Y_i )
\end{gather*}
for $ g, g', g'' \in G(A) $, $ \eta_i, \eta'_i, \eta''_i, \eta_j \in A_\one $, $ i, j \in I $. This defines the functor $ \GP $ on objects,~and one then defines it on morphisms as follows: for any $ \varphi\colon A' \longrightarrow A'' $ in $ \Wsalg $ we let $ \GP(\varphi)\colon \GP (A' )$ $\longrightarrow \GP\big(A''\big) $ be the group morphism uniquely defined on generators by $ \GP(\varphi)( g') := G(\varphi)( g') $, $ \GP(\varphi)( 1 + \eta' Y_i ) := ( 1 + \varphi(\eta') Y_i ) $.

 One proves (see \cite{Ga2,Ga3}) that every such $ \GP $ is in fact a Lie supergroup~-- thought of as a special functor, i.e., identified with its associated Weil--Berezin functor. In addition, {\it the construction $ \cP \mapsto \GP $ is natural in $ \cP $, i.e., it yields a functor $ \Psi\colon \sHCp\relbar\joinrel\longrightarrow \text{\rm ($ \cat{Lsgrp} $)} $}; moreover, {\it the latter is a quasi-inverse to $ \Phi \colon \text{\rm ($ \cat{Lsgrp} $)} \relbar\joinrel\longrightarrow \sHCp $}.
\end{free text}

\section[Lie superalgebras of type $D(2,1;\bsigma)$]{Lie superalgebras of type $\boldsymbol{D(2,1;\bsigma)}$}\label{sect: Lie-superalg_osp(4,2;s)}

 In this section, we present the complex Lie superalgebras of type $ D(2,1;a) $. On the one hand, one can construct them directly, through Kaplansky's representation (cf. \cite{Kap2}; a widely accessible account of it is also in Scheunert's book \cite[Chapter~I, Section~1.5]{Sc}), which depend on parameters. On the other hand, for non-singular values of the parameters one can realize them via Kac' method, choosing a suitable Cartan matrix, which still depends on parameters.

\subsection[Construction of $ \fg_\bsigma $ (after Kaplansky)]{Construction of $\boldsymbol{\fg_\bsigma}$ (after Kaplansky)} \label{subsect_3.1}

 To fix notation, we set $ V := \big\{ (\sigma_1,\sigma_2,\sigma_3) \in \C^3 \,\big|\, \sum_i \sigma_i = 0 \big\} $~-- a plane in $ \C^3 $~-- and $ V^\times := V \cap ( \C^\times )^3 $, where $ \C^\times := \C \setminus \{0\} $; also, $ \boldsymbol{0} := (0,0,0) \in V \setminus V^\times $. The Lie superalgebras $ \fg_\bsigma $ we deal with will depend on a parameter $ \bsigma := (\sigma_1,\sigma_2, \sigma_3) \in V $.

\begin{free text} \label{kapl-real} {\bf Kaplansky's realization.} We recall hereafter the construction of Lie superalgebras introduced by Kaplansky (cf.~\cite{Kap2}) who denoted them by $ \varGamma(A,B,C) $, with $ A, B, C \in \C $. With a suitable normalization, and different terminology, they form the family nowadays called ``of type $ D(2,1;a) $'', with $ a \in \C $ and $ a \not\in \{0,-1\} $ to ensure simplicity; we shall stick to Kaplansky's point of view, but using the parameter $ \bsigma \in \C^3 $ (and later in $ V $) and adopting the convention of denoting by $ \osp(4,2;\bsigma) $ the Lie superalgebra of type $ D(2,1;\bsigma) $. A detailed account of Kap\-lansky's realization of these Lie algebras can be found in Scheunert's book (cf.\ \cite[Chapter~I, Section~1, Example~5]{Sc}). Recall that having a Lie superalgebra $ \fg = \fg_\zero \oplus \fg_\one $ amounts to having:
$(a)$ a Lie algebra $ \fg_\zero $; $(b)$ a $ \fg_\zero $-module $ \fg_\one $; $(c)$ a $ \fg_\zero $-valued symmetric product on $ \fg_\one $, such that the $ \fg_\zero $-action is by derivations of the (symmetric) product.

 Indeed, in the family of Lie superalgebras $ \fg_\bsigma := \osp(4,2;\bsigma) $ parts $(a)$ and $(b)$ above will stand the same for any $ \bsigma \in V $, while the dependence (of the Lie superalgebra structure) on $ \bsigma $ will actually occur only for part~$(c)$.

 {\it \underline{\text{Step $(a)$}}:} Let $ \fsl_i(2) := \C e_i \oplus \C h_i \oplus \C f_i$ $( i=1,2,3) $ be three isomorphic copy of $ \fsl(2) $, in its standard realization. Then we consider their direct sum $ \fsl_1(2) \oplus \fsl_2(2) \oplus \fsl_3(2) $ with its natural structure of Lie algebra: this will be the even part $ {\big( \fg_\bsigma \big)}_\zero $ of our $ \fg_\bsigma $.

 {\it \underline{\text{Step $(b)$}}:} Let $ \Box := \C \vert + \rangle \oplus \C \vert -\rangle $ be the (natural) tautological $ 2 $-dimensional $ \fsl(2) $-module, and $ \Box_i $ (for all $ i = 1, 2, 3 $) its $ i $-th copy for $ \fsl_i(2) $. The odd part $ {\big( \fg_\bsigma \big)}_\one $ of $ \fg_\bsigma $ is (isomorphic to) $ \Box_1 \boxtimes \Box_2 \boxtimes \Box_3 $, with its natural structure of $ \big( \fsl_1(2) \oplus \fsl_2(2) \oplus \fsl_3(2) \big) $-module. We describe a $ \C $-basis of it with the following shorthand notation: for every $ \epsilon_1, \epsilon_2, \epsilon_3 \in \{+,-\} $, we set $ v_{\epsilon_1,\epsilon_2,\epsilon_3} := | \epsilon_1\rangle \otimes | \epsilon_2\rangle \otimes | \epsilon_3\rangle \in \Box_1 \boxtimes \Box_2 \boxtimes \Box_3 $.

 {\it \underline{\text{Step $(c)$}}:} We define a projection $ \psi: \Box^{ \otimes 2} \cong S^2 \Box \oplus \wedge^2 \Box \relbar\joinrel\relbar\joinrel\twoheadrightarrow \wedge^2 \Box \cong \C $ by
\begin{gather*} \vert \pm \rangle \otimes \vert \pm \rangle \longmapsto 0, \qquad
 \vert \pm \rangle \otimes \vert \mp \rangle \longmapsto \pm 2^{-1}. \end{gather*}
Then for $ \sigma' \in \C $, we define the linear map $ p\colon \Box^{ \otimes 2} \cong S^2 \Box \oplus \wedge^2 \Box \relbar\joinrel\relbar\joinrel\twoheadrightarrow S^2 \Box \cong \fsl_2 $ by
\begin{gather*} p(u,v).w := \sigma' \big( \psi(v,w).u - \psi(w,u).v \big) \qquad \forall\, u,v,w \in \Box \end{gather*}
that more explicitly reads
\begin{gather*} \vert + \rangle \otimes \vert + \rangle \longmapsto \sigma' e, \qquad
 \vert \pm \rangle \otimes \vert \mp \rangle \longmapsto -2^{-1} \sigma' h, \qquad
 \vert -\rangle \otimes \vert - \rangle \longmapsto -\sigma' f
\end{gather*}
with $ \{e,h,f\} $ being the standard $ \fsl_2 $-triple, i.e., $ [e,f] = h $, $ [h,e] = 2e $, $ [h,f] = -2f $.

Now, for each triple $ \bsigma := (\sigma_1,\sigma_2,\sigma_3) \in \C^3 $ and each $ i \in \{1,2,3\} $, let $ p_i \colon \Box_i^{ \otimes 2} \relbar\joinrel\relbar\joinrel\twoheadrightarrow \fsl_i(2) $ be the above map with scalar factor $ \sigma' := -2\sigma_i \in \C $. The Lie superbracket $ [ \,\ ] $ on $ \fg_{\bar{1}} \times \fg_{\bar{1}} $ can be expressed as
\begin{gather*} \big[ {\otimes}_{i=1}^3 u_i, \otimes_{i=1}^3 v_i \big] = \sum_{\tau \in \fS_3} \psi(u_{\tau(1)}, v_{\tau(1)}) \psi(u_{\tau(2)}, v_{\tau(2)}) p_{\tau(3)}(u_{\tau(3)}, v_{\tau(3)}).
\end{gather*}

 Tidying everything up, we can define a bracket on $ \fg_\bsigma = {\big( \fg_\bsigma \big)}_\zero \oplus {\big( \fg_\bsigma \big)}_\one $ by the formula
\begin{gather*} [ x+v, y+w ] := [x,y ] + x.w - y.v + [v,w ] \qquad \forall\, x, y \in ( \fg_\bsigma )_\zero,\quad v, w \in ( \fg_\bsigma )_\one.\end{gather*}
\end{free text}

The following proposition resumes the outcome of this construction (see also Theorem~\ref{thm_g-spec} later on for what happens for singular values of $ \bsigma $):

\begin{prop}[{cf.\ \cite[Chapter~II, Section~4.5, p.~135]{Sc}}] \label{prop: properties-g_sigma} Let $ \bsigma \in \C^3 $.
\begin{enumerate}\itemsep=0pt
\item[$(a)$] The bracket given above defines a structure of Lie superalgebra on $ \fg_\bsigma $ if and only if the given $ \bsigma \in \C^3 $ satisfies the condition $ \sigma_1 + \sigma_2 + \sigma_3 = 0 $, that is, $ \bsigma \in V := \C^3 \bigcap \big\{ \sum_i \sigma_i = 0 \big\} $.
\item[$(b)$] Let $ \bsigma \in V $. Then the Lie superalgebra $ \fg_\bsigma $ is simple if and only if $ \bsigma \in V^\times $.
\item[$(c)$] Let $ \bsigma', \bsigma'' \in V $. Then the Lie superalgebras $ \fg_{\bsigma'} $ and $ \fg_{\bsigma''} $ are isomorphic if and only if there exists $ \tau \in \fS_3 $ such that $ \sigma'' $ and $ \tau.\sigma' $ are proportional.
\end{enumerate}
\end{prop}

 Thus, the isomorphism classes of our $ \fg_\bsigma $'s are in bijection with the orbits of the $ \fS_3 $-action onto $ {\mathbb{P}}(V) \bigcup \{ \boldsymbol{0} := (0,0,0) \} \cong {\mathbb{P}}^1_\C \bigcup \{*\} $, a complex projective line plus an extra point.

\begin{free text}{\bf The multiplicative table of $\boldsymbol{\fg_\bsigma}$.} \label{mult-table}
For later use, we record hereafter the {\it complete multiplication table} of the Lie superalgebra $ \fg_\bsigma $~-- for every $ \bsigma = (\sigma_1,\sigma_2,\sigma_3) \in V $~-- with respect to the $ \C $-basis
$\{ h_i, e_j, f_j, v_{\epsilon_1,\epsilon_2,\epsilon_3} \,|\, i, j \in \{1,2,3\}, \epsilon_1, \epsilon_2, \epsilon_3 \in \{+,-\} \} $ given from scratch. In addition, in the formulas below we also take into account the following {\it coroots}:
 \begin{gather*}
 h_{\beta_1}:= + 2^{-1} \sigma_1 h_1 - 2^{-1} \sigma_2 h_2 - 2^{-1} \sigma_3 h_3, \\
 h_{\beta_2}:= - 2^{-1} \sigma_1 h_1 + 2^{-1} \sigma_2 h_2 - 2^{-1} \sigma_3 h_3, \\
 h_{\beta_3}:= - 2^{-1} \sigma_1 h_1 - 2^{-1} \sigma_2 h_2 + 2^{-1} \sigma_3 h_3, \\
 h_\theta:= + 2^{-1} \sigma_1 h_1 + 2^{-1} \sigma_2 h_2 + 2^{-1} \sigma_3 h_3.
 \end{gather*}

 In short, the table is the following ($ \forall\, i, j \in \{1,2,3\} $, $ \epsilon_1, \epsilon_2, \epsilon_3 \in \{+,-\} $):
\begin{gather*}
[ h_i, h_j ] = 0, \qquad [ e_i, e_j ] = 0, \qquad [ f_i, f_j ] = 0, \\
[ h_i, e_j ] = 2 \delta_{ij} e_j, \qquad [ h_i, f_j ] = -2 \delta_{ij} f_j, \qquad [ e_i, f_j ] = \delta_{ij} h_j, \\
[ e_i, v_{\epsilon_1,\epsilon_2,\epsilon_3} ] = \delta_{\epsilon_i,-} v_{{(-1)}^{\delta_{1,i}}\epsilon_1,{(-1)}^{\delta_{2,i}}\epsilon_2,{(-1)}^{\delta_{3,i}}\epsilon_3}, \\
[ h_i, v_{\epsilon_1,\epsilon_2,\epsilon_3} ] = \epsilon_i v_{\epsilon_1,\epsilon_2,\epsilon_3}, \\
[ f_i, v_{\epsilon_1,\epsilon_2,\epsilon_3} ] = \delta_{\epsilon_i,+} v_{{(-1)}^{\delta_{1,i}}\epsilon_1,{(-1)}^{\delta_{2,i}}\epsilon_2,{(-1)}^{\delta_{3,i}}\epsilon_3}, \\
 [ v_{+,+,+}, v_{+,-,-} ] = -\sigma_1 e_1, \qquad [ v_{+,+,+}, v_{-,+,-} ] = -\sigma_2 e_2, \qquad \\
[ v_{+,+,+}, v_{-,-,+} ] = -\sigma_3 e_3, \qquad [ v_{+,+,-}, v_{+,-,+} ] = +\sigma_1 e_1, \\
 [ v_{+,+,-}, v_{-,+,+} ] = +\sigma_2 e_2, \qquad [ v_{+,-,+}, v_{-,+,+} ] = +\sigma_3 e_3, \\
 [ v_{+,+,+}, v_{-,-,-} ] = + 2^{-1} \sigma_1 h_1 + 2^{-1} \sigma_2 h_2 + 2^{-1} \sigma_3 h_3 = h_\theta, \\
 [ v_{+,+,-}, v_{-,-,+} ] = - 2^{-1} \sigma_1 h_1 - 2^{-1} \sigma_2 h_2 + 2^{-1} \sigma_3 h_3 = h_{\beta_3}, \\
 [ v_{+,-,+}, v_{-,+,-} ] = - 2^{-1} \sigma_1 h_1 + 2^{-1} \sigma_2 h_2 - 2^{-1} \sigma_3 h_3 = h_{\beta_2}, \\
[ v_{-,+,+}, v_{+,-,-} ] = + 2^{-1} \sigma_1 h_1 - 2^{-1} \sigma_2 h_2 - 2^{-1} \sigma_3 h_3 = h_{\beta_1}, \\
 [ v_{-,+,+}, v_{-,-,-} ] = +\sigma_1 f_1, \qquad
 [ v_{+,-,+}, v_{-,-,-} ] = +\sigma_2 f_2, \qquad \\
 [ v_{+,+,-}, v_{-,-,-} ] = +\sigma_3 f_3, \qquad
 [ v_{-,+,-}, v_{-,-,+} ] = -\sigma_1 f_1, \\
[ v_{+,-,-}, v_{-,-,+} ] = -\sigma_2 f_2, \qquad
 [ v_{+,-,-}, v_{-,+,-} ] = -\sigma_3 f_3.
 \end{gather*}
\end{free text}

\subsection{Kac' realization} \label{Kac-real}

 We show now how the Lie superalgebras $ \fg_\bsigma $ of Section~\ref{subsect_3.1}, for {\it non-singular values} $ \bsigma \in V^\times $, can be also realized as contragredient Lie superalgebras (via Kac' method).

\begin{free text} \label{1st_Kac-constr}
 {\bf First construction.} For any $ \bsigma \in V $, we consider the Cartan matrix
\begin{gather*} A_\bsigma = ( a_{i,j} )_{i=1,2,3}^{j=1,2,3} =
 \begin{pmatrix}
 0 & - \sigma_3 & - \sigma_2 \\
 - \sigma_3 & 0 & - \sigma_1 \\
 - \sigma_2 & - \sigma_1 & 0
 \end{pmatrix}
\end{gather*}
(see \cite{ccll} for a far-reaching analysis of what else is associated with such a Cartan matrix). We associate to it a Dynkin diagram for which the nodes are defined as in \cite[Section~2.5.5]{K1}, and we join any two vertices $ i $ and $ j $ with an edge labeled by $ a_{ij} $: the resulting diagram then is
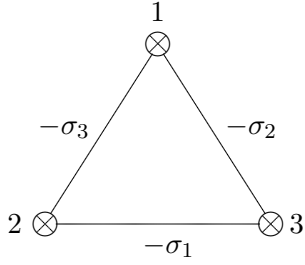
\begin{gather*}
 \setlength{\unitlength}{1mm}
\begin{picture}(50,35)(0,-5)
\put(0,0){\circle{3}}
\put(-5,-1.2){$2$}
\qbezier[20](1.05,1.05)(0,0)(-1.05,-1.05)
\qbezier[20](1.05,-1.05)(0,0)(-1.05,1.05)
\put(30,0){\circle{3}}
\put(32.5,-1.2){$3$}
\qbezier[20](31.05,1.05)(30,0)(28.95,-1.05)
\qbezier[20](31.05,-1.05)(30,0)(28.95,1.05)
\put(15,24){\circle{3}}
\put(14,27){$1$}
\qbezier[20](16.05,25.05)(15,24)(13.95,22.95)
\qbezier[20](16.05,22.95)(15,24)(13.95,25.05)
\qbezier[200](1.5,0)(15,0)(28.5,0)
\qbezier[200](0.75,1.4)(7.5,12)(14.25,22.6)
\qbezier[200](29.25,1.4)(22.5,12)(15.75,22.6)
\put(13,-3.8){$-\sigma_1$}
\put(24,12){$-\sigma_2$}
\put(-0.9,12){$-\sigma_3$}
\end{picture}
\end{gather*}
Following Kac, we consider a realization $ \big( \fh, \Pi^\vee = {\{H_{\beta_i}\}}_{i=1,2,3}$, $\Pi = {\{\beta_i\}}_{i=1,2,3} \big) $ of $ A_\bsigma $, that is
\begin{enumerate}\itemsep=0pt
 \item[1)] $ \fh $ is a $ \C $-vector space,
 \item[2)] $ \Pi^\vee $ is the set of {\it simple coroots}, a basis of $ \fh $,
 \item[3)] $ \Pi $ is the set of {\it simple roots}, a basis of $ \fh^\ast $,
 \item[4)] $ \beta_j(H_{\beta_i}) = a_{i,j} $ for all $ 1 \leq i, j \leq 3 $.
\end{enumerate}
 The (contragredient) Lie superalgebra $ \fg(A_\bsigma) $ is, by definition, the simple Lie superalgebra defined as follows. First, let $ \tilde{\fg}(A_\bsigma) $ be the Lie superalgebra with the nine generators $ H_{\beta_i}$, $X_{\pm \beta_i}$ ($ i=1,2,3 $), and relations (for $ 1 \leq i, j \leq 3 $)
\begin{gather*}
[ H_{\beta_i}, H_{\beta_j} ] = 0,\qquad [ H_{\beta_i}, X_{\pm \beta_j}] = \pm \beta_j(H_{\beta_i}) X_{\pm \beta_j}, \\
 [ X_{\beta_i}, X_{-\beta_j} ] = \delta_{i,j} H_{\beta_i},\qquad [ X_{\pm \beta_i}, X_{\pm \beta_i} ] = 0
\end{gather*}
 with parity $|H_{\beta_i}| = \bar{0} $ and $|X_{\pm\beta_i}| = \bar{1} $ for all $ i $. Then one considers the maximal homogeneous ideal $ I_\bsigma $ of $ \tilde{\fg}(A) $ which meets trivially the $ \C $-span of the generators, and finally defines $ \fg(A_\bsigma) := \tilde{\fg}(A_\bsigma) / I_\bsigma $.

A straightforward (and easy) analysis shows that $ \fg(A_\bsigma) $ is finite-dimensional. Then, by general results (cf.\ \cite[Section~2.5.1]{K1}), there exists an epimorphism $ \Phi_\bsigma\colon \fg_\bsigma \relbar\joinrel\relbar\joinrel\twoheadrightarrow \fg(A_\bsigma) $ of Lie superalgebras uniquely determined by
\begin{alignat*}{4}
 & h_{\beta_1} \mapsto H_{\beta_1}, \qquad&& h_{\beta_2} \mapsto H_{\beta_2}, \qquad && h_{\beta_3} \mapsto H_{\beta_3}, &\\
 & v_{-,+,+} \mapsto X_{+\beta_1}, \qquad && v_{+,-,+} \mapsto X_{+\beta_2}, \qquad && v_{+,+,-} \mapsto X_{+\beta_3},& \\
& v_{+,-,-} \mapsto X_{-\beta_1}, \qquad && v_{-,+,-} \mapsto X_{-\beta_2}, \qquad && v_{-,-,+} \mapsto X_{-\beta_3}.&
 \end{alignat*}

{\it When $ \bsigma $ is {\it non-singular}, that is $ \bsigma \in V^\times $, this epimorphism $ \Phi_\bsigma $ is actually an {\it isomorphism}, so that $ \osp(4,2;\bsigma) =: \fg_\bsigma \cong \fg(A_\bsigma) $}. One can see it in two ways: first, since $ \fg_\bsigma $ is simple for $ \bsigma \in V^\times $, the kernel of $ \Phi_\bsigma $ is then necessarily trivial; second, direct inspection shows that for $ \bsigma \in V^\times $ the ideal $ I_\bsigma $ has the ``correct'' codimension in $ \tilde{\fg}(A_\bsigma) $ so that $ \Phi_\bsigma $ be injective.

{\it We assume now, for the rest of the present subsection, that $ \bsigma \in V^\times $ $($non-singular case$)$}.

 Then the set $ \Delta^+ $ of positive roots of $ \fg(A_\bsigma) $ has the following description:
\begin{gather*} \Delta^+ = \big\{ \beta_1, \beta_2, \beta_3, \beta_1+\beta_2, \beta_2+\beta_3, \beta_3+\beta_1, \beta_1+\beta_2+\beta_3 \big\}.\end{gather*}
This can be seen as a consequence of $ \fg_\bsigma \cong \fg(A_\bsigma) $, or more directly by inspection (namely, describing $ I_\bsigma $ explicitly). The set of roots then is $ \Delta = \Delta^+ \cup \Delta^- $, with $ \Delta^- := -\Delta^+ $.

 The dual $ \fh^\ast $ of the Cartan subalgebra has the following description. Let $ {\{ \vep_i \}}_{i=1,2,3} \subset \fh^\ast $ be an orthogonal basis normalized by the conditions $ (\vep_i, \vep_i) = - \frac{ 1 }{ 2 } \sigma_i $ ($ i = 1, 2, 3 $); then one can verify that $ (\beta_i, \beta_j) = -\sigma_k $ with $ \{i,j,k\} = \{1,2,3\} $, where the simple roots are
\begin{gather*} \beta_1 = -\vep_1 + \vep_2 + \vep_3, \qquad \beta_2 = \vep_1 - \vep_2 + \vep_3, \qquad \beta_3 = \vep_1 + \vep_2 - \vep_3.\end{gather*}
Now let $ \Delta_{\bar{0}}^+ $ and $ \Delta_{\bar{1}}^+ $ be the set of even (resp.\ odd) positive roots. One has
\begin{gather*} \Delta_{\bar{0}}^+ = \big\{ \beta_1 + \beta_2, \beta_2 + \beta_3, \beta_3 + \beta_1 \big\} = \big\{ 2\vep_i \,|\, 1 \leq i \leq 3 \big\},
 \qquad \Delta_{\bar{1}}^+ = \big\{ \beta_1, \beta_2, \beta_3, \theta \big\}, \end{gather*}
 where $ \theta := \beta_1 + \beta_2 + \beta_3 = \vep_1 + \vep_2 + \vep_3 $ is the highest root.

 We introduce now further root vectors and coroots, defined by
\begin{alignat*}{4}
& X_{2\vep_1} := [ X_{\beta_2}, X_{\beta_3} ], \qquad&&
 X_{2\vep_2} := [ X_{\beta_3}, X_{\beta_1} ], \qquad&&
 X_{2\vep_3} := [ X_{\beta_1}, X_{\beta_2} ], &\\
& X_{-2\vep_1} := - [ X_{-\beta_2}, X_{-\beta_3} ], \qquad&&
 X_{-2\vep_2} := - [ X_{-\beta_3}, X_{-\beta_1} ], \qquad&&
 X_{-2\vep_3} := - [ X_{-\beta_1}, X_{-\beta_2} ], &\\
& H_{2\vep_1} := - ( H_{\beta_2} + H_{\beta_3} ), \qquad&&
 H_{2\vep_2} := - ( H_{\beta_3} + H_{\beta_1} ), \qquad&&
 H_{2\vep_3} := - ( H_{\beta_1} + H_{\beta_2} ). &
\end{alignat*}
 It can be checked that, for $ i, j \in \{1,2,3\} $, one has
\begin{gather} \label{[X_e,X_e] & [H_e,X_e]}
[ X_{2\vep_i}, X_{-2\vep_j} ] = \sigma_i \delta_{i,j} H_{2\vep_i}, \qquad
[ H_{2\vep_i}, X_{\pm 2\vep_j}] = \pm 2 \sigma_i \delta_{i,j} X_{\pm 2\vep_j},
\end{gather}
which implies that each $ \fa_i :=\C X_{2\vep_i} \oplus \C H_{2\vep_i} \oplus \C X_{-2\vep_i} $ (for $ 1 \leq i \leq 3 $) is a Lie sub-(super)algebra, with $ [\fa_j, \fa_k] = 0 $ for $ j \neq k $, and $ \fa_i $ {\it is isomorphic to $ \fsl(2) $ since $ \sigma_i \neq 0 $}. In particular, the even part of the Lie superalgebra $ \fg(A_\bsigma) $ is nothing but $ {\fg(A_\bsigma)}_\zero = \bigoplus_{i=1}^3 \fa_i $.

 For $ i=1,2,3 $ we set $ X_\theta^i = \big[ X_{2\vep_i}, X_{\beta_i} \big] \in {\fg(A_\bsigma)}_{\theta} $ and $ X_{-\theta}^i = \big[ X_{-2\vep_i}, X_{-\beta_i} \big] \in {\fg(A_\bsigma)}_{-\theta} $; then the following identities hold:
\begin{gather} \label{rels-X_th}
 \sum_{i=1}^3 X_{\pm \theta}^i = 0, \qquad \big[ X_\theta^j, X_{-\theta}^k \big] = - \sigma_j \sigma_k \sum_{i=1}^3 H_{\beta_i}.
\end{gather}
These formulas imply that there exists $ X_\theta \in {\fg(A_\bsigma)}_{\theta} $ and $ X_{-\theta} \in {\fg(A_\bsigma)}_{-\theta} $ such that
\begin{gather} \label{X_th^i vs. X_th}
 X_\theta^i = \sigma_i X_\theta, \qquad X_{-\theta}^i = \sigma_i X_{-\theta},
\end{gather}
for any $ 1 \leq i \leq 3 $. Hence, setting $ H_\theta = - ( H_{\beta_1} + H_{\beta_2} + H_{\beta_3} ) $, one has also $ [ X_\theta, X_{-\theta} ] = H_\theta $. Moreover, it also follows that $ H_{2\vep_1} + H_{2\vep_2} + H_{2\vep_3} = 2 H_\theta $. Eventually, from all this we see that the odd part $ {\fg(A_\bsigma)}_\one $ of the Lie superalgebra $ \fg(A_\bsigma) $ is the $ \C $-span of $\{ X_{\pm \beta_i} \}_{i=1,2,3} \cup \{ X_{\pm \theta} \} $.

 Finally, from the previous description we see that the isomorphism $ \osp(4,2;\bsigma) \cong \fg(A_\bsigma) $ can be described on basis elements by
\begin{alignat}{5}
 & h_i \mapsto \sigma_i^{-1} H_{2 \varepsilon_i}, \qquad&& e_i \mapsto \sigma_i^{-1} X_{2 \varepsilon_i}, \qquad && f_i \mapsto \sigma_i^{-1} X_{-2 \varepsilon_i}, \qquad && i = 1, 2, 3, & \nonumber\\
 & v_{-,+,+} \mapsto X_{\beta_1} \phantom{k}, \qquad && v_{+,-,+} \mapsto X_{\beta_2} \phantom{k}, \qquad && v_{+,+,-} \mapsto X_{\beta_3} \phantom{k}, \qquad && v_{+,+,+} \mapsto X_{\theta},& \nonumber\\
 & v_{+,-,-} \mapsto X_{-\beta_1}, \qquad && v_{-,+,-} \mapsto X_{-\beta_2}, \qquad && v_{-,-,+} \mapsto X_{-\beta_3}, \qquad && v_{-,-,-} \mapsto X_{-\theta}. &\label{eq: descr-isom_Kac-Kap}
 \end{alignat}

{\it Note} that in first line of \eqref{eq: descr-isom_Kac-Kap} the non-singularity of $ \bsigma \in V^\times $ plays a key role!
\end{free text}

\begin{free text} \label{2nd_Kac-constr} {\bf Second construction.} For the reader's convenience, we present now a second construction based upon a different, more familiar Cartan matrix (and associated Dynkin diagram). The link with the previous construction of Section~\ref{1st_Kac-constr} is through the application of the odd reflection with respect to the root~$ \beta_2 $, following V.~Serganova (see~\cite{Se2}).

 To begin with, set $ \alpha_1 = \beta_2 + \beta_3 $, $ \alpha_2 = -\beta_2 $, $ \alpha_3 = \beta_1 + \beta_2 $. Then $ \Pi' := {\{\alpha_i\}}_{i=1,2,3} $ is another set of simple roots of $ \fg(A_\bsigma) $, which is not Weyl-group conjugate to $ \Pi $; the corresponding set of coroots $ (\Pi')^\vee = {\big\{ h_{\alpha_i} \big\}}_{i=1,2,3} $ should be taken as $ h_{\alpha_1} = H_{2\vep_1} $, $ h_{\alpha_2} = H_{\beta_2} $, $ h_{\alpha_3} = H_{2\vep_3} $. With such a choice, the associated Cartan matrix $ A'_\bsigma := ( a'_{i,j} := \alpha_j(h_{\alpha_i}) )_{i=1,2,3}^{j=1,2,3} $ is given by
\begin{gather*} A'_\bsigma =
 \begin{pmatrix}
 2 \sigma_1 & - \sigma_1 & 0 \\
 - \sigma_1 & 0 & - \sigma_3 \\
 0 & - \sigma_3 & 2 \sigma_3
 \end{pmatrix}
 = \sigma_1
 \begin{pmatrix}
 2 & -1 & 0 \\
 -1 & 0 & -\frac{\sigma_3}{\sigma_1} \\
 0 & -\frac{\sigma_3}{\sigma_1} & 2\frac{\sigma_3}{\sigma_1}
 \end{pmatrix}, \end{gather*}
 where the second equality is available only if $ \sigma_1 \neq 0 $. In particular, {\it for $ \sigma_1 \neq 0 $ and $ \sigma_3 \neq 0 $, our original $ \fg(A_\bsigma) $ can be also defined via the Cartan matrix}
\begin{gather*} A'_a =
 \begin{pmatrix}
 2 & -1 & 0 \\
 1 & 0 & a \\
 0 & -1 & 2
 \end{pmatrix} \end{gather*}
with $ a := \frac{\sigma_3}{\sigma_1} $. When $ a \not= -1 $, this in turn corresponds~-- following Kac' conventions, up to renumbering the vertices (cf.~\cite[Section~2.5]{K1})~-- to the simple Lie superalgebra attached to the following {\it distinguished} (i.e., with just one odd vertex) Dynkin diagram of type $ D(2,1;a) $
\begin{gather*}
\setlength{\unitlength}{1mm}
\begin{picture}(50,15)(0,0)
\put(0.6,10){$\alpha_1$}
\put(9.6,10){$\alpha_2$}
\put(18.6,10){$\alpha_3$}
\put(2, 5){\circle{3}}
\put(3.5, 5){\line(1,0){6}}
\put(11, 5){\circle{3}}
\put(12.5, 5){\line(1,0){6}}
\put(20, 5){\circle{3}}
\qbezier[90](9.9,4)(11,5)(12.1,6)
\qbezier[90](12.1,4)(11,5)(9.9,6)
\put(6,2.5){${}_1$}
\put(14.8,2.5){${}_a$}
\end{picture}
\end{gather*}
(instead, $ a = -1 $ corresponds to $ \sigma_2 = 0 $, that is a {\it singular} case). Kac' results tell us that, {\it for all $ a \in \C \setminus \{-1,0\} $}, the set of positive roots with respect to $ \Pi' $ is given by
\begin{gather*} \Delta'^{,+} = \big\{ \alpha_1, \alpha_2, \alpha_3, \alpha_1+\alpha_2, \alpha_2+\alpha_3, \alpha_1+\alpha_2+\alpha_3, \alpha_1+2\alpha_2+\alpha_3 \big\},\end{gather*}
while the coroots can be expressed as
\begin{gather*}
 h_{\alpha_1} = H_{\beta_2 + \beta_3} = - ( H_{\beta_2} + H_{\beta_3} ), \qquad h_{\alpha_2} = H_{-\beta_2} = H_{\beta_2}, \\
 h_{\alpha_3} = H_{\beta_2 + \beta_1} = - \big( H_{\beta_2} + H_{\beta_1} \big), \\
 h_{\alpha_1 + \alpha_2} = -H_{\beta_3} = h_{\alpha_1} + h_{\alpha_2}, \qquad h_{\alpha_3 + \alpha_2} = -H_{\beta_1} = h_{\alpha_3} + h_{\alpha_2}, \\
 h_{\alpha_1 + \alpha_2 + \alpha_3} = H_{\beta_1 + \beta_2 + \beta_3} = -H_{\beta_1} - H_{\beta_2} - H_{\beta_3} = h_{\alpha_1} + h_{\alpha_2} + h_{\alpha_3}, \\
 h_{\alpha_1 + 2\alpha_2 + \alpha_3} = H_{\beta_1 + \beta_3} = -H_{\beta_1} - H_{\beta_3} = h_{\alpha_1} + 2 h_{\alpha_2} + h_{\alpha_3}.
\end{gather*}
\end{free text}

\begin{free text} \label{singul-case} {\bf The singular case.} We saw in Section~\ref{1st_Kac-constr} that for every non-singular parameter $ \bsigma \in V^\times $, we have a Lie superalgebra isomorphism $ \Phi_\bsigma\colon \osp(4,2;\bsigma) = \fg_\bsigma {\buildrel \cong \over {\lhook\joinrel\relbar\joinrel\relbar\joinrel\relbar\joinrel\twoheadrightarrow}} \fg(A_\bsigma) $. For singular values $ \bsigma \in V \setminus V^\times $, instead, the epimorphism $ \Phi_\bsigma $ is no longer an isomorphism: indeed, in this case, the Lie ideal $ I_\bsigma $ in $ \tilde{\fg}(A_\bsigma) $ is bigger than in the non-singular case, for instance (as a straightforward calculation shows), we have
\begin{gather*} \big[ X_{\beta_i}, X_{\beta_j} \big] \in I_\bsigma \ \iff \ \sigma_k = 0 \qquad \forall\, \{i,j,k\} = \{1,2,3\}. \end{gather*}

Similarly, the contragredient Lie superalgebra $ \fg (A'_a ) $ of Section~\ref{2nd_Kac-constr} can be defined also for the ``singular value'' $ a = -1 $ (corresponding to $ \sigma_1 + \sigma_3 = 0 $, which is equivalent to $ \sigma_2 = 0 $). However, in this case the corresponding Lie ideal $ I'_a $ in $ \tilde{\fg} (A'_a ) $ is bigger, as we have, for instance,
\begin{gather*} [ [ [ X_{\alpha_1}, X_{\alpha_2} ], X_{\alpha_3} ], X_{\alpha_2} ] \in I'_a \quad \iff \quad a = -1,
\end{gather*}
therefore $ \fg\big(A'_{a=-1}\big) $ has strictly smaller dimension than $ \osp(4,2;\bsigma) $ for $ \bsigma := (1,0,-1) $, whereas $ \fg (A'_a ) \cong \fg_\bsigma := \osp(4,2;\bsigma) $ via $ \Phi_\bsigma $ for all $ \bsigma = (1,a+1,a ) \not= (1,0,-1) $.

 This shows that, in a sense, describing our objects $ \fg_\bsigma = \osp(4,2;\bsigma) $ of Section~\ref{subsect_3.1} as contragredient Lie superalgebras is problematic, so to say, at singular values of $ \bsigma $, in that the contragredient construction yields not the outcome we are looking for.
\end{free text}

\subsection[Bases of $ \fg_\bsigma $]{Bases of $\boldsymbol{\fg_\bsigma}$} \label{bases-g_s}

In this subsection, for any given $ \bsigma \in V^\times $ we sort out three special bases of $ \fg_\bsigma $. Later on (in Section~\ref{forms-degen.s - Lie s-alg.s}) we use them to construct three different families, indexed by~$ V $, of Lie superalgebras: by construction these families will coincide on all non-singular parameters $ \bsigma \in V^\times $ but will actually differ instead on singular values $ \bsigma \in V \setminus V^\times $.

\begin{free text} \label{first-basis} {\bf First basis.} Let $ B := \{ H_{2\vep_1}, H_{2\vep_2}, H_{2\vep_3} \} \bigcup \{ X_\alpha \}_{\alpha \in \Delta} $ and $ B_+ := B \cup \{ H_\theta \} $ be the subsets of $ \fg_\bsigma $ whose elements are defined by
\begin{gather}
 H_{2 \varepsilon_i} := \sigma_i h_i, \qquad X_{\phantom{+}2 \varepsilon_i} := \sigma_i e_i, \qquad X_{-2 \varepsilon_i} := \sigma_i f_i, \qquad i = 1, 2, 3, \qquad H_{\theta} := 2^{-1} \sum_{i=1}^3 \sigma_i h_i, \nonumber\\
X_{\beta_1}\phantom{k} := v_{-,+,+}, \qquad X_{\beta_2} \phantom{k}:= v_{+,-,+}, \qquad X_{\beta_3}\phantom{k} := v_{+,+,-}, \qquad \phantom{k}X_{\theta} := v_{+,+,+}, \nonumber\\
 X_{-\beta_1} := v_{+,-,-}, \qquad X_{-\beta_2} := v_{-,+,-}, \qquad X_{-\beta_3} := v_{-,-,+}, \qquad X_{-\theta} := v_{-,-,-},\label{1st-basis-g}
\end{gather}
 (cf.\ Section~\ref{kapl-real}); the analysis in Section~\ref{kapl-real} tells us that, for every $ \bsigma \in V^\times $, the set $ B $ is a~$ \C $-basis of $ \fg_\bsigma $, hence $ B_+ $ is a spanning set. Again from Section~\ref{subsect_3.1}~-- in particular Section~\ref{mult-table}~-- for the Lie brackets among the elements of~$ B_+ $ we find the following multiplication table
\begin{gather*}
 [ H_{2\vep_i}, H_{2\vep_j} ] = 0, \qquad
 [ H_{2\vep_i}, X_{\pm 2\vep_j} ] = \pm 2 \delta_{i,j} \sigma_i X_{\pm 2\vep_j}, \\
 [ X_{2\vep_i}, X_{2\vep_j} ] = 0, \qquad
 [ X_{-2\vep_i}, X_{-2\vep_j} ] = 0, \qquad
 [ X_{2\vep_i}, X_{-2\vep_j} ] = \delta_{i,j} \sigma_i H_{2\vep_i}, \\
 [ H_{2\vep_i}, X_{\pm\beta_j} ] = \pm {(-1)}^{\delta_{i,j}} \sigma_i X_{\pm\beta_j}, \qquad
 [ H_{2\vep_i}, X_{\pm\theta} ] = \pm \sigma_i X_{\pm\theta}, \\
 [ H_\theta, X_{\pm 2\vep_i} ] = \pm \sigma_i X_{2\vep_i}, \qquad
 [ H_\theta, X_{\pm \beta_i} ] = \mp \sigma_i X_{\pm \beta_i}, \qquad
 [ H_\theta, X_{\pm \theta} ] = 0, \\
 [ X_{2\vep_i}, X_{\beta_j} ] = \delta_{i,j} \sigma_i X_{\theta}, \qquad
 [ X_{2\vep_i}, X_{-\beta_j} ] = (1-\delta_{i,j}) \sigma_i X_{\beta_k}, \\
 [ X_{-2\vep_i}, X_{\beta_j} ] = (1-\delta_{i,j}) \sigma_i X_{-\beta_k}, \qquad
 [ X_{-2\vep_i}, X_{-\beta_j} ] = \delta_{i,j} \sigma_i X_{-\theta}, \\
 [ X_{2\vep_i}, X_{\theta} ] = 0, \quad
 [ X_{2\vep_i}, X_{-\theta}] = \sigma_i X_{-\beta_i}, \qquad
 [ X_{-2\vep_i}, X_{\theta}] = \sigma_i X_{\beta_i}, \qquad
 [ X_{-2\vep_i}, X_{-\theta} ] = 0, \\
 [ X_{\beta_i}, X_{\beta_j} ] = (1-\delta_{i,j}) X_{2\vep_k}, \qquad
 [ X_{-\beta_i}, X_{-\beta_j} ] = -(1-\delta_{i,j}) X_{-2\vep_k}, \\
 [ X_{\beta_i}, X_{-\beta_j} ] = \delta_{i,j} ( H_{2\vep_i} - H_\theta ), \\
 [ X_{\beta_i}, X_{\theta} ] = 0, \qquad
 [ X_{\beta_i}, X_{-\theta} ] = X_{-2\vep_i}, \qquad
 [ X_{-\beta_i}, X_{\theta} ] = -X_{2\vep_i}, \qquad
 [ X_{-\beta_i}, X_{-\theta}] = 0, \\
 [ X_{\theta}, X_{\theta}] = 0, \qquad
 [ X_{\theta}, X_{-\theta} ] = H_\theta, \qquad
 [ X_{-\theta}, X_{-\theta} ] = 0
 \end{gather*}
for all $ i, j \in \{1,2,3\} $, with $ k \in \{1,2,3\} \setminus \{i,j\} $.
\end{free text}

\begin{free text} \label{second-basis} {\bf Second basis.} Let now $ B' := \{ H'_{2\vep_1}, H'_{2\vep_2}, H'_{2\vep_3} \} \bigcup \{ X'_\alpha \}_{\alpha \in \Delta} $ and $ B'_+ := B' \cup \{ H'_\theta \} $ be the subsets of $ \fg_\bsigma $ with elements
\begin{gather}
 H'_{2 \varepsilon_i} := h_i, \qquad X'_{+2 \varepsilon_i} := e_i, \qquad X'_{-2 \varepsilon_i} := f_i, \qquad i = 1, 2, 3, \qquad H'_{\theta} := 2^{-1} \sum_{i=1}^3 \sigma_i h_i, \nonumber\\
 X'_{\beta_1} := v_{-,+,+}, \qquad X'_{\beta_2} := v_{+,-,+}, \qquad X'_{\beta_3} := v_{+,+,-}, \qquad X'_{\theta} := v_{+,+,+}, \nonumber\\
 X'_{-\beta_1} := v_{+,-,-}, \qquad X'_{-\beta_2} := v_{-,+,-}, \qquad X'_{-\beta_3} := v_{-,-,+}, \qquad X'_{-\theta} := v_{-,-,-},\label{2nd-basis-g}
 \end{gather}
 (cf.\ Section~\ref{kapl-real}). Again from Section~\ref{kapl-real} we see that, for every $ \bsigma \in V $ (including also the singular locus), $ B' $ is a $ \C $-basis of $ \fg_\bsigma $, so $ B'_+ $ is a spanning set: indeed, $ B' $ is nothing but a different notation for the natural, built-in $ \C $-basis of $ \fg_\bsigma $ in Kaplansky's realization (cf.\ Section~\ref{kapl-real}). Then from Section~\ref{mult-table} we get the following multiplication table for Lie brackets among elements of~$ B'_+ $
\begin{gather*}
 [ H'_{2\vep_i}, H'_{2\vep_j} ] = 0, \qquad
 [ H'_{2\vep_i}, X'_{\pm 2\vep_j} ] = \pm 2 \delta_{i,j} X'_{\pm 2\vep_j}, \cr
 [ X'_{2\vep_i}, X'_{2\vep_j} ] = 0, \qquad
 [ X'_{-2\vep_i}, X'_{-2\vep_j} ] = 0, \qquad
 [ X'_{2\vep_i}, X'_{-2\vep_j} ] = \delta_{i,j} H'_{2\vep_i}, \\
 [ H'_{2\vep_i}, X'_{\pm\beta_j} ] = \pm {(-1)}^{\delta_{i,j}} X'_{\pm\beta_j}, \qquad
 [ H'_{2\vep_i}, X'_{\pm\theta} ] = \pm X'_{\pm\theta}, \\
 [ H'_\theta, X'_{\pm 2\vep_i} ] = \pm \sigma_i X'_{2\vep_i}, \qquad
 [ H'_\theta, X'_{\pm \beta_i} ] = \mp \sigma_i X'_{\pm \beta_i}, \qquad
 [ H'_\theta, X'_{\pm \theta} ] = 0,\\
 [ X'_{2\vep_i}, X'_{\beta_j} ] = \delta_{i,j} X'_{\theta}, \qquad
 [ X'_{2\vep_i}, X'_{-\beta_j} ] = (1-\delta_{i,j}) X'_{\beta_k}, \\
 [ X'_{-2\vep_i}, X'_{\beta_j} ] = (1-\delta_{i,j}) X'_{-\beta_k}, \qquad
 [ X'_{-2\vep_i}, X'_{-\beta_j} ] = \delta_{i,j} X'_{-\theta}, \\
 [ X'_{2\vep_i}, X'_{\theta} ] = 0, \qquad
 [ X'_{2\vep_i}, X'_{-\theta}] = X'_{-\beta_i}, \qquad
 [ X'_{-2\vep_i}, X'_{\theta}] = X'_{\beta_i}, \qquad
 [ X'_{-2\vep_i}, X'_{-\theta} ] = 0, \\
 [ X'_{\beta_i}, X'_{\beta_j} ] = (1-\delta_{i,j}) \sigma_i X'_{2\vep_k}, \qquad
 [ X'_{-\beta_i}, X'_{-\beta_j} ] = -(1-\delta_{i,j}) \sigma_i X'_{-2\vep_k}, \\
 [ X'_{\beta_i}, X'_{-\beta_j} ] = \delta_{i,j} ( \sigma_i H'_{2\vep_i} - H'_\theta ), \\
 [ X'_{\beta_i}, X'_{\theta} ] = 0, \qquad
 [ X'_{\beta_i}, X'_{-\theta} ] = \sigma_i X'_{-2\vep_i}, \qquad
 [ X'_{-\beta_i}, X'_{\theta} ] = -\sigma_i X'_{2\vep_i}, \qquad
 [ X'_{-\beta_i}, X'_{-\theta}] = 0, \\
 [ X'_{\theta}, X'_{\theta}] = 0, \qquad
 [ X'_{\theta}, X'_{-\theta} ] = H'_\theta, \qquad
 [ X'_{-\theta}, X'_{-\theta} ] = 0,
 \end{gather*}
 for all $ i, j \in \{1,2,3\} $, with $ k \in \{1,2,3\} \setminus \{i,j\} $.
\end{free text}

\begin{free text} \label{third-basis} {\bf Third basis.} Let now $ \bsigma \in V^\times $ (generic case again!). We fix as third basis (and spanning set) of $ \fg_\bsigma $ a suitable blending of the two ones in Sections~\ref{first-basis} and~\ref{second-basis} above. Namely, we set $ B'' := \{ H'_{2\vep_1}, H'_{2\vep_2}, H'_{2\vep_3} \} \bigcup \{ X_\alpha \}_{\alpha \in \Delta} $ and $ B''_+ := B'' \cup \big\{ H'_\theta \big\} $, with elements defined by~\eqref{2nd-basis-g} and~\eqref{1st-basis-g}. Then $ B'' $ is another $ \C $-basis, and $ B''_+ $ another spanning set, of $ \fg_\bsigma $.

 For later use we record hereafter the complete table of Lie brackets among elements of $ B''_+ $, which can be argued at once from those for $ B_+ $ and $ B'_+ $ in Sections~\ref{1st-basis-g} and~\ref{2nd-basis-g}, respectively:
\begin{gather*}
 [ H'_{2\vep_i}, H'_{2\vep_j} ] = 0, \qquad
 [ H'_{2\vep_i}, X_{\pm 2\vep_j} ] = \pm 2 \delta_{i,j} X_{\pm 2\vep_j}, \\
 [ X_{2\vep_i}, X_{2\vep_j} ] = 0, \qquad
 [ X_{-2\vep_i}, X_{-2\vep_j} ] = 0, \qquad
 [ X_{2\vep_i}, X_{-2\vep_j} ] = \delta_{i,j} \sigma_i^{ 2} H'_{2\vep_i}, \\
 [ H'_{2\vep_i}, X_{\pm\beta_j} ] = \pm {(-1)}^{\delta_{i,j}} X_{\pm\beta_j}, \qquad
 [ H'_{2\vep_i}, X_{\pm\theta} ] = \pm X_{\pm\theta}, \\
 [ H'_\theta, X_{\pm 2\vep_i} ] = \pm \sigma_i X_{2\vep_i}, \qquad
 [ H'_\theta, X_{\pm \beta_i} ] = \mp \sigma_i X_{\pm \beta_i}, \qquad
 [ H'_\theta, X_{\pm \theta} ] = 0, \\
 [ X_{2\vep_i}, X_{\beta_j} ] = \delta_{i,j} \sigma_i X_{\theta}, \qquad
 [ X_{2\vep_i}, X_{-\beta_j} ] = (1-\delta_{i,j}) \sigma_i X_{\beta_k}, \\
 [ X_{-2\vep_i}, X_{\beta_j} ] = (1-\delta_{i,j}) \sigma_i X_{-\beta_k}, \qquad
 [ X_{-2\vep_i}, X_{-\beta_j} ] = \delta_{i,j} \sigma_i X_{-\theta}, \\
 [ X_{2\vep_i}, X_{\theta} ] = 0, \qquad
 [ X_{2\vep_i}, X_{-\theta}] = \sigma_i X_{-\beta_i}, \qquad
 [ X_{-2\vep_i}, X_{\theta}] = \sigma_i X_{\beta_i}, \\
 [ X_{-2\vep_i}, X_{-\theta} ] = 0, \qquad
 [ X_{\beta_i}, X_{\beta_j} ] = (1-\delta_{i,j}) X_{2\vep_k}, \\
 [ X_{-\beta_i}, X_{-\beta_j} ] = -(1-\delta_{i,j}) X_{-2\vep_k}, \qquad
 [ X_{\beta_i}, X_{-\beta_j} ] = \delta_{i,j} ( \sigma_i H'_{2\vep_i} - H'_\theta ), \\
 [ X_{\beta_i}, X_{\theta} ] = 0, \qquad
 [ X_{\beta_i}, X_{-\theta} ] = X_{-2\vep_i}, \qquad
 [ X_{-\beta_i}, X_{\theta} ] = -X_{2\vep_i}, \qquad
 [ X_{-\beta_i}, X_{-\theta}] = 0, \\
 [ X_{\theta}, X_{\theta}] = 0, \qquad
 [ X_{\theta}, X_{-\theta} ] = H'_\theta, \qquad
 [ X_{-\theta}, X_{-\theta} ] = 0,
\end{gather*}
for all $ i, j \in \{1,2,3\} $, with $ k \in \{1,2,3\} \setminus \{i,j\} $.
\end{free text}

\section[Integral forms $\&$ degenerations for Lie superalgebras of type $ D(2,1;\bsigma)$]{Integral forms $\&$ degenerations for Lie superalgebras\\ of type $\boldsymbol{D(2,1;\bsigma)}$} \label{forms-degen.s - Lie s-alg.s}

Let $ \mathfrak{l} $ be any Lie (super)algebra over a field $ \K $, and $ R $ any subring of~$ \K $. By {\it integral form of~$ \mathfrak{l} $ over~$ R $}, or {\it (integral) $ R $-form of $ \mathfrak{l} $}, we mean by definition any Lie $ R $-sub(super)algebra $ \mathfrak{t}_R $ of~$ \mathfrak{l} $ whose scalar extension to $ \K $ is~$ \mathfrak{l} $ itself: in other words $ \K \otimes_R \mathfrak{t}_R \cong \mathfrak{l} $ as Lie (super)algebras over~$\K$. In this subsection we introduce five particular integral forms of $ \mathfrak{l} = \fg_{\bsigma} $, and study some remarkable specializations of them. Let $ \Delta:=\Delta^+\cup(-\Delta^+) $ be the root system of $ \fg_\bsigma $.

 From now on, for any $ \bsigma := (\sigma_1, \sigma_2, \sigma_3) \in V := \big\{ \bsigma \in \C^3 \,|\, \sum\limits_{i=1}^3 \sigma_i = 0 \big\} $ we denote by $ \Zbsigma $ the (unital) subring of $ \C $ generated by $ \{ \sigma_1, \sigma_2, \sigma_3 \} $.

 We warn the reader that the choice of a $ \Z[\bsigma] $-form becomes very important when one considers a singular degeneration: one cannot speak instead of {\it the} singular degeneration, in that any degeneration actually depends not only on the specific specialization value taken by $ \bsigma $ but also on the previous choice of a specific $ \Z[\bsigma] $-form, that must be fixed in advance. Some specific features of this phenomenon are presented in Theorems~\ref{thm_g-spec}, \ref{thm_g'-spec}, \ref{thm_g''-spec} etc.

\subsection[First family: the Lie superalgebras $ \fgs $]{First family: the Lie superalgebras $\boldsymbol{\fgs}$} \label{sect-g-sigma}

\begin{free text} \label{subsect-g(sigma)}
 {\bf Construction of the $\boldsymbol{\fgs}$'s.} For any $ \bsigma \in V $ (cf.\ Section~\ref{subsect_3.1}), let $ \fgzs $ be the Lie superalgebra over $ \Zbsigma $ defined as follows. As a $ \Zbsigma $-module, $ \fgzs $ is spanned by the formal set of $ \Zbsigma $-linear generators $ B_\fg := \{ H_{2\vep_1}, H_{2\vep_2}, H_{2\vep_3}, H_\theta \} \bigcup \{ X_\alpha \}_{\alpha \in \Delta} $ subject only to the single relation $ H_{2\vep_1} + H_{2\vep_2} + H_{2\vep_3} = 2 H_\theta $. Then it follows, in particular, that $ \fgzs $ is clearly a free $ \Zbsigma $-module, with basis $ B_\fg \setminus \{H_{2\vep_i}\} $ for any $ 1 \leq i \leq 3 $. The Lie superalgebra structure of $ \fgzs $ is defined by the formulas in Section~\ref{first-basis}, now taken as definitions of the Lie brackets among the (linear) generators of $ \fgzs $ itself.
 Overall, all these $ \fgzs $'s form a family of Lie superalgebras indexed by the points of the complex plane~$ V $. Moreover, taking $ \fgs := \C \otimes_{\raisebox{-.7mm}{\hbox{$ \scriptstyle \Zbsigma $}}} \fgzs $ for all $ \bsigma \in V $ we find a more regular situation, as now these (extended) Lie superalgebras $ \fgs $ all share~$ \C $ as their common ground ring. Moreover,
\begin{gather} \label{gs=g_sing}
 \text{\it in the non-singular case, i.e., for $ \bsigma_i \in V^\times $, we have $ \fgs \cong \fg_{\bsigma} $}
\end{gather}
 by Section~\ref{first-basis} and the very definition of $ \fgs $ itself.
 Indeed, the analysis in Section~\ref{first-basis}~-- describing a $ \C $-basis and its multiplication table for $ \fg_\bsigma $ with $ \bsigma \in V^\times $~-- prove that {\it for all \mbox{$ \bsigma \in V^\times $}, the Lie superalgebra $ \fgzs $ over $ \Zbsigma $ identifies to an integral $ \Zbsigma $-form of $ \fg_\bsigma $}.
 In order to formalize the description of the family $\{ \fgs\}_{\bsigma \in V} $, we proceed as follows. Let $ \Zbx := \Z[V] \cong \Z[x_1,x_2,x_3] / ( x_1 + x_2 + x_ 3 ) $ be the ring of global sections of the $ \Z$-scheme associated with~$ V $. In the construction of~$ \fgzs $, formally replace~$ \bx $ to $ \bsigma $ (hence the $ x_i $'s to the $ \sigma_i $'s): this makes sense, provides a meaningful definition of a Lie superalgebra over $ \Zbx $, denoted by $ \fgzx $, and then $ \fgx := \Cbx \otimes_{\raisebox{-.7mm}{\hbox{$ \scriptstyle \Zbx $}}} \fgzx $ by scalar extension, which is a Lie superalgebra over $ \Cbx:=\C \otimes_{\Z}\Zbx $. Definitions imply that, for any $ \bsigma \in V $, we have a Lie $ \Zbsigma $-superalgebra isomorphism
\begin{gather*} \fgzs \cong \Zbsigma \mathop{\otimes}\limits_\Zbx \fgzx \end{gather*}
-- through the ring isomorphism $ \Zbsigma \cong \Zbx \Big/ \big( x_i-\sigma_i \big)_{i=1,2,3} $~-- and similarly
\begin{gather*}\fgs \cong \C \mathop{\otimes}\limits_{\Cbx} \fgx \end{gather*}
as Lie $ \C $-superalgebras, through the ring isomorphism $ \C \cong \Cbx / ( x_i-\sigma_i )_{i=1,2,3} $. In geometrical language, all this can be formulated as follows. The Lie superalgebra $ \fgx $~-- being a free, finite rank $ \Cbx $-module~-- defines a coherent sheaf $ \mathcal{L}_{ \fg_{{}_\Cbx}} $ of Lie superalgebras over $ \operatorname{Spec} (\Cbx)$. Moreover, there exists a unique fibre bundle
 over $ \operatorname{Spec} (\Cbx) $, say $ \mathbb{L}_{ \fg_{{}_\Cbx}} $, whose sheaf of sections is exactly $ \mathcal{L}_{ \fg_{{}_\Cbx}} $. This fibre bundle can be thought of as a (total) deformation space over the base space $ \operatorname{Spec} (\Cbx) $, in which every fibre can be seen as a ``deformation'' of any other one, and also any single fibre can be seen as a degeneration of the original Lie superalgebra $ \fgx $. Moreover, the fibres of $ \mathbb{L}_{ \fg_{{}_\Cbx}} $ on $ \operatorname{Spec} (\Cbx) = V \cup \{\star\} $ are, by definition, given by $( \mathbb{L}_{ \fg_{{}_\Cbx}})_\bsigma = \C \mathop{\otimes}\limits_\Cbx \fgx \cong \fgs $ for any {\it closed point} $ \bsigma \in V \subseteq \operatorname{Spec} (\Cbx) $, while for the {\it generic point} $ \star \in \operatorname{Spec} (\Cbx) $ we have $( \mathbb{L}_{ \fg_{{}_\Cbx}})_\star = \C(\bx) \mathop{\otimes}\limits_\Cbx \fgx ( =: \fgcpx) $. Finally, it follows from our construction that these sheaf and fibre bundle do admit an action of $ \C^\times \times \fS_3 $, that on the base space $ \operatorname{Spec} (\Cbx) = V \cup \{\star\} $ simply fixes $ \{\star\} $ and is the standard $ \big( \C^\times \times \fS_3 \big) $-action on $ V $. In the next result we describe the structure of these fibres $ ( \mathbb{L}_{ \fg_{{}_\Cbx}})_{\bsigma} \cong \fgs $ for all $ \bsigma \in V $. The outcome is that in the ``regular'' locus $ V^\times $ they are simple (as Lie superalgebras), while in the ``singular'' locus $ V \setminus V^\times $ they are not, and we can describe explicitly their structure.
\end{free text}

\begin{thm} \label{thm_g-spec} Let $ \bsigma \in V $ as above, and set $ \fa_i :=\C X_{2\vep_i} \oplus \C H_{2\vep_i} \oplus \C X_{-2\vep_i} $ with $ X_{2\vep_i} \, $, $ H_{2\vep_i} $ and $ X_{-2\vep_i} $ as defined in Section~{\rm \ref{first-basis}}, for all $ i = 1, 2, 3 $.
\begin{enumerate}\itemsep=0pt
\item[$(1)$] If $ \bsigma \in V^\times $, then the Lie superalgebra $ \fgs $ is simple.
\item[$(2)$] If $ \bsigma \in V \setminus V^\times $, with $ \sigma_i = 0 $ and $ \sigma_j \not= 0 \not= \sigma_k $ for $ \{i,j,k\} = \{1,2,3\} $, then $ \fa_i $ is a {\it central} Lie ideal of $ \fgs $, isomorphic to $ \C^{ 3|0} $, and $ \fgs $ is the universal central extension of $ \mathfrak{psl}(2|2) $ by $ \fa_i $ $($cf.\ {\rm \cite[Theorem 4.7]{IK})}; in other words, there exists a short exact sequence of Lie superalgebras
\begin{gather*} 0 \relbar\joinrel\relbar\joinrel\relbar\joinrel\relbar\joinrel\longrightarrow \C^{ 3|0} \cong \fa_i \relbar\joinrel\relbar\joinrel\relbar\joinrel\relbar\joinrel\longrightarrow \fgs \relbar\joinrel\relbar\joinrel\relbar\joinrel\relbar\joinrel\longrightarrow \mathfrak{psl}(2|2) \relbar\joinrel\relbar\joinrel\relbar\joinrel\relbar\joinrel\longrightarrow 0.
\end{gather*}
A parallel result also holds true when working with $ \fgzs $ over the ground ring $ \Zbsigma $.
\item[$(3)$] If $ \bsigma = \boldsymbol{0}$ $(\! \in V \setminus V^\times ) $, i.e., $ \sigma_h = 0 $ for all $ h \in \{1,2,3\} $, then $ {\fg(\boldsymbol{0})}_\zero \cong \C^{ 9|0} $ is the center of $ \fg(\boldsymbol{0}) $, and the quotient $ \fg(\boldsymbol{0}) \Big/ {\fg(\boldsymbol{0})}_\zero \cong \C^{ 0|8} $ is Abelian; in particular, $ \fg(\boldsymbol{0}) $ is a {\it non-trivial}, {\it non-Abelian} central extension of $ \C^{ 0|8} $ by $ \C^{ 9|0} $, i.e., there exists a short exact sequence of Lie superalgebras, with {\it non-Abelian} middle term,
\begin{gather*} 0 \relbar\joinrel\relbar\joinrel\relbar\joinrel\relbar\joinrel\longrightarrow \C^{ 9|0} \cong {\fg(\boldsymbol{0})}_\zero \relbar\joinrel\relbar\joinrel\relbar\joinrel\relbar\joinrel\longrightarrow \fg(\boldsymbol{0}) \relbar\joinrel\relbar\joinrel\relbar\joinrel\relbar\joinrel\longrightarrow \C^{ 0|8} \relbar\joinrel\relbar\joinrel\relbar\joinrel\relbar\joinrel\longrightarrow 0.
\end{gather*}
 A parallel result holds true when working with $ \fg_{\raisebox{-.6mm}{\hbox{$ \scriptscriptstyle \Z $}}}(\boldsymbol{0}) $ over the ground ring $ \Z[\boldsymbol{0}] = \Z $.
\end{enumerate}
\end{thm}

\begin{proof} Part (1) is a direct consequence of \eqref{gs=g_sing} and Proposition \ref{prop: properties-g_sigma}.
 Claim~(2) instead follows at once by direct inspection of the formulas in Section~\ref{first-basis}. For instance, reading the lines in the first and second line in the table of formulas for Lie brackets therein, we see that $ \fa_i $ is Abelian when $ \sigma_i = 0 $, so that $ \fa_i \cong \C^3 $ as claimed. Moreover, the formulas from the third to the seventh in the same table tell us also that all brackets of the generators of $ \fa_i $ with all other generators turn to zero when $ \sigma_i = 0 $: thus $ \fa_i $ is central in $ \fgs $~-- hence, in particular, it is a Lie ideal~-- as claimed. Similarly, a direct verification (setting $ \sigma_i = 0 $ in those formulas) shows that the quotient $ \fgs / \fa_i $ is isomorphic to $ \mathfrak{psl}(2|2) $.

 Claim~(3) follows again by straightforward analysis of the table of formulas in Section~\ref{first-basis}. Indeed, the first two lines in the table describes the structure of the Lie algebra $ {\fgs}_\zero $: when $ \bsigma = \boldsymbol{0} $ they simply tell that this structure is trivial, that is $ \fg(\boldsymbol{0}) $ is Abelian, hence isomorphic to~$ \C^9 $ with trivial Lie bracket. Moreover, the lines from third to seventh in the same table describe the adjoint action of $ \fgs_\zero $ onto $ \fgs_\one $: when $ \bsigma = \boldsymbol{0} $, all the Lie brackets therein turn to zero, so the $ {\fg(\boldsymbol{0})}_\zero $-action is trivial, which means exactly that $ {\fg(\boldsymbol{0})}_\zero $ is central. Finally, the lines from eighth to eleventh describe the Lie brackets among linear generators of $ \fgs_\one $: all these brackets are independent of $ \bsigma $, and prove that none of these generators is central. Therefore we can conclude that the center of $ \fg(\boldsymbol{0}) $ is exactly $ {\fg(\boldsymbol{0})}_\zero $.

 As $ {\fg(\boldsymbol{0})}_\zero $ is the center of $ \fg(\boldsymbol{0}) $, the space $ \fg(\boldsymbol{0}) / {\fg(\boldsymbol{0})}_\zero $ bears the quotient Lie superalgebra structure, entirely odd. Since $ \fg(\boldsymbol{0}) = {\fg(\boldsymbol{0})}_\zero \oplus {\fg(\boldsymbol{0})}_\one $, this Lie structure is automatically trivial (no need of looking at formulas whatsoever \dots) because $ [ {\fg(\boldsymbol{0})}_\zero, {\fg(\boldsymbol{0})}_\zero ] \subseteq {\fg(\boldsymbol{0})}_\zero $. Therefore $ \fg(\boldsymbol{0}) / {\fg(\boldsymbol{0})}_\zero $ is Abelian, and isomorphic to $ \C^{ 0|8} $ because $\dim_\C ( {\fg(\boldsymbol{0})}_\one ) = 8 $.

 Finally, we stress the point that $ \fg(\boldsymbol{0}) $ has {\it non-trivial} structure, as the lines from eighth to eleventh (in the table) display non-zero Lie brackets among some of its generators.
\end{proof}

\subsection[Second family: the Lie superalgebras $ \fgps $]{Second family: the Lie superalgebras $\boldsymbol{\fgps}$}\label{sect_g'-sigma}

\begin{free text} \label{subsect-g'(sigma)}
 {\bf Construction of the $ \fgps $'s.} For any $ \bsigma \in V $ (cf.\ Section~\ref{subsect_3.1}), we define the Lie superalgebra $ \fgpzs $ over $ \Zbsigma $ as follows. As a $ \Zbsigma $-module, $ \fgpzs $ is spanned by the formal set of $ \Zbsigma $-linear generators $ B'_\fg := \{ H'_{2\vep_1}, H'_{2\vep_2}, H'_{2\vep_3}, H'_\theta \} \bigcup \{ X'_\alpha \}_{\alpha \in \Delta} $ subject only to the single relation $ \sigma_1 H'_{2\vep_1} + \sigma_2 H'_{2\vep_2} + \sigma_3 H'_{2\vep_3} = 2 H'_\theta $. The Lie superalgebra structure of $ \fgpzs $ is defined by the formulas in Section~\ref{second-basis}, which now we read as definitions of the Lie brackets among the (linear) generators of $ \fgpzs $ itself.

 Altogether, the $ \fgpzs $'s form a family of Lie superalgebras indexed by the points of $ V $. Then taking $ \fgps := \C \otimes_{\raisebox{-.7mm}{\hbox{$ \scriptstyle \Zbsigma $}}} \fgpzs $ for all $ \bsigma \in V $ we find a more regular situation, as these Lie superalgebras $ \fgps $ now share $ \C $ as their common ground ring. In fact, this is just another manner of presenting the family of Kaplansky's Lie superalgebras $ \fg_\bsigma $, in that
\begin{gather} \label{gps=g_sing}
 \text{\it for all $ \bsigma_i \in V $, we have $ \fgps \cong \fg_{\bsigma} $}
\end{gather}
 by Sections~\ref{kapl-real} and~\ref{second-basis}~-- cf.\ \eqref{2nd-basis-g}~-- and the very construction of $ \fgps $.
 In fact, the analysis in Section~\ref{second-basis} (describing a $ \C $-basis of $ \fg_\bsigma $, for any $ \bsigma \in V $, and its multiplication table) prove that {\it for all $ \bsigma \in V $, the Lie $ \Zbsigma $-superalgebra $ \fgpzs $ identifies to an integral $ \Zbsigma $-form of $ \fg_\bsigma $}.
 We can formalize the description of the family $ \{ \fgps \}_{\bsigma \in V} $ proceeding like in Section~\ref{subsect-g(sigma)}; in particular we keep the same notation, such as $ \Zbx := \Z[V] \cong \Z[x_1,x_2,x_3] / ( x_1 + x_2 + x_ 3 ) $, etc.
 In the construction of $ \fgpzs $, formally replace $ \bx $ to $ \bsigma $: this yields a meaningful definition of a Lie superalgebra over $ \Zbx $, denoted by $ \fgpzx $, and also $ \fgpx := \Cbx \otimes_{\raisebox{-.7mm}{\hbox{$ \scriptstyle \Zbx $}}} \fgpzx $ by scalar extension. Now definitions imply that, for any $ \bsigma \in V $, we have a Lie $ \Zbsigma $-superalgebra isomorphism
\begin{gather*} \fgpzs \cong \Zbsigma \mathop{\otimes}\limits_\Zbx \fgpzx \end{gather*}
-- through the ring isomorphism $ \Zbsigma \cong \Zbx / ( x_i-\sigma_i )_{i=1,2,3} $~-- and similarly
\begin{gather*} \fgps \cong \C \mathop{\otimes}\limits_{\Cbx} \fgpx \end{gather*}
 as Lie $ \C $-superalgebras, through the ring isomorphism $ \C \cong \Cbx / ( x_i-\sigma_i )_{i=1,2,3} $.

 One can argue similarly as in Section~\ref{subsect-g(sigma)} to have a geometric picture of the above description: this amounts to literally replacing $ \fgs $ with $ \fgps $, which eventually provide a coherent sheaf $ \cL_{ \fg'_{{}_\Cbx}} $ of complex Lie superalgebras over $ V $, with a $ ( \C^\times \times \fS_3 ) $-action on it, and a corresponding fibre bundle $ \mathbb{L}_{ \fg'_{{}_\Cbx}} $, whose fibres are the $ \fgps $'s; details are left to the reader. In the next result we describe these Lie superalgebras $ \fgps$ $({\cong} \fg_\bsigma) $, for all $ \bsigma \in V $: like for the $ \fgs $'s, the outcome is again that in the ``regular'' locus $ V^\times $ they are simple, while in the ``singular'' locus $ V \setminus V^\times $ we can describe explicitly their non-simple structure.
\end{free text}

\begin{thm} \label{thm_g'-spec} Given $ \bsigma \in V $, let $ \fa'_i := \C X'_{2\vep_i} \oplus \C H'_{2\vep_i} \oplus \C X'_{-2\vep_i} $, for all $ i \in \{1,2,3\} $.
\begin{enumerate}\itemsep=0pt
\item[$(1)$] If $ \bsigma \in V^\times $, then the Lie superalgebra $ \fgps $ is simple.
\item[$(2)$] If $ \bsigma \in V \setminus V^\times $, with $ \sigma_i = 0 $ and $ \sigma_j \not= 0 \not= \sigma_k $ for $ \{i,j,k\} = \{1,2,3\} $, then if
\begin{gather*} \mathfrak{b}'_i :=
 \bigg( \sum_{\alpha \not= \pm 2\vep_i} \C X'_\alpha \bigg) \oplus \bigg( \sum_{j \not= i} \C H'_{2\vep_j} \bigg)
\end{gather*}
we have $ \mathfrak{b}'_i \trianglelefteq \fgps $ $($a Lie ideal$)$, $ \fa'_i \leq \fgps $ $($a Lie subsuperalgebra$)$, and there exist isomorphisms $ \mathfrak{b}'_i \cong \mathfrak{psl}(2|2) $, $ \fa'_i \cong \fsl(2) $ and $ \fgps \cong \fsl(2) \ltimes \mathfrak{psl}(2|2) $~-- a semidirect product of Lie superalgebras. In short, there exists a {\it split} short exact sequence
\begin{gather*} 0 \relbar\joinrel\relbar\joinrel\relbar\joinrel\longrightarrow \mathfrak{psl}(2|2) \cong \mathfrak{b}'_i \relbar\joinrel\relbar\joinrel\relbar\joinrel\longrightarrow \fgps {\buildrel
{{\displaystyle \dashleftarrow \hskip-4pt \text{$-$} \text{$-$} \text{$-$}}} \over {\relbar\joinrel\relbar\joinrel\relbar\joinrel\relbar\joinrel\longrightarrow}} \fa'_i \cong \fsl(2) \relbar\joinrel\relbar\joinrel\relbar\joinrel\longrightarrow 0.
\end{gather*}
 A parallel result also holds true when dealing with $ \fgpzs $ over the ground ring $ \Zbsigma $.
\item[$(3)$] If $ \bsigma = \boldsymbol{0}$ $(\! \in V \setminus V^\times ) $, i.e., $ \sigma_h = 0 $ for all $ h \in \{1,2,3\} $, then $ {\fg'(\boldsymbol{0})}_\zero \cong {\fsl(2)}^{\oplus 3} $ as Lie $($super$)$algebras, the Lie $($super$)$bracket is trivial on $ {\fg'(\boldsymbol{0})}_\one $ and $ {\fg'(\boldsymbol{0})}_\one \cong \Box^{ \boxtimes 3} $ as modules over $ {\fg'(\boldsymbol{0})}_\zero \cong {\fsl(2)}^{\oplus 3} $. Finally, we have
\begin{gather*} \fg'(\boldsymbol{0}) \cong {\fg'(\boldsymbol{0})}_\zero \ltimes {\fg'(\boldsymbol{0})}_\one \cong {\fsl(2)}^{\oplus 3} \ltimes \Box^{ \boxtimes 3}
\end{gather*}
-- a semidirect product of Lie superalgebras. In other words, there is a {\it split} short exact sequence
\begin{gather*} 0 \relbar\joinrel\relbar\joinrel\relbar\joinrel\relbar\joinrel\longrightarrow \Box^{ \boxtimes 3} \cong {\fg'(\boldsymbol{0})}_\one \relbar\joinrel\relbar\joinrel\relbar\joinrel\relbar\joinrel\longrightarrow \fg'(\boldsymbol{0}) {\buildrel
{{\displaystyle \dashleftarrow \hskip-4pt \text{$-$} \text{$-$} \text{$-$}}} \over {\relbar\joinrel\relbar\joinrel\relbar\joinrel\relbar\joinrel\longrightarrow}} {\fg'(\boldsymbol{0})}_\zero \cong {\fsl(2)}^{\oplus 3} \relbar\joinrel\relbar\joinrel\relbar\joinrel\relbar\joinrel\longrightarrow 0.
\end{gather*}
 A parallel result holds true when working with $ \fg_{\raisebox{-.6mm}{\hbox{$ \scriptscriptstyle \Z $}}}(\boldsymbol{0}) $ over the ground ring $ \Z[\boldsymbol{0}] = \Z $.
\end{enumerate}
\end{thm}

\begin{proof} Part (1) is a direct consequence of \eqref{gps=g_sing} and Proposition~\ref{prop: properties-g_sigma}. Claim~(2) instead follows easily from direct inspection of the formulas in Section~\ref{second-basis}. Indeed, $ \fa'_i \cong \fsl(2) $ follows from the first two lines of the table of formulas for Lie brackets in Section~\ref{second-basis}, which also show that $ \fa'_i $ is a Lie subsuperalgebra of $ \fgps $~-- for all $ \bsigma \in V $, indeed.

 Similarly, those formulas show that $ \fb'_i $ is stable by the adjoint $ \fa'_i $-action exactly if and only if $ \sigma_i = 0 $: in fact, the critical point is that $ H'_\theta = 2^{-1} \sum\limits_{k=1}^3 \sigma_k H'_{2\vep_k} \in \fb'_i \iff \sigma_i = 0 $, and also $ [ X'_{2\vep_i}, H'_\theta ] = \mp \sigma_i X'_{2\vep_i} \in \fb'_i \iff \sigma_i = 0 $.

 Moreover, the same formulas altogether show also that $ \fb'_i $ is a Lie subsuperalgebra if and only if $ \sigma_i = 0 $~-- for instance, because $ \big[ X'_{\beta_j}, X'_{\beta_k} \big] = (1 - \delta_{j,k}) \sigma_i X'_{2\vep_i} \in \fb'_i \iff \sigma_i = 0 $. Then, when $ \sigma_i = 0 $, looking closely at the specific form of the Lie (sub)superalgebra $ \fb'_i $ these formulas also show that $ \fb'_i \cong \mathfrak{psl}(2|2) $ as claimed.

 Finally, as $ \fa'_i $ is a direct sum complement of $ \fb'_i $ in $ \fgps $, since $ \fb'_i $ is a Lie subsuperalgebra and is $ \fa'_i $-stable (when $ \sigma_i = 0 $) we conclude that it is also a Lie ideal, q.e.d.

 Claim~(3) follows again by straightforward analysis of the table of formulas in Section~\ref{second-basis}~-- just a matter of sheer bookkeeping.
\end{proof}

\subsection[Third family: the Lie superalgebras $ \fgss $]{Third family: the Lie superalgebras $\boldsymbol{\fgss}$}\label{sect-g''-sigma}

\begin{free text} \label{subsect-g''_Z[ts]} {\bf Construction of the $\boldsymbol{\fgss}$'s.} For any $ \bsigma \in V $ (cf.\ Section~\ref{subsect_3.1}), we define the Lie superalgebra $ \fgszs $ over $ \Zbsigma $ as follows. By definition, $ \fgszs $ is the $ \Zbsigma $-module spanned by the set of formal (linear) generators $ B_{\fg''} := \{ H'_{2\vep_1}, H'_{2\vep_2}, H'_{2\vep_3}, H'_\theta \} \bigcup \{ X_\alpha \}_{\alpha \in \Delta} $ subject only to the single relation $ \sigma_1 H'_{2\vep_1} + \sigma_2 H'_{2\vep_2} + \sigma_3 H'_{2\vep_3} = 2 H'_\theta $. The Lie superalgebra structure of $ \fgszs $ is defined by the formulas in Section~\ref{third-basis}, which now must be read as definitions for Lie brackets among the (linear) generators of $ \fgszs $.

 The $ \fgszs $'s altogether form a family of Lie superalgebras indexed by the points of $ V $. Setting $ \fgss := \C \otimes_{\raisebox{-.7mm}{\hbox{$ \scriptstyle \Zbsigma $}}} \fgszs $ for all $ \bsigma \in V $ we find a more regular situation, in that these Lie superalgebras $ \fgss $ share $ \C $ as their common ground ring. In addition
\begin{gather} \label{gss=g_sing}
 \text{\it in the {\it non-singular case}, i.e., for $ \bsigma_i \in V^\times $, we have $ \fgss \cong \fg_{\bsigma} $}
\end{gather}
 by Section~\ref{third-basis} and the very definition of $ \fgss $ itself. Indeed, from Section~\ref{third-basis}~-- where a~$ \C $-basis and its multiplication table for $ \fg_\bsigma $, with $ \bsigma \in V^\times $, are described~-- we see that {\it for all $ \bsigma \in V^\times $, the Lie superalgebra $ \fgszs $ over $ \Zbsigma $ identifies to an integral $ \Zbsigma $-form of~$ \fg_\bsigma $}.

Keeping notation as before, we can describe the family $\{ \fgss \}_{\bsigma \in V} $ in a formal way, taking its ``version over $ \Zbx $'', denoted by $ \fgszx $~-- just replacing the complex parameters $ (\sigma_1, \sigma_2, \sigma_3) =: \bsigma $ with a triple of formal parameters $ (x_1,x_2,x_3) =: \bx $ adding to zero~-- and its complex-based counterpart $ \fgsx := \Cbx \otimes_{\raisebox{-.7mm}{\hbox{$ \scriptstyle \Zbx $}}} \fgszx $. Then the very construction implies that, for any $ \bsigma \in V $, one has a Lie $ \Zbsigma $-superalgebra isomorphism
\begin{gather*} \fgszs \cong \Zbsigma \mathop{\otimes}\limits_\Zbx \fgszx \end{gather*}
-- through $ \Zbsigma \cong \Zbx / ( x_i-\sigma_i )_{i=1,2,3} $~-- and similarly
\begin{gather*} \fgss \cong \C \mathop{\otimes}\limits_{\Cbx} \fgsx \end{gather*}
as Lie $ \C $-superalgebras, through $ \C \cong \Cbx / ( x_i-\sigma_i )_{i=1,2,3} $. Finally, the reader can easily mimick what is done in Section~\ref{subsect-g(sigma)} and find a geometric description of the family of the $ \fgss $'s.

Like in Section~\ref{subsect-g(sigma)}, we can re-cast all this in geometrical terms, defining a coherent sheaf~$ \cL_{\fg''_{{}_\Cbx}} $ of complex Lie superalgebras over $ V $, with a $ ( \C^\times \times \fS_3 ) $-action on it, and a corresponding fibre bundle $ \mathbb{L}_{ \fg''_{{}_\Cbx}} $ with the $ \fgss $'s as fibres; the reader can easily fill in the details.

 In the next result we describe these Lie superalgebras $ \fgss $, for all $ \bsigma \in V $: like for the previous two families, the outcome is that in the ``regular'' locus $ V^\times $ they are simple, while in the ``singular'' locus $ V \setminus V^\times $ we can describe explicitly their non-simple structure.
\end{free text}

\begin{thm} \label{thm_g''-spec} Let $ \bsigma \in V $, and set $ \fc''_i := ( \C X_{2\vep_i} \oplus \C X_{-2\vep_i} ) $, $ \overline{\fb''_i} := \fgs / \fc''_i $~-- a subsuperspace and a quotient superspace of $ \fgss $, in general~-- for all $ i \in \{1,2,3\} $.
\begin{enumerate}\itemsep=0pt
\item[$(1)$] If $ \bsigma \in V^\times $, then the Lie superalgebra $ \fgss $ is simple.
\item[$(2)$] If $ \bsigma \in V \setminus V^\times $, with $ \sigma_i = 0 $ and $ \sigma_j \not= 0 \not= \sigma_k $ for $ \{i,j,k\} = \{1,2,3\} $, then $ \fc''_i := ( \C X_{2\vep_i} \oplus \C X_{-2\vep_i} ) $ is an Abelian Lie ideal of $ \fgss $, hence $ \overline{\fb''_i} $ is a quotient Lie superalgebra of~$ \fgss $; therefore, there exists a short exact sequence
\begin{gather*} 0 \relbar\joinrel\relbar\joinrel\relbar\joinrel\longrightarrow \fc''_i := \big( \C X_{2\vep_i} \oplus \C X_{-2\vep_i} \big) \relbar\joinrel\relbar\joinrel\relbar\joinrel\longrightarrow \fgss \relbar\joinrel\relbar\joinrel\relbar\joinrel\longrightarrow \overline{\fb''_i} \relbar\joinrel\relbar\joinrel\relbar\joinrel\longrightarrow 0.
\end{gather*}
Furthermore, setting $ \overline{\fd''_i} := \big( \bigoplus_{j \not= i} \big( \C \overline{H'_{2\vep_j}} \oplus \C \overline{X_{+2\vep_j}} \oplus \C \overline{X_{-2\vep_j}} \big) \big) \bigoplus \big( \bigoplus_{\gamma \in \Delta_\one} \C \overline{X_\gamma} \big) $~-- in the quotient Lie superalgebra $ \overline{\fb''_i} $~-- we have that
$ \overline{\fd''_i} $ is a Lie ideal and $ \C \overline{H'}_{2\vep_i} $ is a Lie subsuperalgebra of $ \overline{\fb''_i} $, with $ \overline{\fd''_i} \cong {\mathfrak{psl}(2|2)} $, $ \C \overline{H'}_{2\vep_i} \cong \C $, and $ \overline{\fb''_i} \cong ( \C \overline{H'}_{2\vep_i} ) \ltimes \overline{\fd''_i} $~-- a~semidirect product of Lie superalgebras. In short, there exists a {\it split} short exact sequence
\begin{gather*} 0 \relbar\joinrel\relbar\joinrel\relbar\joinrel\relbar\joinrel\longrightarrow {\mathfrak{psl}(2|2)} \cong \overline{\fd''_i} \relbar\joinrel\relbar\joinrel\relbar\joinrel\relbar\joinrel\longrightarrow \overline{\fb''_i} {\buildrel
{{\displaystyle \dashleftarrow \hskip-4pt \text{$-$} \hskip-0pt \text{$-$} \hskip0pt \text{$-$}}} \over {\relbar\joinrel\relbar\joinrel\relbar\joinrel\relbar\joinrel\longrightarrow}} \C \overline{H'}_{2\vep_i} \cong \C \relbar\joinrel\relbar\joinrel\relbar\joinrel\relbar\joinrel\longrightarrow 0.
\end{gather*}
A parallel result holds true when working with $ \fgszs $ over the ground ring $ \Zbsigma $.
\item[$(3)$] If $ \bsigma = \boldsymbol{0} ( \!\in V \setminus V^\times ) $, i.e., $ \sigma_h = 0 $ for all $ h \in \{1,2,3\} $, then $ \fc'' := \bigoplus\limits_{i=1}^3 \fc''_i $ is an Abelian Lie ideal, hence $ \overline{\fb''} := \fgss / \fc'' $ is a quotient Lie superalgebra $($of $ \fgss )$; therefore, there exists a short exact sequence
\begin{gather*} 0 \relbar\joinrel\relbar\joinrel\relbar\joinrel\relbar\joinrel\longrightarrow \fc'' := {\textstyle \bigoplus_{i=1}^3} \fc''_i \relbar\joinrel\relbar\joinrel\relbar\joinrel\relbar\joinrel\longrightarrow \fg''(\boldsymbol{0}) \relbar\joinrel\relbar\joinrel\relbar\joinrel\relbar\joinrel\longrightarrow \overline{\fb''} \relbar\joinrel\relbar\joinrel\relbar\joinrel\relbar\joinrel\longrightarrow 0.
\end{gather*}
Moreover, there exists a second, {\it split} short exact sequence
\begin{gather*} 0 \relbar\joinrel\relbar\joinrel\relbar\joinrel\relbar\joinrel\longrightarrow {\textstyle \bigoplus_{\alpha \in \Delta_\one}} \C \overline{X}_\alpha \relbar\joinrel\relbar\joinrel\relbar\joinrel\relbar\joinrel\longrightarrow \overline{\fb''} {\buildrel
{{\displaystyle \dashleftarrow \hskip-4pt \text{$-$} \text{$-$} \text{$-$}}} \over {\relbar\joinrel\relbar\joinrel\relbar\joinrel\relbar\joinrel\longrightarrow}} {\textstyle \bigoplus_{i=1}^3} \C \overline{H'}_{2\vep_i} \relbar\joinrel\relbar\joinrel\relbar\joinrel\relbar\joinrel\longrightarrow 0,
\end{gather*}
 where $ {\textstyle \bigoplus_{\alpha \in \Delta_\one}} \C \overline{X}_\alpha \trianglelefteq \overline{\fb''} $ is an {\it Abelian} Lie ideal and $ {\textstyle \bigoplus_{i=1}^3} \C \overline{H'}_{2\vep_i} \leq \overline{\fb''} $ is an {\it Abelian} Lie subsuperalgebra $($of $ \overline{\fb''})$, so that $ \overline{\fb''} \cong \big( {\textstyle \bigoplus_{i=1}^3} \C \overline{H'}_{2\vep_i} \big) \ltimes \big( {\textstyle \bigoplus_{\alpha \in \Delta_\one}} \C \overline{X}_\alpha \big) $~-- a semidirect product of Lie superalgebras.

 A parallel result holds true when working with $ \fg''(\boldsymbol{0}) $ over the ground ring $ \Z[\boldsymbol{0}] = \Z $.
\end{enumerate}
\end{thm}

\begin{proof}Claim (1) is a direct consequence of \eqref{gss=g_sing} along with Proposition~\ref{prop: properties-g_sigma}.
Claims~(2) and~(3), like for the previous, parallel results, both follow as direct outcome of the formulas for Lie brackets among linear generators of $ \fgss $, which we read from Section~\ref{third-basis}.

 Indeed, for claim~(2) we notice that $ [ X_{2\vep_i}, X_{-2\vep_j} ] = \delta_{i,j} \sigma_i^{ 2} H'_{2\vep_i} $ is zero when $ \sigma_i = 0 $, so that $ \fc''_i $ is then an Abelian Lie subsuperalgebra. Moreover, $ \fc''_i $ is stable for the adjoint action by elements of $ \{ H_{2\vep_1}, H_{2\vep_2}, H_{2\vep_3}, H_\theta \} $ by construction (this holds true for every $ \bsigma \in V $ indeed). Finally, $ \fc''_i $ is also stable for the adjoint action by odd root vectors such as $ X_{\pm\beta_j} $ and $ X_{\pm\theta} $ because the formulas in fifth, sixth and seventh line of the table in Section~\ref{third-basis} all give zero Lie brackets when $ \sigma_i = 0 $. Overall, this means that $ \fc''_i $ is an Abelian Lie ideal of $ \fgss $.

 The claim about $ \overline{\fb''} := \fgss / \fc'' $ and the short exact sequence now are obvious consequences of $ \fc''_i $ being a Lie ideal.

 As to the rest of claim~(1), one sees again that everything follows from straightforward bookkeeping, nothing more.
For claim~(3) one has again to carry out a similar analysis. For instance, $ \fc'' := \bigoplus\limits_{i=1}^3 \fc''_i $ is an Abelian Lie ideal because of claim {\it (2)} and the fact that the $ \fc''_i $'s commute with each other.

 Something less obvious only occurs with the analysis of $ {\textstyle \bigoplus_{\alpha \in \Delta_\one}} \C \overline{X}_\alpha $. Indeed, the fact that this is an {\it Abelian} Lie ideal of $ \overline{\fb''} $ follows once more from the formulas in Section~\ref{third-basis} {\it but} one also has to pay attention to some detail.

 Indeed, a first bunch of Lie brackets among odd root vectors which are non-zero in $ \fg(\boldsymbol{0}) $ but are trivial instead in its subquotient $ {\textstyle \bigoplus_{\alpha \in \Delta_\one}} \C \overline{X}_\alpha $ are the following:
\begin{gather*}
 [ \overline{X_{\beta_i}}, \overline{X_{\beta_j}} ] = (1-\delta_{i,j}) \overline{X_{2\vep_k}} = -\overline{0}, \qquad
 [ \overline{X_{-\beta_i}}, \overline{X_{-\beta_j}} ] = -(1-\delta_{i,j}) \overline{X_{-2\vep_k}} = -\overline{0},\\
 [ \overline{X_{\beta_i}}, \overline{X_{-\theta}} ] = \overline{X_{-2\vep_i}} = -\overline{0}, \qquad
 [ \overline{X_{-\beta_i}}, \overline{X_{\theta}} ] = -\overline{X_{2\vep_i}} = -\overline{0}.
 \end{gather*}
 Second, the remaining non-obvious Lie brackets are described by the two formulas
\begin{gather*}
 [ X_{\beta_i}, X_{-\beta_j} ] = \delta_{i,j} ( \sigma_i H'_{2\vep_i} - H'_\theta ), \qquad [ X_{\theta}, X_{-\theta} ] = H'_\theta
 \end{gather*}
but the relation $ \sigma_1 H'_{2\vep_1} + \sigma_2 H'_{2\vep_2} + \sigma_3 H'_{2\vep_3} = 2 H'_\theta $ in $ \fgss $ reads now $ H'_\theta = 0 $ in $ \fg''(\boldsymbol{0}) $, hence from the last formulas above we get
 $[ \overline{X_{\beta_i}}, \overline{X_{-\beta_j}} ] = \overline{0} $ and $[ \overline{X_{\theta}}, \overline{X_{-\theta}} ] = \overline{0} $ in $ \overline{\fb''} $.

 All other parts of claim {\it (3)} follow equally from a similar analysis~-- a sheer matter of bookkeeping~-- so we leave them to the reader.
\end{proof}

\subsection[Degenerations from contractions: the $ \widehat{\fg}(\bsigma) $'s and the $ \widehat{\fg}^{ \prime}(\bsigma) $'s]{Degenerations from contractions: the $\boldsymbol{\widehat{\fg}(\bsigma)}$'s and the $ \boldsymbol{\widehat{\fg}^{ \prime}(\bsigma)}$'s} \label{contractions}

 We finish our study of remarkable integral forms of $ \fg_\bsigma $ by introducing some further ones, that all are obtained through a general construction; when specializing these forms, one obtains again degenerations, now of the kind that is often referred to as ``contraction'' (see, e.g.,~\cite{dr}).
We start with a very general construction. Let $ R $ be a (commutative, unital) ring, and let~$ \mathcal{A} $ be an ``algebra'' (not necessarily associative, nor unitary), in some category of $ R $-bimodules, for some ``product'' denoted by ``$ \cdot $'': we assume in addition that
\begin{gather} \label{props_A=F+C}
 \mathcal{A} = \mathcal{F} \oplus \mathcal{C} \qquad \text{with} \qquad \mathcal{F} \cdot \mathcal{F} \subseteq \mathcal{F}, \qquad \mathcal{F} \cdot \mathcal{C} \subseteq \mathcal{C}, \qquad \mathcal{C} \cdot \mathcal{F} \subseteq \mathcal{C}, \qquad \mathcal{C} \cdot \mathcal{C} \subseteq \mathcal{F}
\end{gather}
Choose now $ \tau $ be a non-unit in $ R $, and correspondingly consider in $ \mathcal{A} $ the $ R $-submodules
\begin{gather} \label{def-A_eta,+}
 \mathcal{F}_\tau := \mathcal{F}, \qquad
 \mathcal{C}_{ \tau} := \tau \mathcal{C}, \qquad
 \mathcal{A}_{ \tau} := \mathcal{F}_\tau + \mathcal{C}_{ \tau} = \mathcal{F} \oplus (\tau \mathcal{C}).
\end{gather}
Fix also a (strict) ideal $ I \trianglelefteq R $; then set $ R_I := R / I $ for the corresponding quotient ring, and use notation $ \mathcal{A}_{ \tau,I} := \mathcal{A}_{ \tau} / I \mathcal{A}_{ \tau} \cong ( R / I ) \otimes_R \mathcal{A}_{ \tau} $, $ \mathcal{F}_{\tau,I} := \mathcal{F}_\tau / I \mathcal{F}_\tau \cong ( R / I ) \otimes_R \mathcal{F} $ and $ \mathcal{C}_{ \tau,I} := \mathcal{C}_{ \tau} / I \mathcal{C}_{ \tau} = (\tau \mathcal{C}) / ( I \tau \mathcal{C} ) \cong ( R / I ) \otimes_R \mathcal{C}_{ \tau} \cong ( R / I ) \otimes_R (\tau \mathcal{C}) $. By construction we have $ \mathcal{A}_{ \tau,I} \cong \mathcal{F}_{\tau,I} \oplus \mathcal{C}_{ \tau,I} $ as an $ R_I $-module; moreover,
\begin{gather*} \mathcal{F}_{\tau,I} \cdot \mathcal{F}_{\tau,I} \subseteq \mathcal{F}_{\tau,I}, \qquad \mathcal{F}_{\tau,I} \cdot \mathcal{C}_{ \tau,I} \subseteq \mathcal{C}_{ \tau,I}, \qquad \mathcal{C}_{ \tau,I} \cdot \mathcal{F}_{\tau,I} \subseteq \mathcal{C}_{ \tau,I}, \qquad \mathcal{C}_{ \tau,I} \cdot \mathcal{C}_{ \tau,I} \subseteq \overline{\tau}^{ 2} \mathcal{F}_{\tau,I},
\end{gather*}
where the last identity comes from $ \mathcal{C}_{ \tau} \cdot \mathcal{C}_{ \tau} = \tau^2 ( \mathcal{C} \cdot \mathcal{C}) \subseteq \tau^2 \mathcal{F} = \tau^2 \mathcal{F}_\tau $ and we write $ \overline{\tau} := ( \tau \operatorname{mod} I ) \in R / I $. In particular, {\it if $ \tau \in I $, then $ \mathcal{C}_{ \tau,I} \cdot \mathcal{C}_{ \tau,I} = \{0\} $ and we get
\begin{gather} \label{A_eta-I semidir}
 \mathcal{A}_{ \tau,I} = \mathcal{F}_{\tau,I} \ltimes \mathcal{C}_{ \tau,I},
\end{gather}
where $ \mathcal{C}_{ \tau,I} $ bears the $ \mathcal{F}_{\tau,I} $-bimodule structure induced from $ \mathcal{A} $ and is given a trivial product, so that it sits inside $ \mathcal{A}_{ \tau,I} $ as a two-sided Abelian ideal, with \eqref{A_eta-I semidir} being a semidirect product splitting}.
 In fact, $ \mathcal{A}_{ \tau,I} $ is what is called a~``central extension of $ \mathcal{F}_{\tau,I} $ by $ \mathcal{C}_{ \tau,I} $''.

 In short, for $ \tau \in I $ this process leads us from the initial object $ \mathcal{A} $, that splits into $ \mathcal{A} = \mathcal{F} \oplus \mathcal{C} $ as $ R $-module, to the final object $ \mathcal{A}_{ \tau,I} = \mathcal{F}_{\tau,I} \ltimes \mathcal{C}_{ \tau,I} $, now split as a semidirect product.
 Following \cite[Section~2 and references therein]{dr}, we shall refer to this process as ``contraction'', and also refer to $ \mathcal{A}_{ \tau,I} $ as to a~``{\it contraction} of $ \mathcal{A} $''.
 Note, however, that these are contractions of a very special type, in that only the odd part is ``contracted'': in the general theory of contractions of Lie superalgebras, instead, one has to do with a richer variety of objects~-- cf., e.g.,~\cite{va}. For our purpose however we do not need the general theory in its full extent.
We apply now the above contraction procedure to a couple of integral forms of $ \fg_\bsigma $.
 First consider the case $ \mathcal{A} := \fgzx $, $ \mathcal{F} := {\fg_{{}_\Z}(\bx)}_\zero $ and $ \mathcal{C} := {\fg_{{}_\Z}(\bx)}_\one $; here the ground ring is $ R := \Zbx $, and we choose in it $ \tau := x_{1 } x_{2 } x_3 $ and the ideal $ I = I_\bsigma $ generated by $ x_1 - \sigma_1 $, $ x_2 - \sigma_2 $ and $ x_3 - \sigma_3 $. In this case, the ``blown-up'' Lie superalgebras in~\eqref{def-A_eta,+} reads $ {\fgzx}_\tau = {\fgzx}_\zero \oplus ( \tau {\fgzx}_\one ) $, that we write also with the simpler notation $ \widehat{\fg}_{{}_\Z}(\bx) := {\fgzx}_\tau $; similarly we write $ \widehat{\fg}_{{}_\Z}(\bsigma) := {\fgzx}_{\tau,I_\bsigma} $. Note that each $ \widehat{\fg}_{{}_\Z}(\bsigma) $ for {\it non-singular} $ \bsigma \in V^\times $ is yet another $ \Zbsigma $-integral form of our initial complex Lie superalgebra $ \fg_\bsigma $. Similarly occurs if we work over $ \C $, i.e., when we consider $ \mathcal{A} := \fgx $, $ \mathcal{F} := {\fgx}_\zero $, $ \mathcal{C} := {\fgx}_\one $ and the blown-up algebra $ {\fgx}_\tau $ with ground ring $ R := \Cbx $, and the contraction $ \widehat{\fg}(\bsigma) := {\fgx}_{\tau,I_\bsigma} $ over $ \C $. This gives (as in Section~\ref{subsect-g(sigma)}) a new coherent sheaf of complex Lie superalgebras over~$ V $, say $ \cL_{\widehat{\fg}_{{}_\Cbx}} $, with a $ \big( \C^\times \times \fS_3 \big) $-action on it, and an associated fibre bundle $ \mathbb{L}_{ \widehat{\fg}_{{}_\Cbx}} $ with the $ \widehat{\fg}(\bsigma) $'s as fibres; details are left to the reader. Next result describes the structure of the $ \widehat{\fg}(\bsigma) $'s.

\begin{thm} \label{thm_g_eta-spec} Let $ \bsigma \in V $, $ i \in \{1,2,3\} $, $ \widehat{\fa}_i :=\C X_{2\vep_i} \oplus \C H_{2\vep_i} \oplus \C X_{-2\vep_i} $ $( = \fa_i $ of Section~{\rm \ref{1st_Kac-constr})}.
\begin{enumerate}\itemsep=0pt
 \item[$(1)$] If $ \bsigma \in V^\times $, then the Lie superalgebra $ \widehat{\fg}(\bsigma) $ is simple.
\item[$(2)$] If $ \bsigma \in V \setminus V^\times $, with $ \sigma_i = 0 $ and $ \sigma_j \not= 0 \not= \sigma_k $ for $ \{i,j,k\} = \{1,2,3\} $, then
$ \widehat{\fa}_i \trianglelefteq \widehat{\fg} $ is a~central Lie ideal in $ \widehat{\fg}(\bsigma) $, with $ \widehat{\fa}_i \cong \C^{ 3|0} $, while $ \widehat{\fa}_j \cong \widehat{\fa}_k \cong \fsl(2) $ for $ \{j,k\} = \{1,2,3\} \setminus \{i\} $. Moreover, we have a semidirect product splitting
\begin{gather*} \widehat{\fg}(\bsigma) \cong {\widehat{\fg}(\bsigma)}_\zero \ltimes {\widehat{\fg}(\bsigma)}_\one \end{gather*}
with $ {\widehat{\fg}(\bsigma)}_\zero = \oplus_{\ell=1}^{ 3} \widehat{\fa}_\ell \cong \C^{ 3|0} \oplus \fsl(2) \oplus \fsl(2) $ while $ {\widehat{\fg}(\bsigma)}_\one $ is endowed with trivial Lie bracket and $ {\widehat{\fg}(\bsigma)}_\one \cong ( {\raise1,3pt\hbox{$ \sqbullet $}} {\raise1pt\hbox{$ \hskip2pt \oplus \hskip1pt $}} {\raise1,3pt\hbox{$ \sqbullet $}} ) \boxtimes \Box \boxtimes \Box $~-- where $ {\raise1,3pt\hbox{$ \sqbullet $}} $ is the trivial representation~-- as a module over $ {\widehat{\fg}(\bsigma)}_\zero \cong \C^{ 3|0} \oplus \fsl(2) \oplus \fsl(2) $; so there exists a {\it split} short exact sequence of Lie superalgebras
\begin{gather*} 0 \relbar\joinrel\longrightarrow {\widehat{\fg}(\bsigma)}_\one \cong ( {\raise1,3pt\hbox{$ \sqbullet $}} {\raise1pt\hbox{$ \hskip2pt \oplus \hskip1pt $}} {\raise1,3pt\hbox{$ \sqbullet $}} ) \boxtimes \Box \boxtimes \Box \relbar\joinrel\longrightarrow \widehat{\fg}(\bsigma) {\buildrel
{{\displaystyle \dashleftarrow \hskip-4pt \text{$-$}}} \over {\relbar\joinrel\longrightarrow}} {\widehat{\fg}(\bsigma)}_\zero \cong \C^{ 3|0} \oplus \fsl(2) \oplus \fsl(2) \relbar\joinrel\longrightarrow 0.
\end{gather*}
A parallel result also holds true when working with $ \widehat{\fg}_{{}_\Z}(\bsigma) $ over the ground ring $ \Zbsigma $.
\item[$(3)$] If $ \bsigma = \boldsymbol{0} ( \!\in V \setminus V^\times ) $, i.e., $ \sigma_h = 0 $ for all $ h \in \{1,2,3\} $, then $ \widehat{\fg}(\boldsymbol{0}) $ is the Abelian complex Lie superalgebra of superdimension $ 9|8 $, that is $ \widehat{\fg}(\boldsymbol{0}) \cong \C^{ 9|8} $ with trivial bracket.

 A parallel result holds true when working with $ \widehat{\fg}_{{}_\Z}(\boldsymbol{0}) $ over the ground ring $ \Z[\boldsymbol{0}] = \Z $.
 \end{enumerate}
\end{thm}

\begin{proof} The claim follows directly from Theorem \ref{thm_g-spec} once we take also into account the fact that the $ \widehat{\fg}(\bsigma) $'s are specializations of $ {\fgx}_\tau $, and for singular values $ \bsigma \in V^\times $ any such speciali\-zation is a {\it contraction} of $ \widehat{\fg}(\bsigma) $, of the form $ \widehat{\fg}(\bsigma) = {\fg(\bx)}_{\tau,I} $ for the element $ \tau := x_{1 } x_{2 } x_{ 3} $ and the ideal $ I := \{ {( x_i - \sigma_i \big)}_{i=1,2,3} \} $. Otherwise, one can deduce the statement directly from the explicit formulas for (linear) generators of $ \widehat{\fg}(\bx) $: indeed, the latter are easily obtained as slight modification~-- taking into account that odd generators must be ``rescaled'' by the coefficient $ \tau := x_{1 } x_{2 } x_{ 3} $~-- of the similar formulas in Section~\ref{first-basis} for $ \fgs $, which read as formulas for $ \fgx $ just switching the $ \sigma_\ell $'s into $ x_\ell $'s.
\end{proof}

As a second instance, we consider the case $ \mathcal{A} := \fgpzx $, $ \mathcal{F} := {\fgpzx}_\zero $ and $ \mathcal{C} := {\fgpzx}_\one $; the ground ring is again $ R := \Zbx $, and again we choose in it $ \tau := x_{1 } x_{2 } x_{3 } $ and the ideal $ I $ generated by $ x_1 - \sigma_1 $, $ x_2 - \sigma_2 $ and $ x_3 - \sigma_3 $. In this second case, we have again a ``blown-up'' Lie superalgebra as in~\eqref{def-A_eta,+}, that now reads $ {\fgpzx}_\tau = {\fgpzx}_\zero \oplus ( \tau {\fgpzx}_\one ) $, for which we use the simpler notation $ \widehat{\fg}^{ \prime}_{{}_\Z}(\bx) := {\fgpzx}_\tau $; similarly we write also~$ \widehat{\fg}^{ \prime}_{{}_\Z}(\bsigma) := {\fgpzx}_{\tau,I_\bsigma} $. Again, each $ \widehat{\fg}^{ \prime}_{{}_\Z}(\bsigma) $ for {\it non-singular} $ \bsigma \in V^\times $ is another $ \Zbsigma $-integral form of the complex Lie superalgebra $ \fg_\bsigma $ we started with. Similarly, working over $ \C $ we consider $ \mathcal{A} := \fgpx $, $ \mathcal{F} := {\fgpx}_\zero $, $ \mathcal{C} := {\fgpx}_\one $ and the blown-up algebra $ {\fgpx}_\tau $ with ground ring $ R := \Cbx $, and the contraction $ \widehat{\fg}^{ \prime}(\bsigma) := {\fgpx}_{\tau,I_\bsigma} $ over~$ \C $.
This provides one more coherent sheaf of complex Lie superalgebras over $ V $, denoted~$ \cL_{\widehat{\fg}^{ \prime}_{{}_\Cbx}} $, with a~$ ( \C^\times \times \fS_3 ) $-action on it, and an associated fibre bundle $ \mathbb{L}_{ \widehat{\fg}^{ \prime}_{{}_\Cbx}} $ having the $ \widehat{\fg}^{ \prime}(\bsigma) $'s as fibres (just like in Section~\ref{subsect-g(sigma)}: details are left to the reader). Next result describes the structure of these fibres $ \widehat{\fg}^{ \prime}(\bsigma) $'s:

\begin{thm} \label{thm_g'_eta-spec}
 Let $ \bsigma \in V $, and $ \widehat{\fa}^{ \prime}_i := \C X'_{2\vep_i} \oplus \C H'_{2\vep_i} \oplus \C X'_{-2\vep_i} $ for all $ i = 1, 2, 3 $.
\begin{enumerate}\itemsep=0pt
\item[$(1)$] If $ \bsigma \in V^\times $, then the Lie superalgebra $ \widehat{\fg}^{ \prime}(\bsigma) $ is simple.
\item[$(2)$] If $ \bsigma \in V \setminus V^\times $, then
 $ {\widehat{\fg}^{ \prime}(\bsigma)}_\zero \cong {\fsl(2)}^{\oplus 3} $ as Lie superalgebras, %
 $ {\widehat{\fg}^{ \prime}(\bsigma)}_\one \cong \boxtimes_{ i=1}^{ 3} \Box_i $ as modules over $ {\widehat{\fg}^{ \prime}(\bsigma)}_\zero \cong {\fsl(2)}^{\oplus 3} $ and the Lie bracket is trivial on $ {\widehat{\fg}^{ \prime}(\bsigma)}_\one $; finally, we have semidirect product splittings
\begin{gather*} \widehat{\fg}^{ \prime}(\bsigma) \cong {\widehat{\fg}^{ \prime}(\bsigma)}_\zero \ltimes {\widehat{\fg}^{ \prime}(\bsigma)}_\one \cong {\fsl(2)}^{\oplus 3} \ltimes \big( {\boxtimes}_{ i=1}^{ 3} \Box_i \big).
\end{gather*}
 In other words, there exists a {\it split} short exact sequence
\begin{gather*} 0 \relbar\joinrel\relbar\joinrel\relbar\joinrel\relbar\joinrel\longrightarrow \Box^{ \boxtimes 3} \cong {\widehat{\fg}^{ \prime}(\bsigma)}_\one \relbar\joinrel\relbar\joinrel\relbar\joinrel\relbar\joinrel\longrightarrow \widehat{\fg}^{ \prime}(\bsigma) {\buildrel
{{\displaystyle \dashleftarrow \hskip-4pt \text{$-$} \text{$-$} \text{$-$}}} \over {\relbar\joinrel\relbar\joinrel\relbar\joinrel\relbar\joinrel\longrightarrow}} {\widehat{\fg}^{ \prime}(\bsigma)}_\zero \cong {\fsl(2)}^{\oplus 3} \relbar\joinrel\relbar\joinrel\relbar\joinrel\relbar\joinrel\longrightarrow 0.
\end{gather*}
A parallel result also holds true when working with $ \widehat{\fg}^{ \prime}_{{}_\Z}(\bsigma) $ over the ground ring $ \Zbsigma $.
\end{enumerate}
\end{thm}

\begin{proof} One can easily deduce the claim from Theorem \ref{thm_g'-spec} along with the fact that each $ \widehat{\fg}^{ \prime}(\bsigma) $ is a specialization of $ {\fgpx}_\tau $, and in particular for singular values $ \bsigma \in V^\times $ any such specialization is indeed a {\it contraction} of $ \widehat{\fg}^{ \prime}(\bsigma) $, namely of the form $ \widehat{\fg}^{ \prime}(\bsigma) = {\fgpx}_{\tau,I} $ for the element $ \tau := x_{1 } x_{2 } x_{ 3} $ and the ideal $ I := \{ ( x_i - \sigma_i )_{i=1,2,3} \} $. As alternative method, one can obtain the statement by means of a direct analysis of the explicit formulas for (linear) generators of $ \widehat{\fg}^{ \prime}(\bx) $: in fact, one easily obtains such formulas as slight modifications~-- taking into account the ``rescaling'' of odd generators by the coefficient $ \tau := x_{1 } x_{2 } x_{3 } $~-- of the formulas in Section~\ref{second-basis}. We leave details to the interested reader.
\end{proof}

\begin{Remark} \label{rmk: compar-fam.s - alg} We considered five families of Lie superalgebras, denoted by $ \{ \fgs \}_{\sigma \in V} $,\linebreak $\{ \fgps\}_{\bsigma \in V} $, $\{ \fgss \}_{\bsigma \in V} $, $ \{ \widehat{\fg}(\bsigma) \}_{\bsigma \in V} $ and $ \{ \widehat{\fg}^{ \prime}(\bsigma) \}_{\bsigma \in V} $, all being indexed by the points of the complex plane~$ V $. Now, our analysis shows that these five families share most of their elements, namely all those indexed by ``general points'' $ \bsigma \in V^\times :=
 V \cap ( \C^\times )^3 $. On the other hand, the five families are drastically different at all points in the ``singular locus'' $ \cS := V \setminus V^\times = V \cap \big( \bigcup_{i=1,2,3} \{ \sigma_i = 0 \} \big) $. In other words, the five sheaves $ \mathcal{L}_{ \fg_{{}_\Cbx}} $, $ \mathcal{L}_{ \fg'_{{}_\Cbx}} $, $ \mathcal{L}_{ \fg''_{{}_\Cbx}} $, $ \mathcal{L}_{ \widehat{\fg}_{{}_\Cbx}} $ and $ \mathcal{L}_{ \widehat{\fg}^{ \prime}_{{}_\Cbx}} $ of Lie superalgebras over $ \operatorname{Spec} (\Cbx) \cong V \cup \{\star\} \big( {\cong} \mathbb{A}_\C^{ 2} \cup \{\star\} \big) $ share the same stalks on all ``general'' points (i.e., those outside $ \cS $), and have different stalks instead on ``singular'' points (i.e., those in~$ \cS $). Likewise, the five fibre bundles
 $ \mathbb{L}_{ \fg_{{}_\Cbx}} $, $ \mathbb{L}_{ \fg'_{{}_\Cbx}} $, $ \mathbb{L}_{ \fg''_{{}_\Cbx}} $, $ \mathbb{L}_{ \widehat{\fg}_{{}_\Cbx}} $ and $ \mathbb{L}_{ \widehat{\fg}^{ \prime}_{{}_\Cbx}} $ over $ \operatorname{Spec} (\Cbx) $ share the same fibres on all general points and have different fibres on singular points.
 Let us also stress that the second family $ \{ \fgps\}_{\bsigma \in V} $ is just Kaplansky's one, as $ \fgps \cong \fg_\bsigma $~-- cf.\ Section~\ref{sect_g'-sigma}.
 The outcome of the previous discussion is, loosely speaking, that {\it our construction provides five different ``completions'' of the family $ \{ \fg_\bsigma \}_{\sigma \in V \setminus \cS} $ of {\it simple} Lie superalgebras, by adding~-- in five different ways~-- some new {\it non-simple} extra elements} on top of each point of the ``singular locus'' $ \cS $. In particular, this shows that {\it it makes no sense to speak of ``taking the limit for $ \bsigma $ going to $ \cS $'' of the simple Lie superalgebras $ \fg_\bsigma $ $($for $ \bsigma \in V )$}, unless one states exactly what is the {\it total} family~-- namely, indexed over {\it all of $ V $}~-- of Lie superalgebras one has chosen to complete the family $ \{ \fg_\bsigma \}_{\sigma \in V \setminus \cS} $. Indeed, as our results show, depending on such a choice one finds very different, non-isomorphic ``limits''.
\end{Remark}

\section[Lie supergroups of type $ D(2,1;\bsigma) $: presentations and degenerations]{Lie supergroups of type $\boldsymbol{D(2,1;\bsigma)}$:\\ presentations and degenerations}\label{section5}

In this section, we introduce (complex) Lie supergroups of type $ D\big(2,1;\bsigma) $, basing on the five families of Lie superalgebras introduced in Section~\ref{forms-degen.s - Lie s-alg.s} and following the approach of Section~\ref{sHCp's->Liesgrp's}. For simplicity, we formulate everything over $ \C $, but the reader may see some subtleties to discuss about the Chevalley groups over a $ \Z[\bsigma] $-algebra. The latter had been discussed in \cite{FG} and \cite{Ga1} for some basis, i.e., for one particular choice of $ \Z[\bsigma] $-integral form (though with slightly different formalism); in the present case everything works similarly, up to paying attention to the $ \bsigma $-dependence of the commutation relations of the $ \Z[\bsigma] $-form one chooses (cf.~Section~\ref{forms-degen.s - Lie s-alg.s}). The details are left to the reader.

\subsection[First family: the Lie supergroups $ \bG_\bsigma $]{First family: the Lie supergroups $\boldsymbol{\bG_\bsigma}$} \label{subsect: LieSGroups G_s}

 Given $ \bsigma = (\sigma_1,\sigma_2,\sigma_3) \in V $, let $ \fg = \fgs $ be the complex Lie superalgebra associated with $ \bsigma $ as in Section~\ref{sect-g-sigma}, and $ \fg_\zero $ its even part. We recall that $ \fg $ is spanned over $ \C $ by $ \{ H_{2\vep_1}, H_{2\vep_2}, H_{2\vep_3}, H_\theta \} \cup {\{ X_\alpha \}}_\Delta $. Like in Section~\ref{sect-g-sigma}, we set $ \fa_i := \C X_{2\varepsilon_i} \oplus \C H_{2\varepsilon_i} \oplus \C X_{-2\varepsilon_i} $ for each $ i $~-- all these being Lie subalgebras of $ \fgs $, with $ {\fgs}_\zero = \oplus_{i=1}^3 \fa_i $. When $ \sigma_i \not= 0 $, the Lie algebra $ \fa_i $ is isomorphic to $ \fsl(2) $: an explicit isomorphism is realized by mapping $ X_{2\varepsilon_i} \mapsto \sigma_i e $, $ H_{2\varepsilon_i} \mapsto \sigma_i h $ and $ X_{-2\varepsilon_i} \mapsto \sigma_i f $, where $ \{ e, h, f \} $ is the standard basis $ \fsl(2) $. When $ \sigma_i = 0 $ instead, $ \fa_i \cong \C^{\oplus 3} $ becomes the 3-dimensional Abelian Lie algebra.
 \par
 Let us now set $ A_i := \text{SL}_2 $ if $ \bsigma_i \not= 0 $ and $ A_i := \C \times \C^\times \times \C $ if $ \bsigma_i = 0 $, and define $ G := \times_{i=1}^3 A_i $~-- a~complex Lie group such that $ \operatorname{Lie}(G) = {\fgs}_\zero $. One sees that the adjoint action of $ {\fgs}_\zero $ onto~$ \fgs $ integrates to a Lie group action of~$ G $ onto~$ \fgs $ again, so that the pair $ \cP_{\bsigma} := \big( G, \fgs \big) $~-- endowed with that action~-- is a super Harish-Chandra pair (cf.\ Section~\ref{sHCp's}); note that its dependence on $ \bsigma $ lies within all its constituents: the structure of $ G $, the Lie superalgebra $ \fgs $, and the action of the former onto the latter.

 Finally, we let
\begin{gather*} \bG_\bsigma := \bG_{{}_{\cP_{{}_{\bsigma}}}} \end{gather*}
 be the complex Lie supergroup associated with the super Harish-Chandra pair $ \cP_{\bsigma} $ trough the category equivalence given in Section~\ref{sHCp's->Liesgrp's}.

\begin{free text} \label{pres-G_s} {\bf A presentation of $\boldsymbol{\bG_\bsigma}$.} We shall now provide an explicit presentation by generators and relations for the supergroups~$ \bG_\bsigma $, i.e., for the abstract groups $ \bG_\bsigma(A) $, $ A \in \Wsalg $.

 To begin with, inside each subgroup $ A_i $ we consider the elements
\begin{gather*} x_{2\varepsilon_i}(c) := \exp (c X_{2\varepsilon_i} ), \qquad
 h_{2\varepsilon_i}(c) := \exp (c H_{2\varepsilon_i} ), \qquad
 x_{-2\varepsilon_i}(c) := \exp (c X_{-2\varepsilon_i} ) \end{gather*}
for every $ c \in \C $; then $ \varGamma_i := \{ x_{2\varepsilon_i}(c), h_{2\varepsilon_i}(c), x_{-2\varepsilon_i}(c) \}_{ c \in \C} $ is a generating set for $ A_i $.

 We define also elements $ h_\theta(c) := \exp\big(c H_\theta\big) $ for all $ c \in \C $: then the commutation relations $ [ H_{2\varepsilon_r}, H_{2\varepsilon_s} ] = 0 $ and $ H_{2\varepsilon_1} + H_{2\varepsilon_2} + H_{2\varepsilon_3} = 2 H_\theta $ inside $ \fgs $ together imply the group relations $ h_{2\varepsilon_1}(c) h_{2\varepsilon_2}(c) h_{2\varepsilon_3}(c) = {h_\theta(c)}^2 $ for all $ c \in \C $.

 The complex Lie group $ G $ is clearly generated by
\begin{gather*} \varGamma_\zero := \{ x_{2\varepsilon_i}(c), h_{2\varepsilon_i}(c), h_\theta(c), x_{-2\varepsilon_i}(c) \}^{i \in \{1,2,3\}}_{c \in \C} \end{gather*}
(the $ h_\theta(c) $'s might be dropped, but we prefer to add them too as generators).

 In addition, when we consider $ G $ as a (totally even) supergroup and we look at it as a~group-valued functor $ G \colon \Wsalg \longrightarrow \text{\rm ($ \cat{grps} $)} $, {\it the abstract group $ G(A) $ of its $ A $-points~-- for $ A \in \Wsalg $~-- is generated by the set}
\begin{gather} \label{eq: gen-set G_+}
 \varGamma_\zero(A) := \{ x_{2\varepsilon_i}(a), h_{2\varepsilon_i}(a), h_\theta(a), x_{-2\varepsilon_i}(a) \}^{i \in \{1,2,3\}}_{a \in A_\zero}.
\end{gather}

 Note that here the generators do make sense~-- as operators in $ \text{GL} ( A \otimes \fgs ) $, but also formally~-- since $ A = \C \oplus \mathfrak{N}(A) $ (cf.\ Section~\ref{subsec:basic-sobjcs}), so each $ a \in A $ reads as $ a = c + n_a $ for some $ c \in \C $ and a nilpotent $ n_a \in \mathfrak{N}(A) $, hence $ \exp(a X_{2\varepsilon_i}) = \exp(c X_{2\varepsilon_i}) \exp(n_a X_{2\varepsilon_i}) $, etc., are all well-defined.

 Following the recipe in Section~\ref{sHCp's->Liesgrp's}, in order to generate the group $ \bG_\bsigma(A) := \bG_{{}_{\cP_{{}_{\bsigma}}}}(A) $, beside the subgroup $ G(A) $ we need also all the elements of the form $ ( 1 + \eta_i Y_i ) $ with $ (i,\eta_i) \in I \times A_\one $~-- cf.\ Section~\ref{sHCp's->Liesgrp's}~-- where now the $ \C $-basis $\{ Y_i \}_{i \in I} $ of $ \fg_\one $ is $ \{ Y_i \}_{i \in I} = \{ X_{\pm\theta}, X_{\pm\beta_i} \}_{i=1,2,3} $. Therefore, we introduce notation $ x_{\pm\theta}(\eta) := ( 1 + \eta X_{\pm\theta} \big) $, $ x_{\pm\beta_i}(\eta) := ( 1 + \eta X_{\pm\beta_i} ) $ for all $ \eta \in A_\one $, $ i \in \{1,2,3\} $, and we consider the set $ \varGamma_\one(A) := \{ x_{\pm\theta}(\eta), x_{\pm\beta_i}(\eta) \,|\, \eta \in A_\one \} $.

 Now, taking into account that $ G(A) $ is generated by $ \varGamma_\zero(A) $, we can modify the set of relations given in Section~\ref{sHCp's->Liesgrp's} by letting $ g \in G(A) $ range inside the set $ \varGamma_\zero(A) $: then we can find the following {\it full set of relations} (where hereafter {\it we freely use notation $ e^Z := \exp(Z) $}):
\begin{gather*}
 1_{{}_G} = 1, \qquad g' \cdot g'' = g' \cdot_{{}_G} g'', \qquad \forall\, g', g'' \in G(A),\\
h_{2\vep_i}(a) x_{\pm\beta_j}(\eta) {h_{2\vep_i}(a)}^{-1} = x_{\pm\beta_j}\big( e^{\pm {(-1)}^{-\delta_{i,j}} \sigma_i a} \eta\big),\\
 h_{2\vep_i}(a) x_{\pm\theta}(\eta) {h_{2\vep_i}(a)}^{-1} = x_{\pm\theta}\big( e^{\pm \sigma_i a} \eta\big),\\
 h_\theta(a) x_{\pm\beta_i}(\eta) {h_\theta(a)}^{-1} = x_{\pm\beta_i}\big( e^{\mp \sigma_i a} \eta\big), \qquad h_\theta(a) x_{\pm\theta}(\eta) {h_\theta(a)}^{-1} = x_{\pm\theta}(\eta),\\
 x_{2\vep_i}(a) x_{\beta_j}(\eta) {x_{2\vep_i}(a)}^{-1} = x_{\beta_j}(\eta) x_{\theta}( \delta_{i,j} \sigma_i a \eta ),\\
 x_{2\vep_i}(a) x_{-\beta_j}(\eta) {x_{2\vep_i}(a)}^{-1} = x_{-\beta_j}(\eta) x_{\beta_k}( (1-\delta_{i,j}) \sigma_i a \eta ),\\
 x_{-2\vep_i}(a) x_{\beta_j}(\eta) {x_{-2\vep_i}(a)}^{-1} = x_{\beta_j}(\eta) x_{-\beta_k}( (1-\delta_{i,j}) \sigma_i a \eta ),\\
 x_{-2\vep_i}(a) x_{-\beta_j}(\eta) {x_{-2\vep_i}(a)}^{-1} = x_{-\beta_j}(\eta) x_{-\theta}( \delta_{i,j} \sigma_i a \eta ),\\
 x_{2\vep_i}(a) x_{\theta}(\eta) {x_{2\vep_i}(a)}^{-1} = x_{\theta}(\eta),\\
 x_{2\vep_i}(a) x_{-\theta}(\eta) {x_{2\vep_i}(a)}^{-1} = x_{-\theta}(\eta) x_{-\beta_i}( \sigma_i a \eta ),\\
 x_{-2\vep_i}(a) x_{\theta}(\eta) {x_{-2\vep_i}(a)}^{-1} = x_{\theta}(\eta) x_{\beta_i}( \sigma_i a \eta ),\\
 x_{-2\vep_i}(a) x_{-\theta}(\eta) {x_{-2\vep_i}(a)}^{-1} = x_{-\theta}(\eta) \cr
 x_{\beta_i}(\eta_i) x_{\beta_j}(\eta'_j) = x_{2\vep_k}( (1-\delta_{i,j}) \eta'_j \eta_i) x_{\beta_j}(\eta'_j) x_{\beta_i}(\eta_i),\\
 x_{-\beta_i}(\eta_i) x_{-\beta_j}(\eta'_j) = x_{-2\vep_k}( -(1-\delta_{i,j}) \eta'_j \eta_i) x_{-\beta_j}(\eta'_j) x_{-\beta_i}(\eta_i),\\
 x_{\beta_i}(\eta_i) x_{-\beta_j}(\eta'_j) = h_{2\vep_i}(\delta_{i,j} \eta'_j \eta_i) h_{\theta}(-\delta_{i,j} \eta'_j \eta_i) x_{-\beta_j}(\eta'_j) x_{\beta_i}(\eta_i),\\
 x_{\beta_i}(\eta_i) x_{\theta}(\eta) = x_{\theta}(\eta) x_{\beta_i}(\eta_i), \qquad x_{\beta_i}(\eta_i) x_{-\theta}(\eta) = x_{-2\vep_i}(\eta \eta_i) x_{-\theta}(\eta) x_{\beta_i}(\eta_i),\\
 x_{-\beta_i}(\eta_i) x_{\theta}(\eta) = x_{2\vep_i}(-\eta \eta_i) x_{\theta}(\eta) x_{-\beta_i}(\eta_i), \qquad x_{-\beta_i}(\eta_i) x_{-\theta}(\eta) = x_{-\theta}(\eta) x_{-\beta_i}(\eta_i),\\
 x_{\theta}(\eta_+) x_{-\theta}(\eta_-) = h_\theta(\eta_- \eta_+) x_{-\theta}(\eta_-) x_{\theta}(\eta_+),\\
 x_{\pm\beta_i}(\eta') x_{\pm\beta_i}(\eta'') = x_{\pm\beta_i}(\eta'+\eta''), \qquad x_{\pm\theta}(\eta') x_{\pm\theta}(\eta'') = x_{\pm\theta}(\eta'+\eta'')
 \end{gather*}
 with $ \{i,j,k\} \in \{1,2,3\} $.
\end{free text}

\begin{free text} \label{sing-special-G_s}
 {\bf Singular specializations of the supergroup(s) $\boldsymbol{\bG_\bsigma}$.} From the very construction of the supergroups $ \bG_\bsigma $, we get that
\begin{gather*}
\text{\it $ \bG_\bsigma $ is simple $($as a Lie supergroup$)$ for all $ \bsigma = (\sigma_1,\sigma_2,\sigma_3) \in V^\times $},
\end{gather*}
 where we recall that a Lie supergroup is said to be {\it simple} if it has no non-trivial normal closed connected Lie sub-supergroup. This follows from the presentation of $ \bG_\bsigma $ in Section~\ref{pres-G_s} above, or it can be seen as a direct consequence of the relation $ \operatorname{Lie}(\bG_\bsigma) = \fg(\bsigma) = \fg_\bsigma $ and of Proposition~\ref{prop: properties-g_sigma}.

 On the other hand, the situation is different at ``singular values'' of the parameter $ \bsigma $: the following records the whole situation.
\end{free text}

\begin{thm} \label{thm_G-spec} Let $ \bsigma \in V $ as usual, and keep notation as above.
\begin{enumerate}\itemsep=0pt
\item[$(1)$] If $ \bsigma \in V^\times $, then the Lie supergroup $ \bG_\bsigma $ is simple.
\item[$(2)$] If $ \bsigma \in V \setminus V^\times $, with $ \sigma_i = 0 $ and $ \sigma_j \not= 0 \not= \sigma_k $ for $ \{i,j,k\} = \{1,2,3\} $, then $ A_i $ is a {\it central} subgroup of $ \bG_\bsigma $, isomorphic to $ \C \times \C^\times \times \C $, and $ \bG_\bsigma $ is the universal central extension of $ \mathbb{P}{\rm SL}(2|2) $ by $ A_i $; in other words,
there exists a short exact sequence of Lie supergroups
\begin{gather*} \boldsymbol{1} \relbar\joinrel\relbar\joinrel\relbar\joinrel\relbar\joinrel\longrightarrow \C \times \C^\times \times \C \cong A_i \relbar\joinrel\relbar\joinrel\relbar\joinrel\relbar\joinrel\longrightarrow \bG_\bsigma \relbar\joinrel\relbar\joinrel\relbar\joinrel\relbar\joinrel\longrightarrow \mathbb{P}\text{SL}(2|2) \relbar\joinrel\relbar\joinrel\relbar\joinrel\relbar\joinrel\longrightarrow \boldsymbol{1}.
\end{gather*}
\item[$(3)$] If $ \bsigma = \boldsymbol{0} ( \!\in V \setminus V^\times ) $, i.e., $ \sigma_h = 0 $ for all $ h \in \{1,2,3\} $, then $
 ( \bG_\bsigma )_{\rm rd} \cong {\big( \C \times \C^\times \times \C \big)}^{\times 3} $ is the center of $ \bG_\bsigma $, and the quotient $ \bG_\bsigma / ( \bG_\bsigma \big)_{\rm rd}
 \cong \C^{8} $ is Abelian; in particular, $ \bG_\bsigma $ is a~central extension of $ \C^{8} $ by $ ( \C \times \C^\times \times \C )^{\times 3} $, i.e., there exists a short exact sequence of Lie supergroups, with non-Abelian middle term,
\begin{gather*}\boldsymbol{1} \relbar\joinrel\relbar\joinrel\relbar\joinrel\relbar\joinrel\longrightarrow {\big( \C \times \C^\times \times \C \big)}^{\times 3} \cong {\big( \bG_\bsigma \big)}_{\rm rd}
\relbar\joinrel\relbar\joinrel\relbar\joinrel\relbar\joinrel\longrightarrow \bG_\bsigma \relbar\joinrel\relbar\joinrel\relbar\joinrel\relbar\joinrel\longrightarrow \C^{8} \relbar\joinrel\relbar\joinrel\relbar\joinrel\relbar\joinrel\longrightarrow \boldsymbol{1}.
\end{gather*}
\end{enumerate}
\end{thm}

\begin{proof} The claim follows directly from the presentation of $ \bG_\bsigma $ given in Section~\ref{pres-G_s} above, or also from the relation $ \operatorname{Lie}(\bG_\bsigma) = \fgs $ along with Theorem~\ref{thm_g-spec}.
\end{proof}

\subsection[Second family: the Lie supergroups $ \bG'_\bsigma $]{Second family: the Lie supergroups $\boldsymbol{\bG'_\bsigma}$} \label{subsect: LieSGroups G'_s}

 Given $ \bsigma = (\sigma_1,\sigma_2,\sigma_3) \in V $, let $ \fg' := \fgps $ be the complex Lie superalgebra associated with $ \bsigma $ as in Section~\ref{subsect-g'(sigma)}, and let $ \fg'_\zero $ be its even part. Fix the $ \C $-basis $ \{ X'_{2\varepsilon_i}, H'_{2\varepsilon_i}, X'_{-2\varepsilon_i} \}_{i=1,2,3} $ of $ \fg'_\zero $ as in Section~\ref{bases-g_s}, and set $ \fa'_i := \C X'_{2\varepsilon_i} \oplus \C H'_{2\varepsilon_i} \oplus \C X'_{-2\varepsilon_i} $ for each $ i $: each one of these is a Lie subalgebra of $ \fg'_\zero $, with $ \fg'_\zero = \fa'_1 \oplus \fa'_2 \oplus \fa'_3 $. Moreover, each Lie algebra $ \fa'_i $ is isomorphic to $ \fsl(2) $, an explicit isomorphism being given by $ X'_{2\varepsilon_i} \mapsto e $, $ H'_{2\varepsilon_i} \mapsto h $ and $ X'_{-2\varepsilon_i} \mapsto f $, where $ \{ e, h, f \} $ is the standard basis $ \fsl(2) $. It follows that $ \fg'_\zero $ is isomorphic to $ {\fsl(2)}^{\oplus 3} $.
 For each $ i \in \{1,2,3\} $, let $ A'_i $ be a copy of $ \text{\rm SL}_{ 2} $, and set $ G' := A'_1 \times A'_2 \times A'_3 $. By the previous analysis, $ \operatorname{Lie}(G') $ is isomorphic to $ \fg'_\zero $ and the $ \operatorname{Lie} (G' ) $-action lifts to a holomorphic $ G' $-action on $ \fg' $ again: in fact, one easily sees that this action is faithful too. With this action, $ \cP'_{\bsigma} := \big( G', \fg' \big) $ is a super Harish-Chandra pair (cf.\ Section~\ref{sHCp's}), which overall depends on $ \cP'_\bsigma $ (although~$ G' $ alone does not). Finally, we define
\begin{gather*} \bG'_\bsigma := \bG_{{}_{\cP'_{{}_{\bsigma}}}} \end{gather*}
 to be the complex Lie supergroup associated with the super Harish-Chandra pair $ \cP'_{\bsigma} $ via the equivalence of categories given in Section~\ref{sHCp's->Liesgrp's}.

\begin{free text} \label{pres-G'_s}
 {\bf A presentation of $ \boldsymbol{\bG'_\bsigma }$.} The supergroups $ \bG'_\bsigma $ can be described in concrete terms via an explicit presentation by generators and relations of all the abstract groups $ \bG'_\bsigma(A) $, with $ A $ ranging in $ \Wsalg $. To this end, we first consider the Lie group $ G' = A'_1 \times A'_2 \times A'_3 $ with $ A'_i \cong \text{\rm SL}_{ 2} $. Letting $ \exp\colon \fg'_\zero \cong \operatorname{Lie}\big(G'\big) \relbar\joinrel\relbar\joinrel\relbar\joinrel\longrightarrow G' $ be the exponential map, we consider $ x'_{2\varepsilon_i}(c) := \exp (c X'_{2\varepsilon_i} ) $, $ h'_{2\varepsilon_i}(c) := \exp (c H'_{2\varepsilon_i} ) $, $ x'_{-2\varepsilon_i}(c) := \exp (c X'_{-2\varepsilon_i} ) $ and $ h'_\theta(c) := \exp (c H'_\theta ) $ for all $ c \in \C $. Note that the commutation relations $ [ H'_{2\varepsilon_r}, H'_{2\varepsilon_s} ] = 0 $ along with $ \sigma_1 H'_{2\varepsilon_1} + \sigma_2 H'_{2\varepsilon_2} + \sigma_3 H'_{2\varepsilon_3} = 2 H'_\theta $ inside~$ \fgpsc $ together imply, inside $ G'_+ $, the group relations
 $ h'_{2\varepsilon_1}(\sigma_1 c) h'_{2\varepsilon_2}(\sigma_2 c) h'_{2\varepsilon_3}(\sigma_3 c) = {h'_\theta(c)}^2 $ for all $ c \in \C $.
 The complex Lie group $ G' $ is clearly generated by the set
\begin{gather*} \varGamma'_\zero := \{ x'_{2\varepsilon_i}(c), h'_{2\varepsilon_i}(c), h'_\theta(c), x'_{-2\varepsilon_i}(c) \,|\, c \in \C \}
 \end{gather*}
 (where the $ h'_\theta(c) $'s might be discarded, but we prefer to keep them). Then, looking at $ G' $ as a~(totally even) supergroup thought of as a group-valued functor $ G' \colon \Wsalg \longrightarrow \text{\rm ($ \cat{grps} $)} $, {\it each abstract group $ G'(A) $ of its $ A $-points~-- for $ A \in \Wsalg $~-- is generated by the set}
\begin{gather} \label{eq: gen-set G'}
 \varGamma'_\zero(A) := \{ x'_{2\varepsilon_i}(a), h'_{2\varepsilon_i}(a), h'_\theta(a), x'_{-2\varepsilon_i}(a) \,|\, a \in A_\zero \}.
\end{gather}
 Following Section~\ref{sHCp's->Liesgrp's}, we need as generators of $ \bG'_\bsigma(A) := \bG_{{}_{\cP'_{{}_{\bsigma}}}}(A) $ all the elements of $ G'(A) $ and all those of the form $ x'_{\pm\theta}(\eta) := ( 1 + \eta X'_{\pm\theta} ) $ or $ x'_{\pm\beta_i}(\eta) := ( 1 + \eta X'_{\pm\beta_i} ) $ with $ \eta \in A_\one $ and $ i \in \{1,2,3\} $~-- since now we fix $ {\big\{ Y'_i \big\}}_{i \in I} = \{ X'_{\pm\theta}, X'_{\pm\beta_i} \}_{i=1,2,3} $ as our $ \C $-basis of $ \fg'_\one $; we denote the set of all the latter by $ \varGamma'_\one(A) := \{ x'_{\pm\theta}(\eta), x'_{\pm\beta_i}(\eta) \,|\, \eta \in A_\one \} $.

 Implementing the recipe in Section~\ref{sHCp's->Liesgrp's}, and recalling that $ G'(A) $ is generated by $ \varGamma'_\zero(A) $, we can now slightly modify the relations presented in Section~\ref{sHCp's->Liesgrp's} and consider instead the following, alternative {\it full set of relations} among the generators of $ \bG'_\bsigma(A) $:
\begin{gather*}
 1_{{}_G'} = 1, \qquad g' \cdot g'' = g' \cdot_{{}_{G'}} g'', \qquad \forall\, g', g'' \in G'(A),\\
 h'_{2\vep_i}(a) x'_{\pm\beta_j}(\eta) {h'_{2\vep_i}(a)}^{-1} = x'_{\pm\beta_j}\big( e^{\pm {(-1)}^{-\delta_{i,j}} a} \eta \big),\\
 h'_{2\vep_i}(a) x'_{\pm\theta}(\eta) {h'_{2\vep_i}(a)}^{-1} = x'_{\pm\theta}\big( e^{\pm a} \eta \big),\\
 h'_\theta(a) x'_{\pm\beta_i}(\eta) {h'_\theta(a)}^{-1} = x'_{\pm\beta_i}\big( e^{\mp \sigma_i a} \eta \big), \qquad h'_\theta(a) x'_{\pm\theta}(\eta) {h'_\theta(a)}^{-1} = x'_{\pm\theta}(\eta,\\
 x'_{2\vep_i}(a) x'_{\beta_j}(\eta) {x'_{2\vep_i}(a)}^{-1} = x'_{\beta_j}(\eta) x'_{\theta}( \delta_{i,j} a \eta ),\\
 x'_{2\vep_i}(a) x'_{-\beta_j}(\eta) {x'_{2\vep_i}(a)}^{-1} = x'_{-\beta_j}(\eta) x'_{\beta_k}( (1-\delta_{i,j}) a \eta ),\\
 x'_{-2\vep_i}(a) x'_{\beta_j}(\eta) {x'_{-2\vep_i}(a)}^{-1} = x'_{\beta_j}(\eta) x'_{-\beta_k}( (1-\delta_{i,j}) a \eta ),\\
 x'_{-2\vep_i}(a) x'_{-\beta_j}(\eta) {x'_{-2\vep_i}(a)}^{-1} = x'_{-\beta_j}(\eta) x'_{-\theta}( \delta_{i,j} a \eta ),\\
 x'_{2\vep_i}(a) x'_{\theta}(\eta) {x'_{2\vep_i}(a)}^{-1} = x'_{\theta}(\eta),\\
 x'_{2\vep_i}(a) x'_{-\theta}(\eta) {x'_{2\vep_i}(a)}^{-1} = x'_{-\theta}(\eta) x'_{-\beta_i}( a \eta ),\\
 x'_{-2\vep_i}(a) x'_{\theta}(\eta) {x'_{-2\vep_i}(a)}^{-1} = x'_{\theta}(\eta) x'_{\beta_i}( a \eta ),\\
 x'_{-2\vep_i}(a) x'_{-\theta}(\eta) {x'_{-2\vep_i}(a)}^{-1} = x'_{-\theta}(\eta),\\
 x'_{\beta_i}(\eta_i) x'_{\beta_j}(\eta'_j) = x'_{2\vep_k}( +(1-\delta_{i,j}) \sigma_i \eta'_j \eta_i) x'_{\beta_j}(\eta'_j) x'_{\beta_i}(\eta_i),\\
 x'_{-\beta_i}(\eta_i) x'_{-\beta_j}(\eta'_j) = x'_{-2\vep_k}( -(1-\delta_{i,j}) \sigma_i \eta'_j \eta_i) x'_{-\beta_j}(\eta'_j) x'_{-\beta_i}(\eta_i),\\
 x'_{\beta_i}(\eta_i) x'_{-\beta_j}(\eta'_j) = h'_{2\vep_i}(\delta_{i,j} \sigma_i \eta'_j \eta_i) h'_{\theta}(-\delta_{i,j} \eta'_j \eta_i) x'_{-\beta_j}(\eta'_j) x'_{\beta_i}(\eta_i),\\
 x'_{\beta_i}(\eta_i) x'_{\theta}(\eta) = x'_{\theta}(\eta) x'_{\beta_i}(\eta_i), \qquad x'_{\beta_i}(\eta_i) x'_{-\theta}(\eta) = x'_{-2\vep_i}(+\sigma_i \eta \eta_i) x'_{-\theta}(\eta) x'_{+\beta_i}(\eta_i),\\
 x'_{-\beta_i}(\eta_i) x'_{\theta}(\eta) = x'_{2\vep_i}(-\sigma_i \eta \eta_i) x'_{\theta}(\eta) x'_{-\beta_i}(\eta_i), \qquad x'_{-\beta_i}(\eta_i) x'_{-\theta}(\eta) = x'_{-\theta}(\eta) x'_{-\beta_i}(\eta_i),\\
 x'_{\theta}(\eta_+) x'_{-\theta}(\eta_-) = v h'_{\theta}(\eta_- \eta_+ ) x'_{-\theta}(\eta_-) x'_{\theta}(\eta_+),\\
 x'_{\pm\beta_i}(\eta') x'_{\pm\beta_i}(\eta'') = x'_{\pm\beta_i}(\eta'+\eta''), \qquad x'_{\pm\theta}(\eta') x'_{\pm\theta}(\eta'') = x'_{\pm\theta}(\eta'+\eta'')
 \end{gather*}
 with $ \{i,j,k\} = \{1,2,3\} $.
\end{free text}

\begin{free text} \label{sing-special-G'_s} {\bf Singular specializations of the supergroup(s) $\boldsymbol{\bG'_\bsigma}$.} By construction, for the supergroups $ \bG'_\bsigma $ we have that
\begin{gather*}
\text{\it $ \bG'_\bsigma $ is simple $($as a Lie supergroup$)$ for all $ \bsigma = (\sigma_1,\sigma_2,\sigma_3) \in V^\times $.}
\end{gather*}
 Indeed, this follows from the presentation of $ \bG'_\bsigma $ in Section~\ref{pres-G'_s} above, but also as a fallout of the relation $ \operatorname{Lie}\big(\bG'_\bsigma\big) = \fgpsc = \fg_\bsigma $ along with Proposition~\ref{prop: properties-g_sigma}. The situation is different at ``singular values'' of the parameter $ \bsigma $; the complete result is
\end{free text}

\begin{thm} \label{thm_G'-spec}Given $ \bsigma \in V $, keep notation as above.
\begin{enumerate}\itemsep=0pt
\item[$(1)$] If $ \bsigma \in V^\times $, then the Lie supergroup $ \bG'_\bsigma $ is simple.
\item[$(2)$] If $ \bsigma \in V \setminus V^\times $, with $ \sigma_i = 0 $ and $ \sigma_j \not= 0 \not= \sigma_k $ for $ \{i,j,k\} = \{1,2,3\} $, then letting
 $ \bB'_i $ be the Lie subsupergroup of $ \bG'_\bsigma $ defined on every $ A \in \Wsalg $ by
\begin{gather*}
 \bB'_i(A) := \big\langle \{ h'_{2 \varepsilon_t}(a), x'_{\pm\alpha}(b), x'_{\pm\beta}(\eta)\}^{t \not= i, \alpha \in \Delta_\zero \setminus \{2 \varepsilon_i\}, \beta \in \Delta_\one}_{ a \in A_\zero, b \in A_\zero, \eta \in A_\one} \big\rangle
 \end{gather*}
we have $ \bB'_i \trianglelefteq \bG'_\bsigma $ $($a normal Lie subsupergroup$)$, $ A'_i \leq \bG'_\bsigma $ $($a Lie subsupergroup$)$, and there exist isomorphisms $ \mathbf{B}^{ \prime}_i \cong {\mathbb{P}\text{\rm SL}(2|2)} $, $ A'_i \cong \text{\rm SL}_2 $ and $ \bG'_\bsigma \cong \text{\rm SL}_2 \ltimes {\mathbb{P}\text{\rm SL}(2|2)} $~-- a semidirect product of Lie supergroups. In short, there exists a~split short exact sequence
\begin{gather*} \boldsymbol{1} \relbar\joinrel\relbar\joinrel\relbar\joinrel\longrightarrow {\mathbb{P}\text{\rm SL}(2|2)} \cong \bB'_i \relbar\joinrel\relbar\joinrel\relbar\joinrel\longrightarrow \bG'_\bsigma {\buildrel
{{\displaystyle \dashleftarrow \hskip-4pt \text{$-$} \text{$-$} \text{$-$}}} \over {\relbar\joinrel\relbar\joinrel\relbar\joinrel\relbar\joinrel\longrightarrow}} A'_i \cong \text{\rm SL}_2 \relbar\joinrel\relbar\joinrel\relbar\joinrel\longrightarrow \boldsymbol{1}.
\end{gather*}

\item[$(3)$] If $ \bsigma = \boldsymbol{0} (\! \in V \setminus V^\times ) $, i.e., $ \sigma_h = 0 $ for all $ h \in \{1,2,3\} $, then
 letting $ {\big( \bG'_\bsigma \big)}_\one $ be the Lie subsupergroup of $ \bG'_\bsigma $ defined on every $ A \in \Wsalg $ by
\begin{gather*} ( \bG'_\bsigma )_\one(A) := \big\langle \{ x'_\alpha(\eta) \}^{\alpha \in \Delta_\one}_{\eta \in A_\one} \big\rangle \end{gather*}
we have $( \bG'_\bsigma)_{\rm rd} \cong \text{\rm SL}_2^{ 3} $ and $( \bG'_\bsigma )_\one \cong \C^{ 8} $ as Lie $($super$)$groups,
 $( \bG'_\bsigma)_\one \cong \boxtimes_{i=1}^3 \Box_i \cong \Box^{ \boxtimes 3} $ $ \big({\cong} \C^{ 8} \big) $
 as modules over $( \bG'_\bsigma)_\zero \cong \text{\rm SL}_2^{ 3} $~-- where $ \Box_i \cong \Box := \C \vert + \rangle \oplus \C \vert -\rangle $
 is the tautological $2$-dimensional module over the $ i $-th copy $ \text{\rm SL}^{(i)}_2 $ of $ \text{\rm SL}_2 $ $($for $ i = 1, 2, 3)$.
 Finally, we have{\samepage
\begin{gather*} \bG'_\bsigma \cong ( \bG'_\bsigma )_{\rm rd} \ltimes ( \bG'_\bsigma )_\one \cong \text{\rm SL}_2^{ 3} \ltimes \big( {\boxtimes}_{i=1}^3 \Box_i \big) \cong \text{\rm SL}_2^{ 3} \ltimes \Box^{ \boxtimes 3} \end{gather*}
-- a semidirect product of Lie supergroups.}

 In other words, there is a {\it split} short exact sequence
\begin{gather*} \boldsymbol{1} \relbar\joinrel\relbar\joinrel\relbar\joinrel\relbar\joinrel\longrightarrow
 \Box^{ \boxtimes 3}
 \cong {\big( \bG'_\bsigma \big)}_\one \relbar\joinrel\relbar\joinrel\relbar\joinrel\relbar\joinrel\longrightarrow \bG'_\bsigma {\buildrel
{{\displaystyle \dashleftarrow \hskip-4pt \text{$-$} \text{$-$} \text{$-$}}} \over {\relbar\joinrel\relbar\joinrel\relbar\joinrel\relbar\joinrel\longrightarrow}} {\big( \bG'_\bsigma \big)}_{\rm rd} \cong \text{\rm SL}_2^{\times 3} \relbar\joinrel\relbar\joinrel\relbar\joinrel\relbar\joinrel\longrightarrow \boldsymbol{1}.
\end{gather*}
\end{enumerate}
\end{thm}

\begin{proof} Like for Theorem \ref{thm_G-spec}, the present claim can be obtained from the presentation of $ \bG'_\bsigma $ in Section~\ref{pres-G'_s}, or otherwise from the relation $ \operatorname{Lie}\big( \bG'_\bsigma \big) = \fg'(\bsigma) $ along with Theorem \ref{thm_g'-spec}.
\end{proof}

\subsection[Third family: the Lie supergroups $ \bG''_\bsigma $]{Third family: the Lie supergroups $\boldsymbol{\bG''_\bsigma}$} \label{subsect: LieSGroups G''_s}

 Given $ \bsigma = (\sigma_1,\sigma_2,\sigma_3) \in V $, let $ \fg'' := \fgss $ be the complex Lie superalgebra associated with~$ \bsigma $ as in Section~\ref{subsect-g''_Z[ts]}, and let $ \fg''_\zero $ be its even part. Fix the elements $ X_\alpha, H'_{2\varepsilon_i}, H'_\theta $ (with $ \alpha \in \Delta $, $ i \in \{1,2,3\} $) of $ \fg $ as in Section~\ref{third-basis} and set $ \fa''_i := \C X_{2\varepsilon_i} \oplus \C H'_{2\varepsilon_i} \oplus \C X_{-2\varepsilon_i} $ for each $ i $: the latter are Lie subalgebras of $ \fg''_\zero $ such that $ \fg''_\zero = \fa''_1 \oplus \fa''_2 \oplus \fa''_3 $. Moreover, every $ \fa''_i $ is isomorphic to $ \fsl(2) $ when $ \sigma_i \not= 0 $~-- an explicit isomorphism being given by $ X_{2\varepsilon_i} \mapsto \sigma_i e $, $ H'_{2\varepsilon_i} \mapsto h $ and $ X_{-2\varepsilon_i} \mapsto \sigma_i f $, where $ \{ e, h, f \} $ is the standard basis $ \fsl(2) $~-- while for $ \sigma_i = 0 $ it is isomorphic to the Lie subalgebra of $ \mathfrak{b}_+ \oplus \mathfrak{b}_- $, with $ \mathfrak{b}_+ := \C e + \C h $ and $ \mathfrak{b}_- := \C h + \C f $ being the standard Borel subalgebras inside~$ \fsl(2) $, with $ \C $-basis $ \{ (e,0), (h,h), (0,f) \} $.

 Let $ B_\pm $ be the Borel subgroup of $ \text{\rm SL}_{ 2} $ of all upper, resp.\ lower, triangular matrices, and let $ S $ be the subgroup of $ B_+ \times B_- $ whose elements are all the pairs of matrices $ (X_+,X_-) $ such that the diagonal parts of $ X_+ $ and of $ X_- $ are the same. For each $ i \in \{1,2,3\} $, let $ A''_i $ (depending on $ \sigma_i $) respectively be a copy of $ \text{\rm SL}_{ 2} $ if $ \sigma_i \not= 0 $ and a copy of $ S $ otherwise; then set $ G'' := A''_1 \times A''_2 \times A''_3 $. The adjoint action of $ \fg''_\zero \cong \operatorname{Lie}\big(G''\big) $ on $ \fg'' $ lifts to a holomorphic $ G'' $-action on $ \fg'' $, which is faithful again; then the pair $ \cP''_{\bsigma} := \big( G'', \fg'' \big) $ with this action is a super Harish-Chandra pair, in the sense of Section~\ref{sHCp's}.
 At last, we can define
\begin{gather*} \bG''_\bsigma := \bG_{{}_{\cP''_{{}_{\bsigma}}}} \end{gather*}
 as being the complex Lie supergroup associated with the super Harish-Chandra pair $ \cP''_{\bsigma} $ through the equivalence of categories given in Section~\ref{sHCp's->Liesgrp's}.

\begin{free text} \label{pres-G''_s} {\bf A presentation of $ \bG''_\bsigma $.} In order to describe the supergroups $ \bG''_\bsigma $, we aim now for an explicit presentation by generators and relations of the abstract groups $ \bG''_\bsigma(A) $, for all \mbox{$ A \in \Wsalg $}. To start with, let $ G'' = A''_1 \times A''_2 \times A''_3 $ be the complex Lie group considered above, and let $ \exp \colon \fg''_\zero \cong \operatorname{Lie}\big(G''\big) \relbar\joinrel\relbar\joinrel\relbar\joinrel\longrightarrow G'' $ be the exponential map: then consider
 $ x_{2\varepsilon_i}(c) := \exp\big(c X_{2\varepsilon_i}\big) $,
 $ h'_{2\varepsilon_i}(c) := \exp\big(c H'_{2\varepsilon_i}\big) $,
 $ x_{-2\varepsilon_i}(c) := \exp\big(c X_{-2\varepsilon_i}\big) $
and
 $ h'_\theta(c) := \exp (c H'_\theta ) $ for all $ c \in \C $. It is clear that $ G'' $ is generated by the set
\begin{gather*} \varGamma''_\zero := \{ x_{2\varepsilon_i}(c), h'_{2\varepsilon_i}(c), h'_\theta(c), x_{-2\varepsilon_i}(c) \,|\, c \in \C \} \end{gather*}
(actually the $ h'_\theta(c) $'s might be discarded, but we choose to keep them); therefore, looking at $ G'' $ as a supergroup, thought of as a group-valued functor $ G'' \colon \Wsalg \longrightarrow \text{\rm ($ \cat{grps} $)} $, {\it every abstract group $ G''(A) $ of its $ A $-points~-- for $ A \in \Wsalg $~-- is generated by the set}
\begin{gather} \label{eq: gen-set G''}
 \varGamma''_\zero(A) := \{ x_{2\varepsilon_i}(a), h'_{2\varepsilon_i}(a), h'_\theta(a), x_{-2\varepsilon_i}(a) \,|\, a \in A_\zero \}.
\end{gather}

According to Section~\ref{sHCp's->Liesgrp's}, the group $ \bG''_\bsigma(A) := \bG_{{}_{\cP''_{{}_{\bsigma}}}}(A) $ is generated by $ G''(A) $ and all elements of the form $ x_{\pm\theta}(\eta) := \big( 1 + \eta X_{\pm\theta} \big) $ or $ x_{\pm\beta_i}(\eta) := \big( 1 + \eta X_{\pm\beta_i} \big) $ with $ \eta \in A_\one $ and $ i \in \{1,2,3\} $~-- as now the chosen $ \C $-basis of $ \fg''_\one $ is $ {\big\{ Y'_i \big\}}_{i \in I} = \big\{ X_{\pm\theta}, X_{\pm\beta_i} \big\}_{i=1,2,3} $; the set of all the latter is denoted $ \varGamma''_\one(A) := \big\{ x_{\pm\theta}(\eta), x_{\pm\beta_i}(\eta) \,|\, \eta \in A_\one \big\} $~-- coinciding with $ \varGamma'_\one(A) $ in Section~\ref{pres-G'_s}.

 From the recipe in Section~\ref{sHCp's->Liesgrp's}, and the fact that $ G''(A) $ is generated by $ \varGamma''_\zero(A) $, with a~slight modification of the relations in Section~\ref{sHCp's->Liesgrp's} we can find the following {\it full set of relations} among generators of $ \bG''_\bsigma(A) $ (for all $ \{i,j,k\} = \{1,2,3\} $):
\begin{gather*}
 1_{{}_{G''}} = 1, \qquad g' \cdot g'' = g' \cdot_{{}_{G''}} g'', \qquad \forall \, g', g'' \in G''(A),\\
 h'_{2\vep_i}(a) x_{\pm\beta_j}(\eta) {h'_{2\vep_i}(a)}^{-1} = x_{\pm\beta_j}\big( e^{\pm {(-1)}^{\delta_{i,j}} a} \eta \big),\\
 h'_{2\vep_i}(a) x_{\pm\theta}(\eta) {h'_{2\vep_i}(a)}^{-1} = x_{\pm\theta}\big( e^{\pm a} \eta \big),\\
 h_\theta(a) x_{\pm\beta_i}(\eta) {h_\theta(a)}^{-1} = x_{\pm\beta_i}\big( e^{\mp \sigma_i a} \eta \big), \qquad h_\theta(a) x_{\pm\theta}(\eta) {h_\theta(a)}^{-1} = x_{\pm\theta}(\eta),\\
 x_{2\vep_i}(a) x_{\beta_j}(\eta) {x_{2\vep_i}(a)}^{-1} = x_{\beta_j}(\eta) x_{\theta}( \delta_{i,j} \sigma_i a \eta ),\\
 x_{2\vep_i}(a) x_{-\beta_j}(\eta) {x_{2\vep_i}(a)}^{-1} = x_{-\beta_j}(\eta) x_{\beta_k}( (1-\delta_{i,j}) \sigma_i a \eta ),\\
 x_{-2\vep_i}(a) x_{\beta_j}(\eta) {x_{-2\vep_i}(a)}^{-1} = x_{\beta_j}(\eta) x_{-\beta_k}( (1-\delta_{i,j}) \sigma_i a \eta ),\\
 x_{-2\vep_i}(a) x_{-\beta_j}(\eta) {x_{-2\vep_i}(a)}^{-1} = x_{-\beta_j}(\eta) x_{-\theta}( \delta_{i,j} \sigma_i a \eta ),\\
 x_{2\vep_i}(a) x_{\theta}(\eta) {x_{2\vep_i}(a)}^{-1} = x_{\theta}(\eta),\\
 x_{2\vep_i}(a) x_{-\theta}(\eta) {x_{2\vep_i}(a)}^{-1} = x_{-\theta}(\eta) x_{-\beta_i}( \sigma_i a \eta ),\\
 x_{-2\vep_i}(a) x_{\theta}(\eta) {x_{-2\vep_i}(a)}^{-1} = x_{\theta}(\eta) x_{\beta_i}( \sigma_i a \eta ),\\
 x_{-2\vep_i}(a) x_{-\theta}(\eta) {x_{-2\vep_i}(a)}^{-1} = x_{-\theta}(\eta),\\
 x_{\beta_i}(\eta_i) x_{\beta_j}(\eta'_j) = x_{2\vep_k}((1-\delta_{i,j}) \eta'_j \eta_i) x_{\beta_j}(\eta'_j) x_{\beta_i}(\eta_i),\\
 x_{-\beta_i}(\eta_i) x_{-\beta_j}(\eta'_j) = x_{-2\vep_k}( -(1-\delta_{i,j}) \eta'_j \eta_i) x_{-\beta_j}(\eta'_j) x_{-\beta_i}(\eta_i),\\
 x_{\beta_i}(\eta_i) x_{-\beta_j}(\eta'_j) = h'_{2\vep_i}(\delta_{i,j} \sigma_i \eta'_j \eta_i) h'_{\theta}(-\delta_{i,j} \eta'_j \eta_i) x_{-\beta_j}(\eta'_j) x_{\beta_i}(\eta_i),\\
 x_{\beta_i}(\eta_i) x_{\theta}(\eta) = x_{\theta}(\eta) x_{\beta_i}(\eta_i), \qquad x_{\beta_i}(\eta_i) x_{-\theta}(\eta) = x_{-2\vep_i}(\eta \eta_i) x_{-\theta}(\eta) x_{+\beta_i}(\eta_i),\\
 x_{-\beta_i}(\eta_i) x_{\theta}(\eta) = x_{2\vep_i}(-\eta \eta_i) x_{\theta}(\eta) x_{-\beta_i}(\eta_i), \qquad x_{-\beta_i}(\eta_i) x_{-\theta}(\eta) = x_{-\theta}(\eta) x_{-\beta_i}(\eta_i),\\
 x_{\theta}(\eta_+) x_{-\theta}(\eta_-) = h'_{\theta}(\eta_- \eta_+ ) x_{-\theta}(\eta_-) x_{\theta}(\eta_+),\\
 x_{\pm\beta_i}(\eta') x_{\pm\beta_i}(\eta'') = x_{\pm\beta_i}(\eta'+\eta''), \qquad x_{\pm\theta}(\eta') x_{\pm\theta}(\eta'') = x_{\pm\theta}(\eta'+\eta'').
\end{gather*}
\end{free text}

\begin{free text} \label{sing-special-G''_s} {\bf Singular specializations of the supergroup(s) $\boldsymbol{\bG''_\bsigma}$.} One sees easily that for the supergroups $ \bG''_\bsigma $ we have
\begin{gather*}
\text{\it $ \bG''_\bsigma $ is simple $($as a Lie supergroup$)$ for all $ \bsigma = (\sigma_1,\sigma_2,\sigma_3) \in V^\times $.}
\end{gather*}

 This follows from the presentation of $ \bG''_\bsigma $ in Section~\ref{pres-G''_s} above, and also as a consequence of the relation $ \operatorname{Lie}\big(\bG''_\bsigma\big) = \fgss = \fg_\bsigma $ along with Proposition~\ref{prop: properties-g_sigma}.

 Things change, instead, for ``singular values'' of $ \bsigma $: hereafter is the general result.
\end{free text}

\begin{thm} \label{thm_G''-spec} Given $ \bsigma \in V $, keep notation as above.
\begin{enumerate}\itemsep=0pt
\item[$(1)$] If $ \bsigma \in V^\times $, then the Lie supergroup $ \bG''_\bsigma $ is simple.
\item[$(2)$] If $ \bsigma \in V \setminus V^\times $, with $ \sigma_i = 0 $ and $ \sigma_j \not= 0 \not= \sigma_k $ for $ \{i,j,k\} = \{1,2,3\} $, consider the subsupergroup $ \bK''_{ i} $ of $ \bG''_\bsigma $ which is given by
\begin{gather*} \bK''_{ i}(A) := \big\langle \{ x_{2 \vep_i}(a_+), x_{-2 \vep_i}(a_-) \}_{a_+, a_- \in A_\zero} \big\rangle \qquad \forall\, A \in \Wsalg.
\end{gather*}
Then $ \bK''_{ i} $ is an Abelian normal Lie subsupergroup of $ \bG''_\bsigma $, hence~-- letting $ \overline{\bB''_i} := \bG''_\bsigma / \bK''_i $ be the quotient supergroup~-- there exists a short exact sequence
\begin{gather*} \boldsymbol{1} \relbar\joinrel\relbar\joinrel\relbar\joinrel\longrightarrow \bK''_{ i} \relbar\joinrel\relbar\joinrel\relbar\joinrel\longrightarrow \bG''_\bsigma \relbar\joinrel\relbar\joinrel\relbar\joinrel\longrightarrow \overline{\bB''_i} \relbar\joinrel\relbar\joinrel\relbar\joinrel\longrightarrow \boldsymbol{1}.\end{gather*}
Furthermore, defining inside $ \overline{\bB''_i}(A) $ the two subgroups
 $ \overline{\bH''_{ i}}(A) := \big\langle \{ \overline{h'_{2 \vep_i}(a_i)} \}_{a_i \in A_\zero} \big\rangle $ and
 $ \overline{\bD''_i}(A) := \big\langle \{ \overline{h'_{2 \vep_j}(a_j)}, \overline{x_{+2\vep_j}(c_{\scriptscriptstyle +})}, \overline{x_{-2\vep_j}(c_{\scriptscriptstyle -})}, \overline{x_\beta(\eta)} \}_{j \not= i, \beta \in \Delta_\one}^{a_j, c_\pm \in A_\zero, \eta \in A_\one} \big\rangle $~-- for all $ A \in \Wsalg $~-- we overall find two Lie subsupergroups $ \overline{\bD''_i} $ and $ \overline{\bH''_{ i}} $ of $ \overline{\bB''_i} $ such that $ \overline{\bD''_i} $ is a normal Lie subsupergroup with $ \overline{\bD''_i} \cong \mathbb{P}\text{\rm SL}(2|2) $, $ \overline{\bH''_{ i}} \cong \C^\times $, and $ \overline{\bB''_i} \cong \overline{\bH''_{ i}} \ltimes \overline{\bD''_i} $~-- a semidirect product of Lie supergroups. In short, there exists a {\it split} short exact sequence
 \begin{gather*} \boldsymbol{1} \relbar\joinrel\relbar\joinrel\relbar\joinrel\relbar\joinrel\longrightarrow \mathbb{P}\text{\rm SL}(2|2) \cong \overline{\bD''_i} \relbar\joinrel\relbar\joinrel\relbar\joinrel\relbar\joinrel\longrightarrow \overline{\bB''_i} {\buildrel
{{\displaystyle \dashleftarrow \hskip-4pt \text{$-$} \hskip-0pt \text{$-$} \hskip0pt \text{$-$}}} \over {\relbar\joinrel\relbar\joinrel\relbar\joinrel\relbar\joinrel\longrightarrow}} \overline{\bH''_{ i}} \cong \C^\times \relbar\joinrel\relbar\joinrel\relbar\joinrel\relbar\joinrel\longrightarrow \boldsymbol{1}.
\end{gather*}
\item[$(3)$] If $ \bsigma = \boldsymbol{0}$ $(\! \in V \setminus V^\times ) $, i.e., $ \sigma_h = 0 $ for all $ h \in \{1,2,3\} $, then $ \bK'' := \mathop{\times}\limits_{i=1}^3 \bK''_{ i} $ is an Abelian normal Lie subsupergroup, hence $ \overline{\bB''} := \bG''_\bsigma / \bK'' $ is a quotient Lie supergroup $($of~$ \bG''_\bsigma )$; therefore, there exists a short exact sequence
\begin{gather*} \boldsymbol{1} \relbar\joinrel\relbar\joinrel\relbar\joinrel\relbar\joinrel\longrightarrow \bK'' := {\textstyle \mathop{\times}\limits_{i=1}^3} \bK''_{ i} \relbar\joinrel\relbar\joinrel\relbar\joinrel\relbar\joinrel\longrightarrow \bG''_\bsigma \relbar\joinrel\relbar\joinrel\relbar\joinrel\relbar\joinrel\longrightarrow \overline{\bB''} \relbar\joinrel\relbar\joinrel\relbar\joinrel\relbar\joinrel\longrightarrow \boldsymbol{1}.\end{gather*}
Moreover, setting
 $ \overline{\bO''}(A) := \big\langle \{ \overline{x_\beta(\eta)} \}_{\beta \in \Delta_\one}^{\eta \in A_\one} \big\rangle $
 and
 $ \overline{\bT''}(A) := \big\langle \{ \overline{h'_{2\vep_i}(a)} \}_{i \in \{1,2,3\} }^{a \in A_\zero} \big\rangle $
 inside $ \overline{\bB''}(A)\!$ for all $ A \in \Wsalg $ we overall find two subsupergroups $ \overline{\bO''} $ and $ \overline{\bT''} $ of $ \overline{\bB''} $ such that $ \overline{\bO''} $ is normal Abelian, isomorphic to $ \mathbb{A}_{\scriptscriptstyle C}^{0|8} $~-- the $($totally odd$)$ complex affine Abelian supergroup of superdimension $ (0|8) $~--
and $ \overline{\bT''} $ is Abelian, isomorphic to $ \mathbb{T}_{\scriptscriptstyle \C}^{ 3} $~-- the $($totally even$)$ $ 3 $-dimensional complex torus~-- with $ \overline{\bB''} \cong \overline{\bT''} \ltimes \overline{\bO''} $~-- a semidirect product of Lie supergroups. In other words, there exists a second, {\it split} short exact sequence
\begin{gather*} \boldsymbol{1} \relbar\joinrel\relbar\joinrel\relbar\joinrel\relbar\joinrel\longrightarrow \overline{\bO''} \cong \mathbb{A}_{\scriptscriptstyle C}^{0|8} \relbar\joinrel\relbar\joinrel\relbar\joinrel\relbar\joinrel\longrightarrow \overline{\bB''} {\buildrel {{\displaystyle \dashleftarrow \hskip-4pt \text{$-$} \text{$-$} \text{$-$}}} \over {\relbar\joinrel\relbar\joinrel\relbar\joinrel\relbar\joinrel\longrightarrow}} \mathbb{T}_{\scriptscriptstyle \C}^{ 3} \cong \overline{\bT''} \relbar\joinrel\relbar\joinrel\relbar\joinrel\relbar\joinrel\longrightarrow \boldsymbol{1}.\end{gather*}
\end{enumerate}
\end{thm}

\begin{proof} Like for Theorem \ref{thm_G-spec}, one can deduce the claim from the presentation of $ \bG''_\bsigma $ in Section~\ref{pres-G''_s}, or also from the relation $ \operatorname{Lie}\big( \bG''_\bsigma \big) = \fg''(\bsigma) $ along with Theorem \ref{thm_g''-spec}.
\end{proof}

\subsection[Lie supergroups from contractions: the family of the $ \widehat{\bG}_\bsigma $'s]{Lie supergroups from contractions: the family of the $\boldsymbol{\widehat{\bG}_\bsigma}$'s} \label{subsect: LieSGroups hatG_s}

 Given $ \bsigma = (\sigma_1,\sigma_2,\sigma_3) \in V $, following Section~\ref{contractions} we fix the element $ \tau := x_{1 } x_{2 } x_{3 } \in \C[\bx] $ and the ideal $ I = I_\bsigma := ( \{ x_i - \sigma_i \}_{i=1,2,3} ) $, and we consider the corresponding complex Lie algebra $ \widehat{\fg}(\bsigma) $, with $ {\widehat{\fg}(\bsigma)}_\zero $ and $ {\widehat{\fg}(\bsigma)}_\one $ as its even and odd part, respectively. With a slight abuse of notation, for any element $ Z \in \widehat{\fg}(\bx) $ we denote again by $ Z $ its corresponding coset in $ \widehat{\fg}(\bsigma) = \C[\bsigma] \otimes_{{}_{\C[\bx]}} \widehat{\fg}(\bx) \cong \widehat{\fg}(\bx) / I_\bsigma \widehat{\fg}(\bx) $ (see Section~\ref{contractions} for notation). By construction, $ \widehat{\fg}(\bsigma) $ admits as $ \C $-basis the set
\begin{gather*} \widehat{B} := \{ X_\alpha, H_{2\varepsilon_i} \,|\, \alpha \in \Delta, i \in \{1,2,3\} \} \cup \big\{ \widehat{X}_\beta := \tau X_\beta \,|\, \beta \in \Delta_\one \big\}.
\end{gather*}

 We consider also $ \widehat{\fa}_i := \fa_i $ ($ := \C X_{2\vep_i} \oplus \C H_{2\vep_i} \oplus \C X_{-2\vep_i} $) for all $ 1 = 1, 2, 3 $, that all are Lie subalgebras of $ \widehat{\fg}_\zero $, with $ \widehat{\fa}_i \cong \fsl_i(2) $ when $ \sigma_i \not= 0 $ and $ \fa_i \cong \C^{\oplus 3} $~-- the 3-dimensional Abelian Lie algebra~-- if $ \sigma_i = 0 $ (see also Section~\ref{subsect: LieSGroups G_s}).

 Recalling the construction of $ \bG_\bsigma $ in Section~\ref{subsect: LieSGroups G_s}, for each $ i \in \{1,2,3\} $ we set $ \widehat{A}_i := A_i $ (isomorphic to either $ \text{SL}_2 $ or $ \C \times \C^\times \times \C $ depending on $ \sigma_i \not=0 $ or $ \sigma_i = 0 $) and $ \widehat{G} := \times_{i=1}^3 \widehat{A}_i = G $, a complex Lie group such that $ \operatorname{Lie}\big( \widehat{G} \big) = {\widehat{\fg}(\bsigma)}_\zero $. Like in Section~\ref{subsect: LieSGroups G_s}, the adjoint action of $ {\widehat{\fg}(\bsigma)}_\zero $ onto $ \widehat{\fg}(\bsigma) $ integrates to a Lie group action of $ \widehat{G} $ onto $ \widehat{\fg}(\bsigma) $: endowed with this action, the pair $ \widehat{\cP}_{\bsigma} := \big( \widehat{G}, \widehat{\fg}(\bsigma) \big) $ is a super Harish-Chandra pair (cf.\ Section~\ref{sHCp's}). Eventually, we can define
\begin{gather*} \widehat{\bG}_\bsigma := \bG_{{}_{\widehat{\cP}_{{}_{\bsigma}}}}\end{gather*}
 to be the complex Lie supergroup associated with $ \widehat{\cP}_{\bsigma} $ following Section~\ref{sHCp's->Liesgrp's}.

\begin{free text} \label{pres-Ghat_s} {\bf A presentation of $\boldsymbol{\widehat{\bG}_\bsigma}$.} To describe the supergroups $ \widehat{\bG}_\bsigma $, we provide hereafter an explicit presentation by generators and relations of all the abstract groups $ \widehat{\bG}_\bsigma(A) $, with $ A \in \Wsalg $. To begin with, let $ \exp\colon \widehat{\fg}_\zero \cong \operatorname{Lie} (\widehat{G} ) \relbar\joinrel\relbar\joinrel\relbar\joinrel\longrightarrow \widehat{G} $ be the exponential map. Like we did in Section~\ref{pres-G_s} for the supergroup $ \bG_\bsigma $, inside each subgroup $ A_i $ we consider
\begin{gather*} x_{2\varepsilon_i}(c) := \exp (c X_{2\varepsilon_i} ), \qquad
 h_{2\varepsilon_i}(c) := \exp (c H_{2\varepsilon_i} ), \qquad
 x_{-2\varepsilon_i}(c) := \exp (c X_{-2\varepsilon_i} ) \end{gather*}
for every $ c \in \C $; then $ \widehat{\varGamma}_i := \{ x_{2\varepsilon_i}(c), h_{2\varepsilon_i}(c), x_{-2\varepsilon_i}(c) \}_{ c \in \C} $ is a generating set for $ \widehat{A}_i $; also, we consider elements $ h_\theta(c) := \exp(c H_\theta) $ for all $ c \in \C $. It follows that the complex Lie group $ \widehat{G} = \widehat{A}_1 \times \widehat{A}_2 \times \widehat{A}_3 $ is generated by
\begin{gather*} \widehat{\varGamma}_\zero := \{ x_{2\varepsilon_i}(c), h_{2\varepsilon_i}(c), h_\theta(c), x_{-2\varepsilon_i}(c) \}^{i \in \{1,2,3\}}_{c \in \C}\end{gather*}
(we could drop the $ h_\theta(c) $'s, but we prefer to keep them among the generators).

 When we think of $ \widehat{G} $ as a (totally even) supergroup, looking at it as a group-valued functor $ \widehat{G}\colon \Wsalg \longrightarrow \text{\rm ($ \cat{grps} $)} $, {\it the abstract group $ \widehat{G}(A) $ of its $ A $-points~-- for $ A \in \Wsalg $~-- is generated by the set}
\begin{gather} \label{eq: gen-set Ghat_+}
 \widehat{\varGamma}_\zero(A) := \{ x_{2\varepsilon_i}(a), h_{2\varepsilon_i}(a), h_\theta(a), x_{-2\varepsilon_i}(a) \}^{i \in \{1,2,3\}}_{a \in A_\zero}.
\end{gather}

In fact, we would better stress that, by construction (cf.\ Section~\ref{subsect: LieSGroups G_s}), we have an obvious isomorphism $ \widehat{G} \cong G $ (see Section~\ref{pres-G_s} for the definition of $ G $) as complex Lie groups.

 To generate the group $ \widehat{\bG}_\bsigma(A) := \bG_{{}_{\widehat{\cP}_{{}_{\bsigma}}}}(A) $ applying the recipe in Section~\ref{sHCp's->Liesgrp's}, we fix in $ {\widehat{\fg}(\bsigma)}_\one $ the $ \C $-basis $\{ Y_i \}_{i \in I} = \big\{ \widehat{X}_\beta := \tau X_\beta \,|\, \beta \in \Delta_\one = \{ \pm \theta, \pm \beta_1, \pm \beta_2, \pm \beta_3 \} \big\} $. Thus, besides the generating elements from $ \widehat{G}(A) $, we take as generators also those of the set
\begin{gather*} \widehat{\varGamma}_\one(A) := \big\{ \widehat{x}_{\pm\theta}(\eta) := \big( 1 + \eta \widehat{X}_{\pm \theta} \big)
, \widehat{x}_{\pm\beta_i}(\eta) := \big( 1 + \eta \widehat{X}_{\pm \beta_i} \big) \big\}_{\eta \in A_\one}^{i \in \{1,2,3\}}.
 \end{gather*}
 Taking into account that $ \widehat{G}(A) $ is generated by $ \widehat{\varGamma}_\zero(A) $, we can modify the set of relations in Section~\ref{sHCp's->Liesgrp's} by letting $ g_+ \in \widehat{G}(A) $ range inside the set $ \widehat{\varGamma}_\zero(A) $: then we can find the following {\it full set of relations} (freely using notation $ e^Z := \exp(Z) $):
\begin{gather*}
 1_{{}_{\widehat{G}}} = 1, \qquad g' \cdot g'' = g' \cdot_{{}_{\widehat{G}}} g'', \qquad \forall\, g', g'' \in \widehat{G}(A),\\
 h_{2\vep_i}(a) \widehat{x}_{\pm\beta_j}(\eta) {h_{2\vep_i}(a)}^{-1} = \widehat{x}_{\pm\beta_j}\big( e^{\pm {(-1)}^{-\delta_{i,j}} \sigma_i a} \eta\big),\\
 h_{2\vep_i}(a) \widehat{x}_{\pm\theta}(\eta) {h_{2\vep_i}(a)}^{-1} = \widehat{x}_{\pm\theta}\big( e^{\pm \sigma_i a} \eta\big),\\
 h_\theta(a) \widehat{x}_{\pm\beta_i}(\eta) {h_\theta(a)}^{-1} = \widehat{x}_{\pm\beta_i}\big( e^{\mp \sigma_i a} \eta\big), \qquad h_\theta(a) \widehat{x}_{\pm\theta}(\eta) {h_\theta(a)}^{-1} = \widehat{x}_{\pm\theta}(\eta),\\
 x_{2\vep_i}(a) \widehat{x}_{\beta_j}(\eta) {x_{2\vep_i}(a)}^{-1} = \widehat{x}_{\beta_j}(\eta) \widehat{x}_{\theta}( \delta_{i,j} \sigma_i a \eta ),\\
 x_{2\vep_i}(a) \widehat{x}_{-\beta_j}(\eta) {x_{2\vep_i}(a)}^{-1} = \widehat{x}_{-\beta_j}(\eta) \widehat{x}_{\beta_k}( (1-\delta_{i,j}) \sigma_i a \eta ),\\
 x_{-2\vep_i}(a) \widehat{x}_{\beta_j}(\eta) {x_{-2\vep_i}(a)}^{-1} = \widehat{x}_{\beta_j}(\eta) \widehat{x}_{-\beta_k}( (1-\delta_{i,j}) \sigma_i a \eta ),\\
 x_{-2\vep_i}(a) \widehat{x}_{-\beta_j}(\eta) {x_{-2\vep_i}(a)}^{-1} = \widehat{x}_{-\beta_j}(\eta) \widehat{x}_{-\theta}( \delta_{i,j} \sigma_i a \eta ),\\
 x_{2\vep_i}(a) \widehat{x}_{\theta}(\eta) {x_{2\vep_i}(a)}^{-1} = \widehat{x}_{\theta}(\eta),\\
 x_{2\vep_i}(a) \widehat{x}_{-\theta}(\eta) {x_{2\vep_i}(a)}^{-1} = \widehat{x}_{-\theta}(\eta) \widehat{x}_{-\beta_i}( \sigma_i a \eta ),\\
 x_{-2\vep_i}(a) \widehat{x}_{\theta}(\eta) {x_{-2\vep_i}(a)}^{-1} = \widehat{x}_{\theta}(\eta) \widehat{x}_{\beta_i}( \sigma_i a \eta ),\\
 x_{-2\vep_i}(a) \widehat{x}_{-\theta}(\eta) {x_{-2\vep_i}(a)}^{-1} = \widehat{x}_{-\theta}(\eta),\\
 \widehat{x}_{\beta_i}(\eta_i) \widehat{x}_{\beta_j}(\eta'_j) = x_{2\vep_k}\big( (1-\delta_{i,j}) \tau_\bsigma^2 \eta'_j \eta_i \big) \widehat{x}_{\beta_j}(\eta'_j) \widehat{x}_{\beta_i}(\eta_i),\\
 \widehat{x}_{-\beta_i}(\eta_i) \widehat{x}_{-\beta_j}(\eta'_j) = x_{-2\vep_k}\big( {-}(1-\delta_{i,j}) \tau_\bsigma^2 \eta'_j \eta_i \big) \widehat{x}_{-\beta_j}(\eta'_j) \widehat{x}_{-\beta_i}(\eta_i),\\
 \widehat{x}_{\beta_i}(\eta_i) \widehat{x}_{-\beta_j}(\eta'_j) =
 h_{2\vep_i}\big( \delta_{i,j} \tau_\bsigma^2 \eta'_j \eta_i \big) h_{\theta}\big( {-} \delta_{i,j} \tau_\bsigma^2 \eta'_j \eta_i \big) \widehat{x}_{-\beta_j}(\eta'_j) \widehat{x}_{\beta_i}(\eta_i),\\
 \widehat{x}_{\beta_i}(\eta_i) \widehat{x}_{\theta}(\eta) = \widehat{x}_{\theta}(\eta) \widehat{x}_{\beta_i}(\eta_i), \qquad \widehat{x}_{\beta_i}(\eta_i) \widehat{x}_{-\theta}(\eta) = x_{-2\vep_i}\big( \tau_\bsigma^2 \eta \eta_i \big) \widehat{x}_{-\theta}(\eta) \widehat{x}_{\beta_i}(\eta_i),\\
 \widehat{x}_{-\beta_i}(\eta_i) \widehat{x}_{\theta}(\eta) = x_{2\vep_i}\big( {-} \tau_\bsigma^2 \eta \eta_i \big) \widehat{x}_{\theta}(\eta) \widehat{x}_{-\beta_i}(\eta_i), \qquad \widehat{x}_{-\beta_i}(\eta_i) \widehat{x}_{-\theta}(\eta) = \widehat{x}_{-\theta}(\eta) \widehat{x}_{-\beta_i}(\eta_i),\\
 \widehat{x}_{\theta}(\eta_+) \widehat{x}_{-\theta}(\eta_-) = h_\theta\big( \tau_\bsigma^2 \eta_- \eta_+ \big) \widehat{x}_{-\theta}(\eta_-) \widehat{x}_{\theta}(\eta_+),\\
 \widehat{x}_{\pm\beta_i}(\eta') \widehat{x}_{\pm\beta_i}(\eta'') = \widehat{x}_{\pm\beta_i}(\eta'+\eta''), \qquad \widehat{x}_{\pm\theta}(\eta') \widehat{x}_{\pm\theta}(\eta'') = \widehat{x}_{\pm\theta}(\eta'+\eta'')
\end{gather*}
with $ \{i,j,k\} \in \{1,2,3\} $.
\end{free text}

\begin{free text} \label{sing-special-Ghat_s} {\bf Singular specializations of the supergroup(s) $\boldsymbol{\widehat{\bG}_\bsigma}$.} The very construction of the supergroups $ \widehat{\bG}_\bsigma $ implies that
\begin{gather*}
\text{\it $ \widehat{\bG}_\bsigma $ is simple $($as a Lie supergroup$)$ for all $ \bsigma = (\sigma_1,\sigma_2,\sigma_3) \in V^\times $.}
\end{gather*}
 This also follows from the presentation of $ \widehat{\bG}_\bsigma $ in Section~\ref{pres-Ghat_s} above, or as a direct consequence of the relation $ \operatorname{Lie}\big(\widehat{\bG}_\bsigma\big) = \widehat{\fg}(\bsigma) = \widehat{\fg}_\bsigma $ and of the fact that $ \widehat{\fg}_\bsigma \cong \fg_\bsigma $ when $ \sigma_i \not= 0 $ for all $ i $.
\end{free text}

Things change instead at ``singular values'' of $ \bsigma $. The complete result is the following:

\begin{thm} \label{thm_Ghat-spec} Given $ \bsigma \in V $, keep notation as above.
\begin{enumerate}\itemsep=0pt
\item[$(1)$] If $ \bsigma \in V^\times $, then the Lie supergroup $ \widehat{\bG}_\bsigma $ is simple.
\item[$(2)$] If $ \bsigma \in V \setminus V^\times $, with $ \sigma_i = 0 $ and $ \sigma_j \not= 0 \not= \sigma_k $ for $ \{i,j,k\} = \{1,2,3\} $, then
 $ \widehat{A}_i \trianglelefteq \widehat{\bG}_\bsigma $ is a {\it central} Lie subsupergroup of $ \widehat{\bG}_\bsigma $, with $ \widehat{A}_i \cong \C \times \C^\times \times \C $, while $ \widehat{A}_j \cong \widehat{A}_k \cong \text{\rm SL}(2) $ for $ \{j,k\} = \{1,2,3\} \setminus \{i\} $. Also, we have a semidirect product splitting
\begin{gather*} \widehat{\bG}(\bsigma) \cong {\widehat{\bG}(\bsigma)}_{\rm rd} \ltimes {\widehat{\bG}(\bsigma)}_\one\end{gather*}
 with $ {\widehat{\bG}(\bsigma)}_{\rm rd} = \mathop{\times}\limits_{\ell=1}^{ 3} \widehat{A}_\ell \cong ( \C \times \C^\times \times \C ) \times \text{\rm SL}(2) \times \text{\rm SL}(2) $ while $ {\widehat{\bG}(\bsigma)}_\one $ is the subsupergroup of~$ \widehat{\bG}_\bsigma $ generated by the $ \widehat{x}_{\pm \theta} $'s and the $ \widehat{x}_{\pm \beta_i} $'s $($for all $ i )$, which is normal Abelian, isomorphic to $ \mathbb{A}_{\scriptscriptstyle C}^{0|8} $~-- the $($totally odd$)$ complex affine Abelian supergroup of superdimension $ (0|8) $~-- and such that
 $ {\widehat{\bG}(\bsigma)}_\one \cong ( {\raise1,3pt\hbox{$ \sqbullet $}} {\raise1pt\hbox{$ \hskip2pt \oplus \hskip1pt $}} {\raise1,3pt\hbox{$ \sqbullet $}} ) \boxtimes \Box \boxtimes \Box $~-- where $ {\raise1,3pt\hbox{$ \sqbullet $}} $ is the trivial representation~-- as a module over $ {\widehat{\bG}(\bsigma)}_{\rm rd} \cong \big( \C \times \C^\times \times \C \big) \times \text{\rm SL}(2) \times \text{\rm SL}(2) $. In other words, there exists a {\it split} short exact sequence of Lie supergroups
\begin{gather*}\boldsymbol{1} \relbar\joinrel\longrightarrow ( {\raise1,3pt\hbox{$ \sqbullet $}} {\raise1pt\hbox{$ \hskip2pt \oplus \hskip1pt $}} {\raise1,3pt\hbox{$ \sqbullet $}} ) \boxtimes \Box \boxtimes \Box \cong {\widehat{\bG}(\bsigma)}_\one \relbar\joinrel\longrightarrow \widehat{\bG}(\bsigma) {\buildrel
{{\displaystyle \dashleftarrow \hskip-4pt \text{$-$}}} \over {\relbar\joinrel\longrightarrow}} {\widehat{\bG}(\bsigma)}_{\rm rd} \\
\qquad{} \cong \big( \C \times \C^\times \times \C \big) \times {\text{\rm SL}(2)}^2 \relbar\joinrel\longrightarrow \boldsymbol{1}.
\end{gather*}
\item[$(3)$] If $ \bsigma = \boldsymbol{0}$ $(\!\in V \setminus V^\times ) $, i.e., $ \sigma_h = 0 $ for all $ h \in \{1,2,3\} $, then $ \widehat{\bG}_\bsigma $ is the Abelian Lie supergroup
 $ \widehat{\bG}_\bsigma \cong ( \C \times \C^\times \times \C )^3 \times ( {\raise1,2pt\hbox{$ \sqbullet $}} {\raise1pt\hbox{$ \hskip2pt \oplus \hskip1pt $}} {\raise1,2pt\hbox{$ \sqbullet $}} )^{\boxtimes 3} $.
 \end{enumerate}
\end{thm}

\begin{proof} As for the parallel results for $ \bG_\bsigma $, $ \bG'_\bsigma $ and $ \bG''_\bsigma $, we can deduce the claim from the presentation of $ \widehat{\bG}_\bsigma $ in Section~\ref{pres-Ghat_s}, or from the link $ \operatorname{Lie} ( \widehat{\bG}_\bsigma) = \widehat{\fg}(\bsigma) $ along with Theo\-rem~\ref{thm_g_eta-spec}.
\end{proof}

\subsection[Lie supergroups from contractions: the family of the $ \widehat{\bG}^{ \prime}_\bsigma $'s]{Lie supergroups from contractions: the family of the $\boldsymbol{\widehat{\bG}^{ \prime}_\bsigma}$'s} \label{subsect: LieSGroups Ghat'_s}

 Given $ \bsigma = (\sigma_1,\sigma_2,\sigma_3) \in V $, we follow again Section~\ref{contractions} and set $ \tau := x_{1 } x_{2 } x_{3 } \in \C[\bx] $ and $ I = I_\bsigma := ( \{ x_i - \sigma_i \}_{i=1,2,3}) $; but now we consider the corresponding complex Lie algebra $ \widehat{\fg}^{ \prime}(\bsigma) $, with $ {\widehat{\fg}^{ \prime}(\bsigma)}_\zero $ and $ {\widehat{\fg}^{ \prime}(\bsigma)}_\one $ as its even and odd part, respectively (and we still make use of some abuse of notation as in Section~\ref{subsect: LieSGroups hatG_s}). By construction, a $ \C $-basis of $ \widehat{\fg}^{ \prime}(\bsigma) $ is
\begin{gather*} \widehat{B}^{ \prime} := \{ X'_\alpha, H'_{2\varepsilon_i} \,|\, \alpha \in \Delta, i \in \{1,2,3\} \} \cup \big\{ \widehat{X}^{ \prime}_\beta := \tau X'_\beta \,|\, \beta \in \Delta_\one \big\}.
\end{gather*}
Consider also $ \widehat{\fa}^{ \prime}_i := \fa^{ \prime}_i $ ($ := \C X'_{2\vep_i} \oplus \C H'_{2\vep_i} \oplus \C X'_{-2\vep_i} $, cf.\ Section~\ref{subsect-g'(sigma)}) for all $ 1 = 1, 2, 3 $: all these are Lie subalgebras in $ \widehat{\fg}^{ \prime}(\bsigma) $, isomorphic to $ \fsl(2) $, and $ {\widehat{\fg}^{ \prime}(\bsigma)}_\zero = \oplus_{i=1}^3 \widehat{\fa}^{ \prime}_i $.
 The {\it faithful} adjoint action of $ {\widehat{\fg}^{ \prime}(\bsigma)}_\zero $ onto $ \widehat{\fg}^{ \prime}(\bsigma) $ gives a Lie algebra monomorphism $ {\widehat{\fg}^{ \prime}(\bsigma)}_\zero \lhook\joinrel\longrightarrow \fgl ( \widehat{\fg}^{ \prime}(\bsigma) ) $, by which we identify $ {\widehat{\fg}^{ \prime}(\bsigma)}_\zero $ with its image in $ \fgl ( \widehat{\fg}^{ \prime}(\bsigma) ) $. Then $ \exp \colon \fgl\big( \widehat{\fg}^{ \prime}(\bsigma) \big) \longrightarrow \text{GL} ( \widehat{\fg}^{ \prime}(\bsigma) ) $ yields a~Lie subgroup $ \widehat{G}^{ \prime} := \exp ( {\widehat{\fg}^{ \prime}(\bsigma)}_\zero ) $ in $ \text{GL} ( \widehat{\fg}^{ \prime}(\bsigma) ) $ which faithfully acts onto $ \widehat{\fg}^{ \prime}(\bsigma) $ and is such that $ \operatorname{Lie} (\widehat{G}^{ \prime} ) = ( \widehat{\fg}^{ \prime}(\bsigma) )_\zero $; finally, the pair $ \widehat{\cP}^{ \prime}_{\bsigma} := ( \widehat{G}^{ \prime}, \widehat{\fg}^{ \prime}(\bsigma) ) $ with this action is then a super Harish-Chandra pair (cf.\ Section~\ref{sHCp's}).

 As alternative method, we might also construct the super Harish-Chandra pair $ \widehat{\cP}^{ \prime}_{\bsigma} $ via the same procedure, up to the obvious, minimal changes, adopted for $ \cP^{ \prime}_{\bsigma} $ in Section~\ref{subsect: LieSGroups G'_s}; indeed, one can also do the converse, namely use the present method to construct $ \cP^{ \prime}_{\bsigma} $ as well.

 Once we have the super Harish-Chandra pair $ \widehat{\cP}^{ \prime}_{\bsigma} $, it makes sense to define
\begin{gather*}\widehat{\bG}^{ \prime}_\bsigma := \bG_{{}_{\widehat{\cP}^{ \prime}_{{}_{\bsigma}}}} \end{gather*}
 that is the complex Lie supergroup associated with $ \widehat{\cP}^{ \prime}_{\bsigma} $ after the recipe in Section~\ref{sHCp's->Liesgrp's}.

\begin{free text} \label{pres-Ghat'_s}
{\bf A presentation of $\boldsymbol{\widehat{\bG}^{ \prime}_\bsigma}$.} We shall presently describe the supergroups $ \widehat{\bG}^{ \prime}_\bsigma $ by means of an explicit presentation by generators and relations of all the abstract groups $ \widehat{\bG}^{ \prime}_\bsigma(A) $, for all $ A \in \Wsalg $. To begin with, let $ \exp \colon \widehat{\fg}^{ \prime}_\zero \cong \operatorname{Lie} (\widehat{G}^{ \prime} ) \relbar\joinrel\relbar\joinrel\longrightarrow \widehat{G}^{ \prime} $ be the exponential map. Just like for the supergroup $ \bG_\bsigma $ in Section~\ref{pres-G_s}, inside each subgroup $ A'_i $ we consider
\begin{gather*} x'_{2\varepsilon_i}(c) := \exp (c X'_{2\varepsilon_i} ), \qquad
 h'_{2\varepsilon_i}(c) := \exp (c H'_{2\varepsilon_i} ), \qquad
 x'_{-2\varepsilon_i}(c) := \exp (c X'_{-2\varepsilon_i} ) \end{gather*}
for every $ c \in \C $; then $ \widehat{\varGamma}^{ \prime}_i := \{ x'_{2\varepsilon_i}(c), h'_{2\varepsilon_i}(c), x'_{-2\varepsilon_i}(c) \}_{ c \in \C} $ is a generating set for $ \widehat{A}^{ \prime}_i = A'_i $; also, we consider elements $ h'_\theta(c) := \exp\big(c H'_\theta\big) $ for all $ c \in \C $. It follows that the complex Lie group $ \widehat{G}^{ \prime} = \widehat{A}^{ \prime}_1 \times \widehat{A}^{ \prime}_2 \times \widehat{A}^{ \prime}_3 $ is generated by
\begin{gather*} \widehat{\varGamma}^{ \prime}_\zero := \{ x'_{2\varepsilon_i}(c), h'_{2\varepsilon_i}(c), h'_\theta(c), x'_{-2\varepsilon_i}(c) \}^{i \in \{1,2,3\}}_{c \in \C}
\end{gather*}
(as before, we could drop the $ h'_\theta(c) $'s, but we prefer to keep them among the generators).

 When thinking of $ \widehat{G}^{ \prime} $ as a (totally even) supergroup, considered as a group-valued functor $ \widehat{G}^{ \prime} \colon \Wsalg \longrightarrow \text{\rm ($ \cat{grps} $)} $, {\it the abstract group $ \widehat{G}^{ \prime}(A) $ of its $ A $-points~-- for $ A \in \Wsalg $~-- is gene\-ra\-ted by the set}
\begin{gather} \label{eq: gen-set Ghat'_+}
 \widehat{\varGamma}^{ \prime}_\zero(A) := \{ x'_{2\varepsilon_i}(a), h'_{2\varepsilon_i}(a), h_\theta(a), x'_{-2\varepsilon_i}(a) \}^{i \in \{1,2,3\}}_{a \in A_\zero}.
\end{gather}

Indeed, we can also stress that, by construction (cf.\ Section~\ref{subsect: LieSGroups G_s}), there exists an obvious isomorphism $ \widehat{G}^{ \prime} \cong G' $ as complex Lie groups.
 Now, to generate the group $ \widehat{\bG}^{ \prime}_\bsigma(A) := \bG_{{}_{\widehat{\cP}^{ \prime}_{{}_{\bsigma}}}}(A) $ following the recipe in Section~\ref{sHCp's->Liesgrp's}, we fix in $( {\widehat{\fg}^{ \prime}(\bsigma)}_{{}_\C} )_\one $ the $ \C $-basis $ \{ Y_i \}_{i \in I} = \big\{ \widehat{X}'_\beta := \tau X'_\beta \,|\, \beta \in \Delta_\one = \{ \pm \theta, \pm \beta_1, \pm \beta_2, \pm \beta_3 \} \big\} $. Then, beside the generating elements from $ \widehat{G}^{ \prime}(A) $ we take as generators also those of the set
\begin{gather*} \widehat{\varGamma}^{ \prime}_\one(A) := \big\{ \widehat{x}^{ \prime}_{\pm\theta}(\eta) :=
 \big( 1 + \eta \widehat{X}'_{\pm \theta} \big), \widehat{x}^{ \prime}_{\pm\beta_i}(\eta) := \big( 1 + \eta \widehat{X}'_{\pm \beta_i} \big) \big\}_{\eta \in A_\one}^{i \in \{1,2,3\}}.
 \end{gather*}
 Knowing that $ \widehat{G}^{ \prime}(A) $ is generated by $ \widehat{\varGamma}^{ \prime}_\zero(A) $, we can modify the set of relations in Section~\ref{sHCp's->Liesgrp's} by letting $ g \in \widehat{G}^{ \prime}(A) $ range inside $ \widehat{\varGamma}^{ \prime}_\zero(A) $; eventually, we can find the following {\it full set of relations} (freely using notation $ e^Z := \exp(Z) $):
\begin{gather*}
1_{{}_{\widehat{G}^{ \prime}}} = 1, \qquad g' \cdot g'' = g' \cdot_{{}_{\widehat{G}^{ \prime}}} g'', \qquad \forall\, g', g'' \in \widehat{G}^{ \prime}(A),\\
 h'_{2\vep_i}(a) \widehat{x}^{ \prime}_{\pm\beta_j}(\eta) {h'_{2\vep_i}(a)}^{-1} = \widehat{x}^{ \prime}_{\pm\beta_j}\big( e^{\pm {(-1)}^{-\delta_{i,j}} a} \eta \big),\\
 h'_{2\vep_i}(a) \widehat{x}^{ \prime}_{\pm\theta}(\eta) {h'_{2\vep_i}(a)}^{-1} = \widehat{x}^{ \prime}_{\pm\theta}\big( e^{\pm a} \eta \big),\\
 h'_\theta(a) \widehat{x}^{ \prime}_{\pm\beta_i}(\eta) {h'_\theta(a)}^{-1} = \widehat{x}^{ \prime}_{\pm\beta_i}\big( e^{\mp \sigma_i a} \eta \big), \qquad h'_\theta(a) \widehat{x}^{ \prime}_{\pm\theta}(\eta) {h'_\theta(a)}^{-1} = \widehat{x}^{ \prime}_{\pm\theta}(\eta),\\
 x'_{2\vep_i}(a) \widehat{x}^{ \prime}_{\beta_j}(\eta) {x'_{2\vep_i}(a)}^{-1} = \widehat{x}^{ \prime}_{\beta_j}(\eta) \widehat{x}^{ \prime}_{\theta}( \delta_{i,j} a \eta ),\\
 x'_{2\vep_i}(a) \widehat{x}^{ \prime}_{-\beta_j}(\eta) {x'_{2\vep_i}(a)}^{-1} = \widehat{x}^{ \prime}_{-\beta_j}(\eta) \widehat{x}^{ \prime}_{\beta_k}( (1-\delta_{i,j}) a \eta ),\\
 x'_{-2\vep_i}(a) \widehat{x}^{ \prime}_{\beta_j}(\eta) {x'_{-2\vep_i}(a)}^{-1} = \widehat{x}^{ \prime}_{\beta_j}(\eta) \widehat{x}^{ \prime}_{-\beta_k}( (1-\delta_{i,j}) a \eta ),\\
 x'_{-2\vep_i}(a) \widehat{x}^{ \prime}_{-\beta_j}(\eta) {x'_{-2\vep_i}(a)}^{-1} = \widehat{x}^{ \prime}_{-\beta_j}(\eta) \widehat{x}^{ \prime}_{-\theta}( \delta_{i,j} a \eta ),\\
 x'_{2\vep_i}(a) \widehat{x}^{ \prime}_{\theta}(\eta) {x'_{2\vep_i}(a)}^{-1} = \widehat{x}^{ \prime}_{\theta}(\eta),\\
 x'_{2\vep_i}(a) \widehat{x}^{ \prime}_{-\theta}(\eta) {x'_{2\vep_i}(a)}^{-1} = \widehat{x}^{ \prime}_{-\theta}(\eta) \widehat{x}^{ \prime}_{-\beta_i}( a \eta ),\\
 x'_{-2\vep_i}(a) \widehat{x}^{ \prime}_{\theta}(\eta) {x'_{-2\vep_i}(a)}^{-1} = \widehat{x}^{ \prime}_{\theta}(\eta) \widehat{x}^{ \prime}_{\beta_i}( a \eta ),\\
 x'_{-2\vep_i}(a) \widehat{x}^{ \prime}_{-\theta}(\eta) {x'_{-2\vep_i}(a)}^{-1} = \widehat{x}^{ \prime}_{-\theta}(\eta),\\
 \widehat{x}^{ \prime}_{\beta_i}(\eta_i) \widehat{x}^{ \prime}_{\beta_j}(\eta'_j) = x'_{2\vep_k}\big((1-\delta_{i,j}) \tau_\bsigma^2 \sigma_i \eta'_j \eta_i\big) \widehat{x}^{ \prime}_{\beta_j}(\eta'_j) \widehat{x}^{ \prime}_{\beta_i}(\eta_i),\\
 \widehat{x}^{ \prime}_{-\beta_i}(\eta_i) \widehat{x}^{ \prime}_{-\beta_j}(\eta'_j) = x'_{-2\vep_k}\big( {-}(1-\delta_{i,j}) \tau_\bsigma^2 \sigma_i \eta'_j \eta_i\big) \widehat{x}^{ \prime}_{-\beta_j}(\eta'_j) \widehat{x}^{ \prime}_{-\beta_i}(\eta_i),\\
 \widehat{x}^{ \prime}_{\beta_i}(\eta_i) \widehat{x}^{ \prime}_{-\beta_j}(\eta'_j) = h'_{2\vep_i}\big( \delta_{i,j} \tau_\bsigma^2 \sigma_i \eta'_j \eta_i \big) h'_{\theta}\big( {-} \delta_{i,j} \tau_\bsigma^2 \eta'_j \eta_i \big) \widehat{x}^{ \prime}_{-\beta_j}(\eta'_j) \widehat{x}^{ \prime}_{\beta_i}(\eta_i),\\
 \widehat{x}^{ \prime}_{\beta_i}(\eta_i) \widehat{x}^{ \prime}_{\theta}(\eta) = \widehat{x}^{ \prime}_{\theta}(\eta) \widehat{x}^{ \prime}_{\beta_i}(\eta_i), \qquad \widehat{x}^{ \prime}_{\beta_i}(\eta_i) \widehat{x}^{ \prime}_{-\theta}(\eta) = x'_{-2\vep_i}\big( \tau_\bsigma^2 \sigma_i \eta \eta_i \big) \widehat{x}^{ \prime}_{-\theta}(\eta) \widehat{x}^{ \prime}_{+\beta_i}(\eta_i),\\
 \widehat{x}^{ \prime}_{-\beta_i}(\eta_i) \widehat{x}^{ \prime}_{\theta}(\eta) = x'_{2\vep_i}\big( {-} \tau_\bsigma^2 \sigma_i \eta \eta_i \big) \widehat{x}^{ \prime}_{\theta}(\eta) \widehat{x}^{ \prime}_{-\beta_i}(\eta_i), \quad \widehat{x}^{ \prime}_{-\beta_i}(\eta_i) \widehat{x}^{ \prime}_{-\theta}(\eta) = \widehat{x}^{ \prime}_{-\theta}(\eta) \widehat{x}^{ \prime}_{-\beta_i}(\eta_i),\\
 \widehat{x}^{ \prime}_{\theta}(\eta_+) \widehat{x}^{ \prime}_{-\theta}(\eta_-) = h'_{\theta}\big( \tau_\bsigma^2 \eta_- \eta_+ \big) \widehat{x}^{ \prime}_{-\theta}(\eta_-) \widehat{x}^{ \prime}_{\theta}(\eta_+),\\
 \widehat{x}^{ \prime}_{\pm\beta_i}(\eta') \widehat{x}^{ \prime}_{\pm\beta_i}(\eta'') = \widehat{x}^{ \prime}_{\pm\beta_i}(\eta'+\eta''), \qquad \widehat{x}^{ \prime}_{\pm\theta}(\eta') \widehat{x}^{ \prime}_{\pm\theta}(\eta'') = \widehat{x}^{ \prime}_{\pm\theta}(\eta'+\eta'')
\end{gather*}
 with $ \{i,j,k\} \in \{1,2,3\} $.
\end{free text}

\begin{free text} \label{sing-special-Ghat'_s} {\bf Singular specializations of the supergroup(s) $\boldsymbol{\widehat{\bG}^{ \prime}_\bsigma}$.} From the very construction of the supergroups $ \widehat{\bG}^{ \prime}_\bsigma $ we get
\begin{gather*}
\text{\it $ \widehat{\bG}^{ \prime}_\bsigma $ is simple $($as a Lie supergroup$)$ for all $ \bsigma = (\sigma_1,\sigma_2,\sigma_3) \in V^\times $.}
\end{gather*}
 This follows from the presentation of $ \widehat{\bG}^{ \prime}_\bsigma $ in Section~\ref{pres-Ghat'_s} above, but also as direct consequence of the relation $ \operatorname{Lie} (\widehat{\bG}^{ \prime}_\bsigma ) = \widehat{\fg}^{ \prime}(\bsigma) = \widehat{\fg}^{ \prime}_\bsigma $ and of $ \widehat{\fg}^{ \prime}_\bsigma \cong \fg_\bsigma $ when $ \sigma_i \not= 0 $ for all $ i $.
\end{free text}

At ``singular values'' of $ \bsigma $ instead things are quite different; the precise claim is as follows:

\begin{thm} \label{thm_Ghat'-spec} Given $ \bsigma \in V $, keep notation as above.
\begin{enumerate}\itemsep=0pt
\item[$(1)$] If $ \bsigma \in V^\times $, then the Lie supergroup $ \widehat{\bG}^{ \prime}_\bsigma $ is simple.
\item[$(2)$] If $ \bsigma \in V \setminus V^\times $, let $( \widehat{\bG}^{ \prime}_\bsigma)_\one $ be the subsupergroup of $ \widehat{\bG}^{ \prime}_\bsigma $ generated by the~$ \widehat{x}^{ \prime}_{\pm \theta} $'s and the~$ \widehat{x}^{ \prime}_{\pm \beta_i} $'s $($for all~$ i )$. Then $( \widehat{\bG}^{ \prime}_\bsigma)_\one $ is Abelian and normal in $ \widehat{\bG}^{ \prime}_\bsigma $, and there exist isomorphisms $(\widehat{\bG}^{ \prime}_\bsigma)_{\rm rd} \cong {\text{\rm SL}(2)}^{\times 3} $ and $( \widehat{\bG}^{ \prime}_\bsigma)_\one \cong \boxtimes_{ i=1}^{ 3} \Box_i $ as a $( \widehat{\bG}^{ \prime}_\bsigma)_{\rm rd} $-module; in particular, $( \widehat{\bG}^{ \prime}_\bsigma)_\one $ is Abelian; moreover, there exists an isomorphism
\begin{gather*} \widehat{\bG}^{ \prime}_\bsigma \cong \big( \widehat{\bG}^{ \prime}_\bsigma \big)_{\rm rd} \ltimes \big( \widehat{\bG}^{ \prime}_\bsigma \big)_\one \cong {\text{\rm SL}(2)}^{\times 3} \ltimes \Box^{ \boxtimes 3}.
\end{gather*}
In other words, there exists a {\it split} short exact sequence
\begin{gather*} \boldsymbol{1} \relbar\joinrel\relbar\joinrel\relbar\joinrel\relbar\joinrel\longrightarrow \Box^{ \boxtimes 3} \cong \big( \widehat{\bG}^{ \prime}_\bsigma \big)_\one \relbar\joinrel\relbar\joinrel\relbar\joinrel\relbar\joinrel\longrightarrow \widehat{\bG}^{ \prime}_\bsigma {\buildrel
{{\displaystyle \dashleftarrow \hskip-4pt \text{$-$} \text{$-$} \text{$-$}}} \over {\relbar\joinrel\relbar\joinrel\relbar\joinrel\relbar\joinrel\longrightarrow}} \big( \widehat{\bG}^{ \prime}_\bsigma \big)_{\rm rd} \cong {\text{\rm SL}(2)}^{\times 3} \relbar\joinrel\relbar\joinrel\relbar\joinrel\relbar\joinrel\longrightarrow \boldsymbol{1}.
\end{gather*}
\end{enumerate}
\end{thm}

\begin{proof} Much like the similar result for $ \widehat{\bG}_\bsigma $, we can deduce the present claim from the presentation of $ \widehat{\bG}^{ \prime}_\bsigma $ in Section~\ref{pres-Ghat'_s}, or from the relation $ \operatorname{Lie}( \widehat{\bG}^{ \prime}_\bsigma) = \widehat{\fg}^{ \prime}(\bsigma) $ along with Theo\-rem~\ref{thm_g'_eta-spec}.
\end{proof}

\begin{free text} \label{int-case_G_s-alg-sgrp} {\bf The integral case: $\boldsymbol{\bG_\bsigma}$, $\boldsymbol{\bG'_\bsigma}$, $\boldsymbol{\bG''_\bsigma}$, $\boldsymbol{\widehat{\bG}_\bsigma}$ and $\boldsymbol{\widehat{\bG}^{ \prime}_\bsigma}$ as algebraic supergroups.} In the integral case, i.e., when $ \bsigma \in \Z^3 $, the Lie supergroups we have introduced above are, in fact, complex {\it algebraic} supergroups: indeed, this follows as a consequence of an alternative presentation of them, that makes sense if and only if~$ \bsigma \in \Z^3 $.

 Let us look at $ \bG_\bsigma $, for some fixed $ \bsigma \in \Z^3 $. Consider the generating set \eqref{eq: gen-set G_+} for the groups $ \varGamma_\zero(A) $, and for each $ \alpha \in \{ 2 \varepsilon_1, 2 \varepsilon_2, 2 \varepsilon_3, \theta \} $ replace the generators $ h_\alpha(a) := \exp ( a H_\alpha) $~-- for all $ a \in A_\zero $~-- therein with $ \widetilde{h}_\alpha(u) $~-- for all $ u \in U(A_\zero) $, the group of units of $ A_\zero $. Every such $ \widetilde{h}_\alpha(u) $ is the toral element in $ G(A) $ whose adjoint action on $ \fg_\bsigma $ is given by $ \operatorname{Ad} ( \widetilde{h}_\alpha(u) )(X_\gamma) = u^{\gamma(H_\alpha)} X_\gamma $ for all $ \gamma \in \Delta $; note that this makes sense, since we have $ \gamma(H_\alpha) \in \Z $ just because $ \bsigma \in \Z^3 $. Now, the set
\begin{gather*}
 \widetilde{\varGamma}_\zero(A) := \big\{ x_{2\varepsilon_i}(a), \widetilde{h}_{2\varepsilon_i}(u), \widetilde{h}_\theta(u), x_{-2\varepsilon_i}(a) \,|\, i \in \{1,2,3\}, a \in A_\zero, u \in U(A_\zero) \big\}
 \end{gather*}
 still generates $ G(A) $. A moment's thought then shows that $ \bG_\bsigma(A) $ can be realized as the group generated by $ \widetilde{\varGamma}(A) := G(A) \cup \varGamma_\one(A) $ with the same relations as in Section~\ref{pres-G_s} up to the following changes: all relations that involve no generators of type $ h_\alpha(a) $ are kept the same, while the others are replaced by the following ones (with $ \{i,j,k\} \in \{1,2,3\} $):
\begin{gather*}
 \widetilde{h}_{2\vep_i}(u) x_{\pm\beta_j}(\eta) {\widetilde{h}_{2\vep_i}(u)}^{-1} = x_{\pm\beta_j}\big( u^{\pm {(-1)}^{-\delta_{i,j}} \sigma_i} \eta\big),\\
 \widetilde{h}_{2\vep_i}(u) x_{\pm\theta}(\eta) {\widetilde{h}_{2\vep_i}(u)}^{-1} = x_{\pm\theta}\big( u^{\pm \sigma_i} \eta\big),\\
 \widetilde{h}_\theta(u) x_{\pm\beta_i}(\eta) {\widetilde{h}_\theta(u)}^{-1} = x_{\pm\beta_i}\big( u^{\mp \sigma_i} \eta\big), \qquad \widetilde{h}_\theta(u) x_{\pm\theta}(\eta) {\widetilde{h}_\theta(u)}^{-1} = x_{\pm\theta}(\eta),\\
 x_{\theta}(\eta_+) x_{-\theta}(\eta_-) = \widetilde{h}_\theta( \eta_- \eta_+ ) x_{-\theta}(\eta_-) x_{\theta}(\eta_+).
 \end{gather*}
In fact, the key point here is that if (and only if) $ \bsigma \in \Z^3 $, then all our construction does make sense in the framework of algebraic supergeometry, namely $ \cP_{\bsigma} := ( G, \fgs) $ is a super Harish-Chandra pair in the algebraic sense~-- like in \cite{Ga2}~-- and $ \bG_\bsigma := \bG_{{}_{\cP_{{}_{\bsigma}}}} $ is nothing but the corresponding algebraic supergroup associated with $ \cP_{\bsigma} $ trough the algebraic version of category equivalence in Section~\ref{sHCp's->Liesgrp's}~-- cf.~\cite{Ga2} again. If we present the groups $ G(A) $ using $ \widetilde{\varGamma}_\zero(A) $ as generating set, we can also extend such a description~-- as $ \bsigma \in \Z^3 $~-- to a presentation of the groups $ \bG_\bsigma(A) $ as above. Leaving details to the reader, the same analysis applies when we look at $ \bG'_\bsigma $, $ \bG''_\bsigma $, $ \widehat{\bG}_\bsigma $ or~$ \widehat{\bG}^{ \prime}_\bsigma $ instead of~$ \bG_\bsigma $: whenever $ \bsigma \in \Z^3 $, all of them are in fact complex {\it algebraic} supergroups.
\end{free text}

\begin{free text} \label{subsubsect: compar-fam.s - grps} {\bf A geometrical interpretation.} In the previous discussion we considered five families of Lie supergroups indexed by the points of $ V $, namely $\{ \bG_\bsigma\}_{\sigma \in V} $, $\{ \bG'_\bsigma \}_{\sigma \in V} $, $\{ \bG''_\bsigma\}_{\sigma \in V} $, $\{ \widehat{\bG}_\bsigma\}_{\sigma \in V} $ and $\{ \widehat{\bG}^{ \prime}_\bsigma\}_{\bsigma \in V} $. Our analysis shows that these families have in common all the elements indexed by ``general'' points, i.e., elements $ \bsigma \in V^\times := V \setminus \cS $ with $ \cS := \bigcup\limits_{i=1}^3 \{ \sigma_i = 0 \} $. On the other hand, these families are entirely different at all points in the ``singular locus''~$\cS$.

In geometrical terms, each family forms a fibre space, say $ \mathbb{L}_{ \bG_\Cbx} $, $ \mathbb{L}_{ \bG'_\Cbx} $, $ \mathbb{L}_{ \bG''_\Cbx} $, $ \mathbb{L}_{ \widehat{\bG}_\Cbx} $, and~$ \mathbb{L}_{ \widehat{\bG}'_\Cbx} $, respectively, over the base space $ \operatorname{Spec} (\Cbx) \cong V \cup \{\star\}$ $ \big( {\cong} \mathbb{A}_\C^{ 2} \cup \{\star\} \big) $, whose fibres are Lie supergroups. Our results show that the fibres in any two of these fibre spaces do coincide at general points~-- where they are {\it simple} Lie supergroups~-- and do differ instead at singular points~-- where they are non-simple indeed.

As an outcome, loosely speaking we can say that {\it our construction provides {\it five different} ``completions'' of the family $\{\bG_\bsigma\}_{\sigma \in V \setminus \cS} $ of simple Lie supergroups, by adding~-- in five different ways~-- some new {\it non-simple} extra elements}. Note also that, a priori, many other such ``completions'', more or less similar, could be devised: we just presented these ones as significant, interesting examples, with no claims whatsoever of being exhaustive.
\end{free text}

\subsection*{Acknowledgements}
The first author is partially supported by the French {\it Agence Nationale de la Recherche $($ANR GeoLie project ANR-15-CE40-0012$)$}. The second author acknowledges the MIUR {\it Excellence Department Project} awarded to the Department of Mathematics, University of Rome ``Tor Vergata'', CUP E83C18000100006. The authors also would like to thank the anonymous referees for their useful comments and suggestions to improve the presentation of this article.

%\pdfbookmark[1]{References}{ref}
\LastPageEnding

\end{document}